\documentclass[11pt,letterpaper]{amsart}
\usepackage[margin=2.5cm]{geometry}

\usepackage[french,english]{babel}

\usepackage[T1]{fontenc}
\usepackage{lmodern, amsfonts,amsmath,amstext,amsbsy,amssymb,
amsopn,amsthm,upref,eucal,mathptmx,mathtools,url,thmtools}

\usepackage{dutchcal}
\usepackage{mathrsfs}
\usepackage{tcolorbox}
\usepackage[hidelinks]{hyperref}
\usepackage{float}

\usepackage{amsthm,cleveref}

\RequirePackage{xcolor} 
\definecolor{halfgray}
{gray}{0.55}
\definecolor{webgreen}
{rgb}{0,0.4,0}
\definecolor{webbrown}
{rgb}{.8,0.1,0.1}
\definecolor{red}
{rgb}{1,0,0}
\usepackage{microtype}
\usepackage{tikz}
\usepackage{soul} 

\newcommand \cD {{\mathcal D}}
\newcommand \cF {{\mathcal F}}
\newcommand \cH {{\mathcal H}}
\newcommand \cI {{\mathcal I}}
\newcommand \cL {{\mathcal L}}
\newcommand \cR {{\mathcal R}}

\newcommand \R {{ \mathbb R}}
\def\C{{\mathbb C}}

\newcommand \Z {{ \mathbb Z}}
\newcommand \N {{ \mathbb N}}
\newcommand \T {{ \mathbb T}}
\newcommand \tS {{\widetilde S}}
\newcommand \tM {{\widetilde M}}

\newcommand \cB {{\mathscr B}}

\newcommand \bT {{\mathbb T}}
\newcommand \cC {{\mathscr{C}}}
\newcommand \cO {{\mathcal{O}}}

\newcommand \sC {{C_\sharp}}

\newcommand{\horo}{\mathsf{h}}
\newcommand{\geo}{\mathsf{g}}

\newcommand*{\diff}{\mathop{}\!\mathrm{d}}
\newcommand{\one}{{\rm 1\mskip-4mu l}}

\newcommand{\xb}{\mathbf{x}}
\newcommand{\yb}{\mathbf{y}}

\newcommand{\eb}{\mathbf{e}}
\newcommand{\eps}{\epsilon}

\newcommand \re {{%
\operatorname{Re}
}}

\newcommand{\Aut}{%
\operatorname{Aut}
}

\DeclareMathOperator{\vol}{vol}

\DeclareMathOperator{\topp}{top}
\DeclareMathOperator{\Deck}{Deck}
\DeclareMathOperator{\abel}{ab}
\DeclareMathOperator{\Lip}{Lip}

\DeclareMathOperator{\supp}{supp}

\newtheorem{theorem}{Theorem}[section]
\newtheorem {lemma}[theorem]{Lemma}
\newtheorem {proposition}[theorem]{Proposition}
\newtheorem{corollary}[theorem]{Corollary}

\newtheorem{bigthm}{Theorem}

\theoremstyle{definition}
\newtheorem{remark}[theorem]{Remark}

\date{\today}

\author{Roberto Castorrini}

\address{
Scuola Normale Superiore, Piazza dei Cavalieri 7, Pisa, Italy, 56126
}

\email{roberto.castorrini@gmail.com}

\author{Davide Ravotti}

\address{
Universit\'e de Lille, CNRS, UMR 8524 - Laboratoire Paul Painlev\'e, F-59000 Lille, France
}

\email{davide.ravotti@gmail.com}

 \title[Horocycle flows on abelian covers]
 {Horocycle flows on abelian covers of surfaces of negative curvature \\[0.5em]
 Flot horocyclique sur les rev\^etements ab\'eliens de surfaces \`a courbure n\'egative}

\begin{document}


\maketitle

\begin{abstract}
	We consider the unit speed parametrization of the horocycle flow on infinite Abelian covers of compact surfaces of negative curvature. We prove an asymptotic result for the ergodic integrals of sufficiently regular functions. 
	Our approach, which does not rely on symbolic dynamics, is based on a general Fourier decomposition for Abelian covers. This enables us to extend the transfer operator techniques to the case of geodesic flow in an infinite volume setting.
		In the case of constant curvature, where the unit speed and the uniformly contracting parametrizations of horocycles coincide, we recover a result by Ledrappier and Sarig.
	Finally, as a byproduct result, we obtain a power deviation estimate for the horocycle ergodic averages on compact surfaces, without requiring any pinching condition as in previous results.
\end{abstract}

\begin{otherlanguage}{french}
\begin{abstract}
Nous considérons la paramétrisation à vitesse 1 du flot horocyclique sur les revêtements abéliens infinis des surfaces compactes à courbure négative. Nous prouvons un résultat asymptotique pour les intégrales ergodiques de fonctions suffisamment régulières. Notre approche, qui ne repose pas sur la dynamique symbolique, est basée sur une décomposition de Fourier générale pour les revêtements abéliens. Cela nous permet d'étendre les techniques d'opérateurs de transfert au cas du flot géodésique dans un cadre de volume infini. Dans le cas de courbure constante, où les paramétrisations à vitesse 1 et de la contraction uniforme des horocycles coïncident, nous retrouvons un résultat de Ledrappier et Sarig. Enfin, nous déduisons une estimation de la déviation de la puissance pour les moyennes ergodiques des horocycles sur les surfaces compactes, sans condition de pincement comme dans les résultats précédents.
\end{abstract}
\end{otherlanguage}

\setcounter{tocdepth}{1}
\tableofcontents

\section{Introduction}

One of the central problems in ergodic theory is understanding the behaviour of typical orbits of measure-preserving flows and, in particular, the asymptotics of their ergodic integrals. For ergodic flows preserving a finite measure, the Birkhoff ergodic theorem implies that the growth of ergodic integrals is linear (proportional to the space average of the observable) and the same for almost every point. On the other hand, the behaviour of ergodic and conservative flows preserving an infinite, $\sigma$-finite measure is rather different: a celebrated theorem of Aaronson \cite{Aar} states that one cannot find an asymptotic rate that applies to almost every orbit and, as such, the growth of ergodic integrals depends on the starting point. As a consequence of Hopf's ratio ergodic theorem, the dependence on the chosen observable is only as a multiplicative constant, namely its space average.

One can therefore hope to describe the ergodic integrals of an ergodic, conservative, measure-preserving flow $(\varphi_t)_{t\in \R}$ on a $\sigma$-finite measure space $(M,\mu)$, with $\mu(M)=\infty$, as
\begin{equation}\label{eq:expan}
\int_0^T f\circ \varphi_t(x) \diff t=a(T)\Phi_T(x) \mu(f)(1+o(1))
\end{equation}
where the factor $a(T)$ describes the \lq\lq correct\rq\rq\ asymptotic growth, and $\Phi_T(x)$ represents an \lq\lq oscillating\rq\rq\ term, which does not depend on the integrable function $f$; moreover, although $\Phi_T$ does not converge pointwise almost everywhere, it perhaps does in some weaker sense.
This indeed has been achieved for some families of {\it parabolic} flows, see, for instance, \cite{AvDoDu, BFRT, LeSa, LeSa2}.

In this paper, we derive an expansion of the form \eqref{eq:expan} (see Theorem \ref{thm:main_result}) for the {\it unit speed} parametrization of the horocycle flow on infinite Abelian covers of compact surfaces of negative, possibly variable curvature. 
We explicitly describe the rate $a(T)$ and the oscillating factor $\Phi_T(x)$ in \eqref{eq:a(T,x)}, along with an estimate for the error term. Next, we establish equidistribution for geodesic translates of horocycle segments, which is our second result (see Theorem \ref{thm:main_result_2}). Finally, as a bonus, we obtain a power deviation estimate for the horocycle ergodic averages on negatively curved compact surfaces (see Theorem C), without requiring any pinching condition as in \cite[Corollary 3.4]{AdBa}.

In the case of constant curvature $-1$, the unit speed parametrization of horocycles coincides with the so-called {\it uniformly contracting} parametrization. We then recover the result by Ledrappier and Sarig in \cite{LeSa}, which applies in general for the uniformly contracting parametrization. 
Differently from \cite{LeSa} and other previous results (e.g., \cite{AvDoDu}), our approach does not use symbolic dynamics; instead, it is geometric in nature and it relies purely on functional analysis and transfer operators methods. The use of transfer operators acting on anisotropic Banach spaces to study parabolic systems appeared first in the seminal work of Giulietti and Liverani \cite{GiLi}, and has proved successful in several settings \cite{FaGoLa, AdBa, BuSi}. It stems from the idea of {\it renormalization}, which, in our case, is provided by the geodesic flow. The lack of uniformity, in general, for the unit speed parametrization of horocycle flows (as opposed to the uniformly contracting one, as suggested by the name) makes the analysis harder and requires the use of weighted transfer operators, analogously to~\cite{AdBa}.

There are two further difficulties that one needs to overcome to apply the transfer operator method in our setting.
Firstly, the (weighted) transfer operator associated with the geodesic flow on $\mathbb{Z}^d$-covers of a compact space acts on smooth functions defined on a non-compact space. Although it may still be possible to establish a Lasota-Yorke type inequality, which is crucial to obtain quasi-compactness of the operator, the non-compactness of the space precludes the direct application of Hennion's theorem \cite{Hen}. 
To this end, we exploit a Fourier-type decomposition of the space of square integrable functions given by the natural $\Z^d$-action of the Galois group of the cover. 
Roughly speaking, this reduces the problem to studying the action of the transfer operator on a family of spaces of functions which \lq\lq behave as if they were defined on a compact space\rq\rq\ (more precisely, they can be seen as sections of line bundles over the compact base manifolds, as in \cite{FlRa}).
In turn, we will study the action of a family of weighted and twisted transfer operators on a fixed Banach space containing densely the space of smooth functions on the compact base manifold. We then conduct the necessary spectral analysis of these operators associated to the geodesic flow which, to the best of the authors' knowledge, is new in the context of an infinite volume setting. 

The second issue is that the flow is not uniformly hyperbolic due to the presence of the flow direction, rendering it only partially hyperbolic. Because of this neutral direction, proving the existence of a spectral gap is extremely difficult (if possible at all). Only in special cases the existence of a spectral gap for transfer operators of partially hyperbolic systems has been proved (see for instance \cite{CaLi} in discrete time or \cite{Tsu} in continuous time).
Nonetheless, drawing from the works of \cite{Liv} and \cite{But}, quasi-compactness of the resolvent of the generator suffices to obtain a spectral decomposition for the transfer operators.
Combined with the idea of renormalization, it will allow us to prove the expansion in \eqref{eq:expan}.

\subsection{Organization of the paper}
The paper is organized as follows. \Cref{sec:result} contains the main results, Theorem \ref{thm:main_result}, Theorem \ref{thm:main_result_2} and Theorem \ref{thm:main_result_3}, preceded by an introduction to the dynamics on surfaces of negative curvature and to covering spaces with an Abelian Galois group. 
\Cref{sec:renorm} contains some needed results on the dynamics of geodesic and of horocycle flows on their unit tangent bundles. \Cref{sec:twisted_hilbert_spaces} deals with covering spaces with an Abelian Galois group: we prove a Fourier-type decomposition which reduces the problem from studying functions on the cover to vectors in a family of (mutually isomorphic) Banach spaces parametrized by the dual of the Galois group. In \Cref{sec:twisted_transfer_operators} and \Cref{sec:spectral_picture}, we study the spectral theory of (weighted) transfer operators for the geodesic flow, twisted by harmonic 1-forms representing cohomology classes that vanish on the cover. In particular, in \S\ref{sec:twisted_transfer_operators} we establish a Lasota-Yorke inequality for the resolvent and in \S\ref{sec:spectral_picture} we deduce a spectral decomposition for the semigroup of transfer operators. Finally, \Cref{sec:horocycle_integrals} contains the proof of our main results, combining the decomposition from \S\ref{sec:abelian_covers}, the renormalization provided by the geodesic flow, and the results on the transfer operators from \S\ref{sec:spectral_picture}.

\section{Setting and results}\label{sec:result}

\subsection{Geodesic and horocycle flows on surfaces of negative curvature}\label{sec:geo_horo_basics}
Let $S$ be a compact, connected surface equipped with a complete smooth Riemannian metric of negative curvature. 
We let $M = T^1S$ denote the unit tangent bundle of $S$. For any $x = (z,v) \in M$, there exists a unique unit speed geodesic $\gamma_x(t)$, defined for all $t \in \R$, such that $\gamma_x(0) = z$ and $\dot{\gamma_x}(0)=v$.
The geodesic flow $(\geo_t)_{t \in \R}$ is the smooth flow on $M$ defined by 
\[
\geo_t(x) = (\gamma_x(t), \dot{\gamma_x}(t)).
\]
The vector field $X$ on $M$ generating the geodesic flow is given by 
\[
Xf(x) = \frac{\diff}{\diff t} \Big\vert_{t=0} f \circ \geo_t(x),
\]
for all $f\in \mathscr{C}^1(M)$. 
The geodesic  flow is an Anosov flow, in particular the following holds. 
Let $D\geo_t \colon TM \to TM$ denote the differential of $\geo_t$. There exist constants $C \geq 1$ and $\lambda>0$, and there exists a $D\geo_t$-invariant splitting of $TM = E_{-} \oplus E_0 \oplus E_{+}$ into 1-dimensional subbundles, where  $E_0 = \langle X \rangle$ and $E_{\pm}$ satisfy
\begin{equation}\label{eq:hyperb}
\| D\geo_t U\| \leq C e^{-\lambda t} \|U\|,\quad \forall U \in E_{-} \qquad \text{and} \qquad \| D\geo_{-t} V\| \leq C e^{-\lambda t} \|V\|, \quad \forall V \in E_{+},
\end{equation}
for all $t\geq 0$. Explicit bounds on $\lambda$ can be expressed in terms of the geometry of $S$; we refer the reader to \cite[Part IV, \S 17.6]{HasKat} for a detailed discussion.
In the case $S$ has constant negative curvature $-1$, then, for all $x \in M$, we have
\[
D\geo_t(x) U_x = e^{-t} U_{\geo_t(x)}, \text{\ for all $U \in E_{-}(x)$} \qquad \text{and} \qquad D\geo_t(x) V_x = e^{t} V_{\geo_t(x)}, \text{\ for all $V \in E_{+}(x)$},
\]
as we will see in \Cref{lemma:Jacobi} below.

A further important remark is that the flow $(\geo_t)_{t \in \R}$ preserves a \emph{contact} form and hence the associated Liouville measure, which, up to a scalar multiple, coincides with the natural Riemann volume on $M$. We will denote this measure by $\vol$. The geodesic flow is thus an example of a \emph{contact Anosov flow}.

We denote by $h_{\topp}$ the topological entropy of the time-one map $\geo_1$.

The distributions $E_{\pm}$ are orientable and are of class $\mathscr{C}^{2-\varepsilon}$ for every $\varepsilon >0$; more precisely they are $\mathscr{C}^{1}$ and their derivative is in the Zygmund class, as proved by Hurder and Katok \cite{HuKa}.
On the other hand, they fail to be $\mathscr{C}^{2}$, unless the surface $S$ has constant curvature.

Both $E_{-}$ and $E_{+}$ integrate to 1-dimensional $\geo_t$-invariant foliations $W^{-}$ and $W^{+}$, whose leaves are called the \emph{stable} and \emph{unstable manifolds} of $(\geo_t)_{t \in \R}$. 
The unit speed motions along the stable and unstable manifolds define two flows $(\horo^{-}_t)_{t \in \R}$ and $(\horo^{+}_t)_{t \in \R}$ on $M$, called the \emph{stable} and \emph{unstable horocycle flow}. 

In this paper, we are going to focus on the stable horocycle flow, which henceforth we will simply call the horocycle flow and will be denoted by $(\horo_t)_{t \in \R}$. By definition, the generating vector field $U$, given by
\[
Uf(x) = \frac{\diff}{\diff t} \Big\vert_{t=0} f \circ \horo_t(x)
\]
for all $f\in \mathscr{C}^1(M)$, spans the line bundle $E_{-}$. 

The horocycle flow is an important example of a \emph{parabolic flow}. It is minimal \cite{Bow, Pla}, has zero entropy \cite{Gur} and is mixing with respect to the unique invariant probability measure $\mu$ \cite{Marcus2}. We remark that, unless the curvature is constant, the invariant measure $\mu$ is singular with respect to the Liouville measure on $M$. 

\subsection{Abelian covers}\label{sec:abelian_covers}
Let $G = \pi_1(S)$ be the fundamental group of the surface $S$ and let $G' = [G,G]$ denote the derived subgroup. 
The quotient $G^{\abel} = G/G'$ is isomorphic to the first homology $H_1(S,\Z)$ of $S$, which is a free abelian group of rank $2g$, where $g \geq 2$ is the genus of $S$.

Intermediate subgroups $G' \leq \Gamma \leq G$ are in 1-to-1 correspondence with subgroups of $H_1(S,\Z)$ via the projection $\Gamma \mapsto \Gamma^{\abel, G} = \Gamma/G'$. Each $\Gamma$ defines a cover $p_0 \colon \tS \to S$ with a Galois group
\[
\Deck := \Aut(\tS / S)
\]
of deck transformations isomorphic to $G/\Gamma = H_1(S,\Z) /  \Gamma^{\abel, G}$. We will assume that the latter has no torsion, which is always the case up to a finite cover, and hence is a free abelian group of rank $1\leq d \leq 2g$.
We fix $d$ linearly independent primitive homology classes
\[
[\gamma_i] =  \gamma_i \, G' \in H_1(S,\Z), \qquad \text{for $i=1,\dots , d$}
\]
so that the elements
\[
D_i =  \gamma_i \, \Gamma \in \Deck
\]
form a basis of $\Deck \simeq \Z^d$.
We say that the associated cover $p_0 \colon \tS \to S$ is a $\Z^d$-cover of $S$.
The cover $\tS$ is equipped with the pullback Riemannian metric under $p_0$, so that $p_0 \colon \tS \to S$ is a Riemannian cover, and the deck transformations act isometrically.
The action of $\Deck$ extends naturally to an action on the unit tangent bundle $\tM = T^1\tS$ of $\tS$ and we have an associated cover $p \colon \tM \to M$.

We equip the cover $\tS$ with the pullback Riemannian area form under $p_0$. By a little abuse of notation, we still denote by $\vol$ the Liouville form on $\tM$. It induces an infinite measure, normalized so that $\vol(M)=1$.

\subsection{The results} Given a positive definite symmetric $d \times d$ matrix $\Sigma$, we define the $\|\cdot\|_{\Sigma}$ norm on $\R^d$ as
\[
\|\xb\|_{\Sigma} := \sqrt{\xb \cdot \Sigma^{-1} \xb}.
\]
Our first main result is the following, and its proof can be found in \Cref{sec:proof_main_result}.

\begin{bigthm}\label{thm:main_result}
	Let $p \colon \tM \to M$ a $\Z^d$-cover of $M = T^1S$, where $S$ is a negatively curved compact surface. There exist constants $C_M \geq 1$ and $\varepsilon >0$, a positive definite symmetric $d \times d$ matrix $\Sigma$, and functions
	\[
	\begin{split}
&t_{\ast}  \colon M \times [C_M,\infty) \to \R_{>0}, \qquad \text{satisfying} \qquad \|T-e^{h_{\topp} \, t_{\ast}(\cdot,T) }\|_{\infty} \leq C_M T^{1-\varepsilon}
 \qquad \text{and} \\ 
&F_{\ast}\colon M \times [0,\infty) \to \R^d \qquad \text{satisfying} \qquad \frac{F_{\ast} (\cdot, T)}{\sqrt{T}} \to \mathcal{N}(0,\Sigma) \text{\ \ in distribution,}
	\end{split}
	\]
for which the following holds.
For every $f \in \mathscr{C}_c^2(\tM)$ and for every $x \in \tM$, there exists a constant $C(f,x)\geq 0$ depending on the $\mathscr{C}^2$-norm of $f$, on the diameter of its support $\supp(f)$, and on the distance between $x$ and $\supp(f)$, such that for all $T\geq C_M$, denoting $t_{\ast} = t_{\ast} (p(x),T)$, and
\begin{equation}\label{eq:a(T,x)}
a(T):=\frac{ h_{\topp}^{d/2}}{(2\pi )^{d/2}\sqrt{\det \Sigma}} 
 	\frac{T}{(\log T)^{d/2}}, \qquad \Phi_T(x):=\exp \left(-\frac{1}{2} \left\| \frac{F_{\ast}(p(x), t_{\ast} )}{\sqrt{t_{\ast}  }} \right\|_{\Sigma}^2 \right),
\end{equation}
 we have
\[
 	\Bigg\lvert \int_{0}^{T} f \circ \horo_s(x) \diff s -    
 	a(T) \Phi_T(x) \int_{\tM} f \diff \mu  \Bigg\rvert \le C_M C(f,x)  \, \frac{T \cdot \log \log T}{(\log T)^{\frac{d+1}{2}}}.
\]
\end{bigthm}
Before proceeding, let us provide some remarks on the result.

\begin{remark}[Winding cycle and normalizing time] The functions $t_{\ast}$ and $F_{\ast}$ are explicitly defined and have precise geometric meaning. The function $t_{\ast}$ is a \emph{normalizing time}, in the sense that it is the time the geodesic flow takes to normalize the horocycle orbit of length $T$ at $x$ to unit size (see \eqref{eq:choice_t_ast} below for the precise definition). In particular, $C_M=1$ and  $t_{\ast}(x,T) = \log T$ for surfaces of constant negative curvature $-1$. 

The vector $F_{\ast}(x,T)$ is the so-called \emph{Frobenius function}, or the \emph{geodesic winding cycle}: its components are ergodic integrals of harmonic 1-forms along the geodesic orbit at $x$ of length $T$ and they describe the behaviour of this geodesic with respect to the cohomology classes that vanish on $\tM$ (roughly speaking, how the geodesic segment \lq\lq winds around\rq\rq\ the cycles in the first homology of $M$ which do not lift to closed loops on the cover); see \eqref{eq:F} and \Cref{sec:proof_main_result} for the precise definition. The covariance matrix $\Sigma$ is associated to the quadratic form defined in \eqref{eq:sigma}, normalized by $-1/(4\pi^2)$.
\end{remark}

\begin{remark}[Asymptotic expansion]\label{rk:asymptotic_expansion}
Following the proof outlined in Section \ref{sec:proof_main_result}, we observe that the result follows by computing the leading term in the expansion of the integral in \eqref{eq:approx_4}. While it is possible, in principle, to derive a complete expansion using stationary phase methods, we refrain from doing so due to the complexity of the computations, particularly for large $d$; we refer to \cite{BFRT} for the analogous computations in the setting of translation flows. 
\end{remark}


\begin{remark}[Generic points]
	The conclusion of \Cref{thm:main_result} holds for any point $x\in \tM$; however, the term $a(T)\Phi_T(x) \mu(f)$ describes the asymptotics of the horocycle integrals only for those points $x$ for which $\Phi_T(x)$ is not \lq\lq too small\rq\rq, namely for $x\in \tM$ such that $F_{\ast}(p(x),t)/t \to 0$ as $t \to \infty$ (whereas, for the other points, we only obtain an upper bound). Indeed, we expect that the points for which the above convergence holds are generic for $\mu$. In the case of hyperbolic surfaces, there is an $\R^d$ family of ergodic invariant Radon measures \cite{BaLe, Sar}, and Sarig and Schapira showed in \cite{SaSc} that the sets of generic points for these measures are parametrized by the possible limits of $F_{\ast}(p(x),t)/t$. We expect a similar characterization to hold in our general setting of variable curvature.
\end{remark}

\begin{remark}[Curvature and dimension]
In the scenario of constant negative curvature, we recover the result presented in \cite{LeSa}. Notably, Theorem \ref{thm:main_result} extends to compact surfaces of variable negative curvature. This extension is facilitated by the one-dimensionality of the stable and unstable manifolds. While this simplification streamlines many computations, it is primarily crucial for proving a Dolgopyat-type inequality (see Proposition \ref{prop:dolgopyat_inequality}) without assuming any bunching condition\footnote{Looking at \cite[Remark 7.6]{GiLiPo} it becomes clear why the bunching condition was not needed in our setting.}. Hence, another suitable choice of Banach spaces, coupled with a bunching condition, would likely suffice to generalize our result to any dimension. This could be achieved using the arguments presented in \cite{AdBa} or \cite{GiLiPo} to establish the required Dolgopyat inequality.
\end{remark}

\begin{remark}[Limit theorems]
Our proof of the main theorem relies on a pure functional analytic approach, wherein we investigate the spectrum of the weighted transfer operator $\mathcal{L}_t$ associated with the geodesic flow on $T^1S$. Establishing the quasi-compactness of $\mathcal{L}_t$ would grant access to various statistical properties, such as decay of correlations. Particularly, employing a spectral method to study complex perturbations of the operator $\mathcal{L}_{t,\nu}$ (see \cite{Gou}) would imply limit theorems such as the Central Limit Theorem (CLT), local CLT, and large deviations, among others. However, the absence of uniform hyperbolicity has posed significant challenges in proving quasi-compactness for $\mathcal{L}_t$.

Nonetheless, the pioneering work of Liverani \cite{Liv} indicates that quasi-compactness of the resolvent of the generator of the semigroup $\mathcal{L}_t$ is enough to get, at least, exponential decay of correlations. Building on this insight, our analysis on the resolvent of the generator of the perturbed operator $\mathcal{L}_{t,\nu}$ could yield an alternative approach to establish the CLT for the geodesic flow on negatively curved manifolds and exploring its finer statistical properties (See \Cref{prop:pert2} and \Cref{rmk:CLT}). 
\end{remark}

The same methods we develop in this paper yield an equidistribution result for geodesic translates of horocycle segments, which is our second main result. Its proof is contained in \Cref{sec:proof_main_result_2}.

\begin{bigthm}\label{thm:main_result_2}
	Let $p \colon \tM \to M$ be a $\Z^d$-cover of $M = T^1S$, where $S$ is negatively curved compact surface. There exists a a positive definite symmetric $d \times d$ matrix $\Sigma$
for which the following holds.
Let  $\eta$ be a 1-form of class $\mathscr{C}^2$ on $\tM$ with compact support, fix $x \in \tM$ and $\sigma >0$. Define 
\[
\gamma_{x,\sigma}(s) = \horo_s(x), \qquad \text{for} \qquad s \in [0,\sigma],
\]
the horocycle segment starting at $x$ of length $\sigma$. There exist $C(\sigma)$ depending on $\sigma$, $\varpi(\sigma,x)$ depending on $x$ and on $\sigma$, and $C(\eta,x)\geq 0$ depending on the $\mathscr{C}^2$-norm of $\eta$, on the diameter of its support $\supp(\eta)$, and on the distance between $x$ and $\supp(\eta)$, such that for all $t \geq C(\sigma)$ we have
\[
\left\lvert \frac{(2\pi  t)^{\frac{d}{2}}\sqrt{\det \Sigma}}{e^{h_{\topp} t} \, \varpi(\sigma,x)} \int_{\geo_{-t} \circ \gamma_{x, \sigma}} \eta - \int_{\tM} \langle \eta, U\rangle \diff \mu \right\rvert \leq C(\sigma) C(\eta,x) \frac{\log t}{ \sqrt{t}}.
\]
\end{bigthm}
The term $\varpi(\sigma,x)$ is the Margulis length of the arc $\gamma_{x,\sigma}$, see \Cref{prop:Margulis}.
An important difference with \Cref{thm:main_result} is the absence of the oscillating factor in \Cref{thm:main_result_2} above. As it will be clear in \Cref{sec:horocycle_integrals}, the presence of the term $\Phi_T$ involving the Frobenius function in \Cref{thm:main_result} is due to the varying position of the renormalized orbit segment in the cover. On the contrary, for expanding horocycles, the renormalized segment does not move in space, and hence the asymptotics do not oscillate. A somehow similar phenomenon happens in the case of non-compact, finite volume hyperbolic surfaces, see \cite{FlFo} and references therein. \\

Finally, as an outgrowth of our methods, we are also able to prove a power deviation estimate for the horocycle ergodic averages on negatively curved compact surfaces. The case of hyperbolic surfaces is due to Burger \cite{Bur}, and has been later refined by several authors \cite{FlFo, BuFo, Str, Rav}. The case of variable curvature has been studied by Adam and Baladi \cite{AdBa} under some additional assumptions on the curvature of the surface. 
Unlike \cite[Corollary 4.9]{AdBa}, we do not require any pinching condition on the curvature, making our result completely general. The proof of the following theorem is given in Section \ref{sec:proofC}, which provides also more precise information on $a\in (0,1)$ given in \eqref{eq:deviation} in terms of the topological entropy $h_{\topp}$.

\begin{bigthm}\label{thm:main_result_3}
Let $M = T^1S$ be the unit tangent bundle of a compact surface $S$ of negative curvature, let $(\horo_s)_{s \in \R}$ be the unit speed horocycle flow on $M$, and let $\mu$ be the unique invariant probability measure. There exist $a \in (0,1)$ and $\sC>0$ such that, for each $f\in \cC^2(M)$, $x\in M$ and $T \ge 1$, we have
\begin{equation}\label{eq:deviation}
\left\lvert \frac 1T \int_0^T f\circ \horo_s(x)\diff s - \mu(f) \right\rvert \leq \dfrac{\sC}{T^{a}} \|f\|_{\cC^2}.
\end{equation}
\end{bigthm}


\section{Renormalization}\label{sec:renorm}
In this section we collect some results on the dynamics of geodesic and of horocycle flows on the unit tangent bundles $M$ and $\tM$.
Although these two flows do not commute, they satisfy a \lq\lq renormalization relation\rq\rq, which plays a fundamental role in the study of the dynamics and the ergodic theory of the horocycle flow.

\begin{lemma}[Renormalization relation]\label{lem:renormalization}
For every $x\in M$ and $s,t \in \R$, there exists a unique $\tau(s,t,x) \in \R$ such that
\begin{equation}\label{eq:renormalization}
\geo_t \circ \horo_s(x) = \horo_{\tau(s,t,x)} \circ \geo_t(x).
\end{equation}
\end{lemma}
\begin{proof}
For any $x \in M$ and $s \in \R$, we have $\horo_s(x) \in W^{-}(x)$. Since $\geo_t(W^{-}(x)) = W^{-}(\geo_t(x))$ there exists $\tau = \tau(s,t,x) \in \R$ such that \eqref{eq:renormalization} is satisfied.

If there were $\tau \neq \tau'$ satisfying \eqref{eq:renormalization}, then there would be a closed orbit for the horocycle flow, which contradicts minimality \cite{Bow, Pla}.
\end{proof}

The renormalization time $\tau(s,t,x)$ of \Cref{lem:renormalization} is intimately related to the curvature of the surface, as the next lemma shows. Its derivative $J(t,x)$ will play a crucial role in this paper.

\begin{lemma}\label{lemma:Jacobi}
For each $x\in M$, the function
\[
J(t,x) := \frac{\partial}{\partial s}\Big\vert_{s=0} \tau(s,t,x)
\]
satisfies
\[
\partial^2_{t}{J}(t,x) + K(\geo_t(x))J(t,x) =0, \qquad J(0,x)=1, \qquad \lim_{t \to \infty} J(t,x) = 0,
\]
where $K$ denotes the curvature. Moreover, given $\overline k, \underline k<0$ such that $\underline k \leq K(\geo_t(x)) \leq \overline k$ for each $x\in M$ and $t\in \R^+$, then
\begin{equation}\label{eq:Jbound}
e^{-\sqrt{-\underline k}t} \le J(t,x) \le e^{-\sqrt{-\overline k}t}, \qquad \forall t\in \R_{>0}, \quad \forall x\in M.
\end{equation}
\end{lemma}
\begin{proof}
For convenience, let us set $J_x(t):=J(t,x)$. For any $x \in M$, $\geo_t \circ \horo_s(x)$ defines a 1-parameter family of geodesics, hence 
\[
\frac{\partial}{\partial s}\Big\vert_{s=0} \geo_t \circ \horo_s(x) = D\geo_t(x) U_x
\]
is a Jacobi field. In particular, from \eqref{eq:renormalization} we can write
\[
D\geo_t(x) U_x = J_x(t) U_{\geo_t(x)},
\]
and the function $J_x(t)$ satisfies the Jacobi equation \cite[Chapter 5]{doCarmo}
\[
\ddot{J_x}(t) + K(\geo_t(x))J_x(t) =0.
\]
The initial condition $J_x(0)=1$ follows immediately from \eqref{eq:renormalization}.
Note that, since the curvature $K$ is negative, the function $J_x(t)$ is positive and convex. Since $U \in E_{-}$, the function $J_x$ decays exponentially at infinity, uniformly in $x$. \\
It remains to prove \eqref{eq:Jbound}. Fix $x\in M$ and set $k(t):=K(\geo_t(x))$. Let $\bar J(t)=e^{-\sqrt{-\bar k}t}$ be the solution to the initial value problem $\ddot{\overline J}(t)+ \overline J(t)\overline k=0$, $\overline J(0)=1$ and such that $\lim_{t\to \infty} \overline J(t)=0$. We are going to prove that $J(t)\le \overline J(t)$ for each $t \ge 0$, where we dropped the dependence of $J_x(t)$ on $x$. The other bound is done similarly. Since $\dot J(t)$ and $\dot{\overline J}(t)$ are bounded uniformly in $t$ and using that $J(t)$ and $\overline J(t)$ converge to zero as $t \to +\infty$, an integration by parts gives us
\[
\begin{split}
0&=\int_t^{+\infty} [\overline J(s)(\ddot J(s)+k(s)J(s))-J(s)(\ddot{\overline J}(s)+ \overline J(s)\overline k)]\diff s\\
&=-[\overline J \dot J-J \dot{\overline J}](t)+\int_t^{+\infty}[k(s)-\overline k]J(s)\overline J(s) \diff s.
\end{split}
\]
Therefore, since $k(t)\le \overline k<0$, 
\[
[-\overline J \dot J+J \dot{\overline J}](t) \ge 0,\qquad  \forall t \ge 0,
\]
which implies that
\[
\dfrac{\diff}{\diff t} \left( \frac{\overline J(t)}{J(t)} \right) \cdot J^2(t) \ge 0.
\]
Hence, the function $ \dfrac{\overline J(t)}{J(t)} $ is increasing for each $t \ge 0$, from which it follows
\[
 \frac{\overline J(t)}{J(t)}  \ge \frac{\overline J(0)}{J(0)}=1, \qquad \forall t \ge 0.
\]
The last assertion follows directly from \eqref{eq:hyperb} and \eqref{eq:Jbound}.
\end{proof}

\begin{remark}\label{rmk:Ju}
Note that an analogous computation as in the previous lemma applied to the unstable vector field $V$ instead of $U$ yields the same results for a function $J^{u}_{t}$ satisfying 
\[
D\geo_{-t}(x)V_x=J^{u}(-t, x)V_{\geo_{-t}(x)}.
\]
\end{remark}

Let us recall that we have an invariant splitting $TM = E_{-} \oplus E_0 \oplus E_{+}$ of the tangent bundle $TM$, where $E_0=\langle X \rangle $ and $X$ is the generator of the flow $(\geo_t)_{t \in \R}$, and where $E_{-} = \langle U \rangle$ and $U$ is the generator of the stable horocycle flow $(\horo_t)_{t\in \R}$. For further purposes, we denote by $V$ the generator of the unstable horocycle flow $(\horo^+_t)_{t\in \R}$. From \Cref{lemma:Jacobi} and Remark \ref{rmk:Ju}, we have
 \begin{equation}\label{eq:potential}
 \begin{split}
 J_{-t}(x)&=\exp\left({\int_0^t -\operatorname{div} (X|_{E^-})\circ \geo_{-s}(x)\mathrm d s}\right)=D\geo_{-t}(x)|_{E_-},\\
  J^u_{-t}(x)&=\exp\left({\int_0^{-t} \operatorname{div} (X|_{E^+})\circ \geo_s(x)\mathrm d s}\right)=D\geo_{t}(x)|_{E_+},
 \end{split}
 \end{equation}
 where  $D\geo_{-t}(x)|_{E_-}$ denotes the derivative of $\geo_{-t}$ along the (one-dimensional) stable direction $E_-$, and similarly for the derivative $D\geo_{t}(x)|_{E_+}$ along the unstable direction $E_+$. 
In particular, 
\begin{equation}\label{eq:div}
\begin{split}
&\frac{\diff}{\diff t}\Big\vert_{t=0} D\geo_{-t}(U)=\Phi^- U,\\
&\frac{\diff}{\diff t}\Big\vert_{t=0} D\geo_{t}(V)=\Phi^+ V,
\end{split}
\end{equation}
where $\Phi^{\pm}:= \pm \operatorname{div}(X|_{E_{\pm}})$ are strictly positive $\cC^{2-\varepsilon}$ functions for all $\varepsilon > 0$ (see \cite[p.56-57]{HuKa} or \cite[(3.1)]{Gre}). 
In particular, from \eqref{eq:hyperb}, we have $\sqrt{-\overline k}\ge \lambda$.

An immediate consequence of the previous results is the following.
\begin{corollary}\label{cor:inverse_of_Jt}
	For any $x\in M$ and $t,s \in \R$, we have 
	\[
	J_t(x)^{-1} = J_{-t}(\geo_{t}(x)) \qquad \text{and} \qquad \frac{\partial \tau}{\partial s}(s,t,x) = J_t(\horo_s(x)).
	\]
\end{corollary}
\begin{proof}
	The first claim follows directly from \eqref{eq:potential}. The second is a consequence of the cocycle relation $\tau(s+r,t,x) = \tau(r,t,\horo_s(x)) + \tau(s,t,x)$ and the equality $J(t,x) = \frac{\partial}{\partial s}\big\vert_{s=0} \tau(s,t,x)$ from \Cref{lemma:Jacobi}.
\end{proof}

As we already remarked, in case of constant curvature, from \Cref{lemma:Jacobi} we obtain
\[
J_x(t) = e^{-\sqrt{-K} \, t}, \qquad \text{hence} \qquad \tau(s,t,x) = e^{-\sqrt{-K} \, t} \cdot s.
\]

We say that the parametrization of the stable foliation defined by the horocycle flow is \emph{uniformly contracting}.
In the case of variable curvature, although the function $\tau(s,t,x)$ does not have such a simple expression, there still exists a uniformly contracting parametrization of $W^{-}$, as shown by Marcus, from Margulis's work \cite{Marcus, Margulis}.

\begin{proposition}[{\cite{Marcus, Margulis}}]\label{prop:Margulis}
Let $h_{\topp} >0$ denote the topological entropy of the time-1 map $\geo_1$. There exists a continuous additive cocycle $\varpi(t,x)$ such that the flow $(\widetilde{\horo}_t)_{t \in \R}$ defined by
\[
\widetilde{\horo}_{\varpi(t,x)}(x) = \horo_t(x)
\]
satisfies
\[
\geo_t \circ \widetilde{\horo}_s(x) = \widetilde{\horo}_{e^{-h_{\topp} \, t} \cdot s} \circ \geo_t(x).
\]
Moreover, $(\widetilde{\horo}_t)_{t \in \R}$ is uniquely ergodic and the unique probability invariant measure $m$ is the measure of maximal entropy for the geodesic flow.
\end{proposition}
From the previous result it is possible to deduce the asymptotics of $\tau(s,t,x)$ as $s \to \infty$, as the next lemma shows. It was proven originally by Marcus \cite{Marcus2}, we reproduce the proof here for completeness.
\begin{lemma}\label{lemma:Marcus}
There exists a constant $C_\tau >0$ such that
\[
C_\tau^{-1} \leq \frac{\tau(s,t,x)}{s} e^{h_{\topp} \, t} \leq C_\tau,
\]
for all $x \in M$, $t \leq 0$, and $s \geq 1$. Moreover,
\[
\lim_{s \to \infty} \frac{\tau(s,t,x)}{s} = e^{-h_{\topp} \, t}
\]
uniformly in $x \in M$.
\end{lemma}
\begin{proof}
	The first estimate can be found in \cite[Lemma C.3]{GiLiPo}.
	We reproduce Marcus's argument for the second claim.
From \Cref{prop:Margulis} and unique ergodicity of the horocycle flow \cite{Marcus}, we deduce that there exists a constant $a >0$ such that
\[
\frac{1}{n} \varpi(n,x) = \frac{1}{n} \sum_{j=0}^{n-1} \varpi(1,\horo_j(x)) \to a,
\]
as $n \in \N$ tends to infinity, uniformly in $x\in M$. It is easy to see that the same uniform limit $\frac{1}{s} \varpi(s,x) \to a$ holds when $s \in \R$ tends to infinity.
Since we have
\[
\widetilde{\horo}_{e^{-h_{\topp} \, t} \varpi(s,x)} \circ \geo_t(x) = \geo_t \circ \widetilde{\horo}_{ \varpi(s,x)} (x) =  \geo_t \circ {\horo}_{s}(x) = \horo_{\tau(s,t,x)} \circ \geo_t = \widetilde{\horo}_{\varpi(\tau(s,t,x), \geo_t(x))} \circ \geo_t(x), 
\]
it follows that
\[
e^{-h_{\topp} \, t} \varpi(s,x) = \varpi(\tau(s,t,x), \geo_t(x)).
\]
We will recover this relation in \Cref{rk:leafwise_measure}, see \eqref{eq:self_similarity_of_alpha}. 
We deduce that 
\[
\frac{\tau(s,t,x)}{s} = \frac{\tau(s,t,x)}{ \varpi(\tau(s,t,x), \geo_t(x))} \, \frac{e^{-h_{\topp} \, t} \varpi(s,x) }{s},
\]
which tends to $e^{-h_{\topp} \, t}$ as $s \to \infty$, uniformly in $x\in M$. It is easy to see that the limit is also uniform in $t\leq 0$,  therefore the proof is complete.
\end{proof}


Finally, as we will be working on the cover $\widetilde{M}$, it is essential to ensure that our objects are $\Deck$-invariant, enabling us to treat them as defined on the compact manifold $M$. This is demonstrated in the following two results.

\begin{lemma}\label{lem:flows_commute}
The geodesic and horocycle flows commute with all deck transformations.
\end{lemma}
\begin{proof}
We can identify any $D \in \Deck$ with $[\gamma] + \Gamma^{\abel, G}$, where $[\gamma]$ is a homology class in  $H_1(M,\Z)$. 
Since the time-$t$ map $\geo_t$ is isotopic to the identity, it acts trivially on the homology, namely $(\geo_t)_\ast [\gamma] = [\gamma]$. This implies that $\geo_t$ commutes with $D$. The same proof applies to the horocycle flow $\horo_t$.
\end{proof}

\begin{corollary}\label{cor:Jt_invariant}
For any $s,t \in \R$, the functions $\tau(s,t,\cdot)$ from \Cref{lem:renormalization} and $\varpi(t,\cdot)$ from \Cref{prop:Margulis} are $\Deck$-invariant.
\end{corollary}
\begin{proof}
For any $D \in \Deck$, by \Cref{lem:renormalization} and \Cref{lem:flows_commute} we have
\[
D(\horo_{\tau(s,t,D(x))}(x)) = \horo_{\tau(s,t,D(x))}(D(x)) = \geo_{-t} \circ \horo_s \circ \geo_t(D(x)) = D(\geo_{-t} \circ \horo_s \circ \geo_t(x)).
\]
Again by \Cref{lem:renormalization}, we conclude $\tau(s,t,D(x)) = \tau(s,t,x)$. Similarly,
\[
\widetilde{\horo}_{\varpi(t,D(x))}(D(x)) = \horo_t(D(x)) = D(\horo_t(x)) = D(\widetilde{\horo}_{\varpi(t,x)}(x)).
\]
Since, for any $r\in \R$, the map $\widetilde{\horo}_r$ is isotopic to the identity, as in \Cref{lem:flows_commute}, it commutes with $D$. This implies that $\varpi(t,D(x)) = \varpi(t,x)$.
\end{proof}

\section{Twisted Hilbert spaces}\label{sec:twisted_hilbert_spaces}
In this section we introduce the Fourier decomposition of our space. We recall the notations from \Cref{sec:abelian_covers}. 
Let us fix a compact connected subset $\cF \subset \tM$ whose boundary has zero measure such that the restriction of $p$ to the interior of $\cF$ is injective and $p(\cF) = M$. We say that $\cF$ is a fundamental domain for the cover $p$.

Any function on $M$ can be seen as a $\Deck$-invariant function on $\tM$, and vice-versa. Under this identification, it is not hard to see that 
\[
\int_M f \diff \vol = \int_{\cF} f \diff \vol.
\]
In particular, for all $q\geq 1$, the Banach spaces $L^q(M)$ and $L^q(\cF)$ are naturally isomorphic.

For any $i=1,\dots, d$, let $\omega_i \in H^1(S,\R)$ be the cohomology class defined by
\[
\omega_i(\Gamma^{\abel, G})=0, \qquad \text{and} \qquad \omega_i([\gamma_j]) = \delta_{i \, j},
\]
where $ \delta_{i \, j}$ is the Kronecker delta. 
By the Hodge Theory, we can identify $\omega_i$ with a harmonic 1-form on $S$. We denote by $\mathcal{H}$ the real vector space 
$\cH := \langle \omega_1, \dots, \omega_d \rangle$ and by $\cH(\Z)$ the $\Z$-module $\Z\omega_1 \oplus \cdots \oplus \Z \omega_d$. Their quotient is a $d$-dimensional torus
\[
\T^d := \cH / \cH(\Z).
\]

For any $\omega \in \T^d$, we define
\[
E_\omega \colon \Deck \to \mathbb{S}, \qquad E_\omega([\gamma] + \Gamma^{\abel, G}) = e^{2 \pi \imath \omega \cdot [\gamma]}.
\]
Note that, if a function $f \colon \tM \to \C$ satisfies $f \circ D^{-1} = E_\omega(D) \, f$ for all $D \in \Deck$, then $|f|$ is a $\Deck$-invariant function and hence it is well defined on $M$. 
Given $\omega \in \T^d$, we define
\[
L^2(M,\omega) := \left\{ f \colon \tM \to \C \ : \ f \circ D^{-1} = E_\omega(D) \, f \text{\ for all $D \in \Deck$, and\ } \int_{\cF} |f|^2 \diff \vol < \infty \right\}. 
\]
We equip $L^2(M,\omega)$ with the inner product 
\[
\langle f,g\rangle = \int_{\cF} f \cdot \overline{g} \, \diff \vol,
\]
which turns into a Hilbert space.
For every integer $\ell \geq 0$, we also set
\[
\mathscr{C}^\ell(M,\omega) = L^2(M,\omega) \cap \mathscr{C}^\ell(\tM).
\]

For any continuous function $f \in \mathscr{C}^0_c(\tM)$ with compact support, let us define
\[
\pi_\omega(f)(x)= \sum_{D \in \Deck} E_\omega(D) \cdot f\circ D(x).
\]
Note that, since $f$ has compact support, the sum on the right hand side above is finite for any $x \in \tM$.
\begin{lemma}\label{lem:proj_cl}
For every $\ell \geq 0$, 
\[
\pi_\omega \colon \mathscr{C}^\ell_c(\tM) \to \mathscr{C}^\ell(M,\omega).
\]
Moreover, for every $f \in \mathscr{C}^\ell_c(\tM)$ and for any $x \in \tM$, we have
\[
f(x) = \int_{\T^d} \pi_\omega(f)(x) \diff \omega.
\]
\end{lemma}
\begin{proof}
Fix $\omega \in \T^d$. 
From the fact that $\Deck$ acts properly discontinuously, it follows immediately that $\pi_\omega(f)$ is a $\mathscr{C}^\ell$-function whenever $f \in \mathscr{C}^\ell_c(\tM)$.

Let us fix a norm $\| \cdot \|$ on $\Deck \simeq \Z^d$.
Since $f$ has compact support, there exists a constant $C(f)>0$ such that $f\circ D(x)=0$ for all $x \in \cF$ whenever $\|D\| \geq C(f)$. Note that, moreover, 
\[
\widetilde{C}(f) := C(f)^{2(d+1)} \sum_{D \in \Deck} \|D\|^{-2(d+1)} 
\]
is finite.
By Cauchy-Schwarz, for any $x \in \cF$, we have
\[
\begin{split}
|\pi_\omega(f)(x)|^2 &= \left\lvert \sum_{D \in \Deck} E_\omega(D) \cdot f\circ D(x) \cdot \frac{\|D\|^{d+1}}{\|D\|^{d+1}}\right\rvert^2 \\
&\leq \left(\sum_{D \in \Deck}\|D\|^{-2(d+1)} \right)\left(\sum_{D \in \Deck} |f|^2 \circ D(x) \cdot \|D\|^{2(d+1)}\right)\\
&\leq \widetilde{C}(f) \sum_{D \in \Deck} |f|^2 \circ D(x).
\end{split}
\]
Thus, since the orbit of $\cF$ under $\Deck$ tessellates $\tM$, we have
\[
\int_\cF |\pi_\omega(f)|^2 \diff \vol \leq \widetilde{C}(f) \sum_{D \in \Deck} \int_\cF |f|^2 \circ D \, \diff \vol = \widetilde{C}(f) \int_\tM |f|^2 \diff \vol,
\]
which is finite by assumption. 

Let us show that $\pi_\omega(f) \circ D_0^{-1} = E_\omega(D_0) \cdot \pi_\omega(f)$.
For every $D_0 = [\gamma_0] + \Gamma^{\abel,G} \in \Deck$, we have $E_\omega(D_0^{-1}) = E_\omega( -[\gamma_0] + \Gamma^{\abel, G}) = E_{-\omega}(D_0) = E_\omega(D_0)^{-1}$, so that 
\[
\begin{split}
\pi_\omega(f) \circ D_0^{-1} &= \sum_{D \in \Deck} E_\omega(D) \cdot f\circ (D_0^{-1} \cdot D ) \\
&= \sum_{D \in \Deck} E_\omega(D_0^{-1} \cdot D) \cdot  E_\omega(D_0) \cdot f\circ (D_0^{-1} \cdot D ) =E_\omega(D_0) \cdot \pi_\omega(f).
\end{split}
\]
Finally, for the last claim, we note that
\[
\int_{\T^d} E_\omega(D) \diff \omega = \int_{\T^d} e^{2 \pi \imath \omega \cdot [\gamma]} \diff \omega = 0 \qquad \text{if and only if} \qquad D = [\gamma] + \Gamma^{\abel, G} \neq 0,
\]
and is equal to 1 otherwise. Therefore,
\[
\int_{\T^d} \pi_\omega(f)(x) \diff \omega = \sum_{D \in \Deck}  f\circ D(x) \int_{\T^d} E_\omega(D) \diff \omega = f(x),
\]
which completes the proof.
\end{proof}

Henceforth, we will simply write $f_\omega$ in place of $\pi_\omega(f)$. 
We note the following fact: let 
\[
\int^{\oplus}_{\T^d} L^2(M, \omega) \diff \omega
\]
be the direct integral of the Hilbert spaces $L^2(M, \omega)$. Then, by \Cref{lem:proj_cl}, the map 
\[
\Pi \colon \mathscr{C}^0_c(\tM) \to \int^{\oplus}_{\T^d}L^2(M, \omega) \diff \omega \qquad \text{given by} \qquad \Pi(f) = ( f_\omega )_{\omega \in \T^d}
\]
is well defined. It is possible to prove that $\Pi$ extends to a unitary equivalence $\Pi\colon L^2(\tM) \to \int^{\oplus}_{\T^d} L^2(M, \omega) \diff \omega$; we omit the proof of this since we will not use it in the paper.


\subsection{Unitary equivalence of twisted spaces}\label{sec:xi_and_x0}

We now verify that the spaces $L^2(M, \omega)$ are all unitarily equivalent.
Let $\omega \in \cH$. Note that the pullback $p_0^{\ast}\omega$ of $\omega$ on $\tS$ is an exact 1-form.
By a slight abuse of notation, we write $p^{\ast}\omega$ to denote its pullback to $\tM$ under the canonical projection $\tM = T^1\tS \to \tS$.

Fix $x_0 \in \cF$. For any $x \in \tM$, the integral 
\begin{equation}\label{eq:roof}
\xi_\omega(x):= \int_{x_0}^x p^\ast \omega
\end{equation}
is well defined, since it does not depend on the choice of path connecting $x_0$ to $x$. 
For any measurable function $f$ on $\tM$, we define 
\[
\Xi_\omega(f) = f \cdot e^{2 \pi \imath \xi_\omega}.
\]

\begin{lemma}\label{lem:proj_Xi}
For every $\omega \in \cH$, we have
\[
\pi_\omega = \Xi_{-\omega} \circ \pi_0 \circ \Xi_{\omega}.
\]
\end{lemma}
\begin{proof}
Note that $D^\ast p^\ast \omega = p^\ast \omega$ for every deck transformation $D$, since $p \circ D = p$. Thus, 
\[
\xi_\omega(D(x)) = \int_{x_0}^{D(x)} p^\ast \omega = \int_{x_0}^{D(x_0)} p^\ast \omega + \int_{D(x_0)}^{D(x)} p^\ast \omega = \int_{x_0}^{D(x_0)} p^\ast \omega + \int_{x_0}^{x} p^\ast \omega.
\]
Recalling the definition of $E_\omega$, we obtain
\begin{equation}\label{eq:xi_and_D}
e^{2\pi \imath \xi_\omega(D(x))} = E_\omega(D) \cdot e^{2\pi \imath \xi_\omega(x)}. 
\end{equation}
From this, we conclude
\[
\pi_0 \circ \Xi_{\omega}(f) = \sum_{D\in \Deck} ( f \cdot e^{2 \pi \imath \xi_\omega}) \circ D =  \sum_{D\in \Deck} f \circ D \cdot E_\omega(D) \cdot e^{2\pi \imath \xi_\omega}= \Xi_{\omega}\circ \pi_{\omega}(f),
\]
which proves the result.
\end{proof}

\begin{lemma}\label{lem:isom_L2omega}
Let $\omega \in \cH$. For every $\eta \in \cH$, the map $\Xi_{\omega}$ is a unitary operator
\[
\Xi_{\omega} \colon  L^2(M,\eta + \omega) \to  L^2(M,\eta).
\]
Moreover, $\Xi_\omega$ is a linear isomorphism between $\mathscr{C}^\ell(M,\eta + \omega)$ and $\mathscr{C}^\ell(M,\eta)$ for every $\ell \geq 0$.
\end{lemma}
\begin{proof}
Since $\Xi_\omega^{-1} = \Xi_{-\omega}$, the map $\Xi_\omega$ is a linear bijection. Furthermore, 
\[
\langle \Xi_\omega(f), \Xi_\omega(g)\rangle = \int_\cF f \, e^{2 \pi \imath \xi_\omega} \cdot \overline{g \, e^{2 \pi \imath \xi_\omega}} \diff \vol = \int_\cF f  \cdot \overline{g} \diff \vol = \langle f, g\rangle,
\]
which proves the first claim.

We now prove the second claim for $\ell = 1$, the general case is left to the reader. Fix a unit norm vector field $W$ on $\tM$ and a point $x$. Then,
\begin{equation}\label{eq:vect-field}
\begin{split}
|W(\Xi_\omega(f))| &= |W(f \cdot e^{2\pi \imath \xi_\omega})(x)| = |Wf(x)| + |f(x)| \cdot 2\pi |W\xi_\omega(x)| \\
&\leq \|Wf\|_\infty + \|f\|_\infty \cdot 2 \pi |\omega_x(W)| \leq \|f\|_{\mathscr{C}^1} \cdot (1+\|\omega\|_\infty).
\end{split}
\end{equation}
This completes the proof.
\end{proof}


\section{Twisted transfer operators}\label{sec:twisted_transfer_operators}

Fix $r\in(2,3)$. 
For each $t \in \R^+$, we define the transfer operator $\cL_t: \cC^{r-1}(M)\to \cC^{r-1}(M)$ by
\[
\cL_{t}f(x)= J_{-t}(x)\cdot f\circ \geo_{-t}(x), \qquad x\in M, \quad f\in \cC^{r-1}(M),
\]
where $J_{-t}$ is the function given in \eqref{eq:potential}.
 The main technical tool of the paper is a twisted transfer operators which correspond to perturbations of the operator $\cL_t$. The idea is to translate the action of $\cL_t$ on the spaces $\mathscr{C}^\ell(M,\omega)$ into the action of a family of twisted operators $\cL^{(\omega)}_t$, defined below, acting on the same space $\mathscr{C}^\ell(M,0)=\mathscr{C}^\ell(M)$, and then extend it to a suitable Banach space on which $\cL_t^{(\omega)}$ has good spectral properties (see the next section).
 For each $\omega \in \T^d$, we then consider the operator
\begin{equation}\label{eq:TTO}
\cL^{(\omega)}_t=\Xi_{\omega} \circ  \cL_t \circ \Xi_{-\omega}.
\end{equation}
\begin{remark}[Remark on the constants]\label{rmk:constant} To enhance readability, we will use the symbol $\sC$ to represent a general positive constant. This constant may rely on various objects such as $X, V, U$, among others, which depend on the geometry of the system, but notably does not vary with time $t$. Additionally, given our focus on perturbations of the operator $\cL_t$ with respect to $\omega$, we will keep track of quantities depending, among other things, on the $\cC^r$ norm of $\omega$, and use the subscript $\omega$ (e.g. $C_{\omega}$). The values of constants $\sC$ and $C_\omega$ may vary between occurrences, even within the same line.
\end{remark}

\begin{lemma}\label{lem:invar}
For each $\omega \in \T^d$ and each $t \in \R^+$, $\cL_{t}^{(\omega)}$ is a well-defined bounded linear operator on $\cC^{r-1}(M,0)$.
\end{lemma}
\begin{proof}
By \Cref{lem:renormalization}, the function $J_t$ has the same regularity as $E_{-}$, in particular $J_t$ is of class $\mathscr{C}^{r-1}$. 
Moreover, since $J(-t,x)=\partial_s |_{s=0}\tau(s,t,x)$, \Cref{cor:Jt_invariant} implies that $J_{-t}\circ D^{-1}=J_{-t}$ for each $D\in \Deck$, i.e. the function $J_{-t} \in \cC^{r-1}(M,0)$.

Take $f \in \cC^{r-1}(M,0)$ and $x\in M$. Note that $\cL_t (\Xi_{-\omega}f) \in \cC^{r-1}(M, \omega)$. Indeed,
\[
\cL_t (\Xi_{-\omega}f)(x)= \cL_t (e^{-2\pi\imath \xi_\omega}f)(x)=J_{-t}(x) e^{-2\pi\imath \xi_\omega(\geo_{-t}(x))}f(\geo_{-t}(x)),
\]
so that, by \Cref{lem:flows_commute} and the invariance of $J_{-t}$ by $\Deck$, we obtain
\[
\cL_t (\Xi_{-\omega}f)(D^{-1}(x)) = J_{-t}(x) e^{-2\pi\imath \xi_\omega(D^{-1}(\geo_{-t}(x)))}f(\geo_{-t}(x)) = E_\omega(D) \cL_t (\Xi_{-\omega}f)(x),
\]
where we also used \eqref{eq:xi_and_D}. 
Therefore, by \Cref{lem:isom_L2omega}, we conclude $\cL^{(\omega)}_tf \in \cC^{r-1}(M,0)$.  

Explicitly, we have
\[
\begin{split}
\Xi_{\omega}( \cL_t (\Xi_{-\omega}f(x)))&=\Xi_{\omega}( J_{-t}(x)e^{-2\pi\imath \xi_\omega(\geo_{-t}(x)) }f(\geo_{-t}(x)))\\
&= J_{-t}(x)e^{2\pi \imath (\xi_\omega(x) - \xi_\omega(\geo_{-t}(x)) )}f(\geo_{-t}(x)).
\end{split}
\]
Finally, since $|e^{2\pi \imath (\xi_\omega(x)-\xi_\omega(\geo_{-t}(x)))}|\le 1$ for each $x$ and $\omega$, we have 
\[
\|\cL^{(\omega)}_tf\|_{L^2}\le \|\cL_t f\|_{L^2}<\infty. \qedhere
\]
\end{proof}

Next, it is convenient to introduce the function 
\begin{equation}\label{eq:F}
F_{t,\omega}(x)= \xi_\omega(\geo_{t}(x))- \xi_\omega(x) = \int_x^{\geo_{t}(x)} p^{\ast}\omega = \int_0^t \langle \omega, X \rangle\circ \geo_{s}(x) \diff s,
\end{equation}
where we used the fact that $p^\ast \omega$ is exact on $\tM$.
The function $F_{t,\omega}$ defined in \eqref{eq:F} is an equivalent formulation of the {\it Frobenius function} for \cite{AvDoDu} and the {\it geodesic winding cycle} of \cite{KatSun} or \cite{GuLe}.
The proof of \Cref{lem:invar} then shows that
\begin{equation}\label{eq:twisted}
\cL^{(\omega)}_t f= \cL_t(e^{2\pi \imath F_{t,\omega}}f), \qquad \text{for any} \qquad f\in \cC^{r-1}(M,0).
\end{equation}


\subsection{Anisotropic Banach spaces} It is now well-established (see e.g. the pioneering work \cite{BlKeLi}) that $\cC^{r-1}(M)$ isn't suitable for analyzing the spectral properties of operators resembling $ \cL_t$. Considering the presence of the stable direction, we need to identify a Banach space embedded within the space of distributions. Furthermore, complications arise due to the flow direction, introducing partial hyperbolicity and further intricacies into the analysis. One approach to address this challenge is based on \cite{Liv} and involves investigating the spectrum of the semigroup generator. Several possibilities exist for the "appropriate" anisotropic Banach space, depending on the specific characteristics of the system. However, for our present purpose, the following simplified version of spaces and norms proposed in \cite{GoLi} will suffice\footnote{Note, however, that a more meticulous selection of the norm could allow us to shrink the essential spectrum as much as we want with much better outcomes in terms of optimality of the so called {\textit{correlation spectra} (see e.g. \cite{BuLi})}. However, this is not always possible and it depends on the dynamics, as shown in \cite{BuCaCa, BuCaJa, BaCa}.}.\\

Let us take $\rho>0$ to be chosen arbitrarily but fixed, and let $\cI_\rho$ be the set of segments of length $\rho$ in the direction of the vector field $U$. Let $p\in \N \cup \{0\}$ and $q\in \R^+$. We denote by $\cC^{q}_c(I)$ the set of complex-valued functions with compact support on $I \subset \cI_\rho$, which are $\lfloor q\rfloor$-times continuously differentiable and whose $\lfloor q\rfloor$-th derivative is H\"older continuous of exponent $q- \lfloor q\rfloor$, if $q$ is not an integer.
We endow $\cC^{q}$ with a norm $\|\cdot \|_{\cC^q}$ such that it is a Banach algebra, namely $\|fg\|_{\cC^q}\le \|f\|_{\cC^q}\|g\|_{\cC^q}$.

We will take the parameters such that $p \le 1$ and $q \le 1+\alpha$, where $\alpha \in(1,0)$ is arbitrary but fixed. In particular, $p+q \le r$. Finally, given a vector field $v$ and a function $f$, as before we denote by $v f$ the Lie derivative of $f$ along $v$ and, for $j\in \N\cup \{0\}$, by $v^j=\prod_{k=0}^j v$ the composition of vector fields and $v^0=\operatorname{Id}$. 
We can now define our norms and spaces: for $f\in \cC^\infty(M), p \le 1$ and $0 < q \leq 1+\alpha$, we set
\begin{equation}\label{eq:norm}
\|f\|_{p,q}=\sup_{j\le p}\sup_{I\in \cI_\rho} \sup_{v\in \{X,V\}}\sup_{\substack{\varphi \in \cC_c^{p+q}(I) \\ \|\varphi\|_{\cC^{p+q}}\le 1}}\left|\int \varphi \cdot   v^j f  \diff U\right|,
\end{equation}
where the integration with respect to $U$ means $
\int_{I_x} f \diff U= \int_{0}^\rho f\circ \horo_s(x) \diff s,
$ for any horocycle arc $I_x$ starting at $x$.
The Banach spaces we will be working with are then defined as $\cB_{p,q}= \overline{\cC^{\infty}(M, \C)}^{\|\cdot\|_{p,q}}$.
The following useful results can be inherited from \cite[Lemma 2.1, Proposition 4.1]{GoLi}.
\begin{lemma}\label{lem:useful}
If $p+q \le r$, the unit ball of $\cB_{p,q}$ is relatively compact in $\cB_{p-1,q+1}$. Moreover, $\cC^r$ is continuously embedded into $\cB_{p,q}$ as a dense subset and, letting $\cD_r$ be the space of distributions of order $r$, the embedding $\mathscr E: \cC^r\to \cD_r$ given by $\langle \mathscr E f, g \rangle=\int fg \diff \vol $ extends to a continuous injection from $\cB_{p,q}$ to $\cD_q$.
\end{lemma}
As we fix the parameters $p,q,\alpha$, we will use the symbols $\cB$ and $\cB_w$ to denote $\cB_{1,\alpha}$ (the strong space) and  $\cB_{0,1+\alpha}$ (the weak space) respectively, with $\|\cdot\|_{\cB}$ and $ \|\cdot\|_{\cB_w}$ the relative norms, when it does not create any confusion.

The proof of \Cref{lem:useful} uses the following fact, that is an easy consequence of the definition of the norms.
\begin{lemma}\label{lem:integral_is_cont_funct}
	Let $\psi \colon [0,1] \to \C$ be of class $\mathscr{C}^{1+\alpha}$ with compact support in $(0,1)$. Then, for any fixed $x \in M$, the functional 
	\[
	\mathcal{I}[x,\psi] \colon f \mapsto \int_0^1 f \circ \horo_s(x) \cdot \psi(s) \diff  s,
	\]
	\sloppy defined for $f \in \mathscr{C}^r(M)$, extends to a continuous linear functional on $\cB$ and on $\cB_w$ with norms $\| \mathcal{I}[x,\psi] \|_{\cB \to \C} \leq \|\psi\|_{\mathscr{C}^{1}}$ and $\| \mathcal{I}[x,\psi] \|_{\cB_w \to \C} \leq \|\psi\|_{\mathscr{C}^{1+\alpha}}$ respectively.
\end{lemma}


\subsection{Lasota-Yorke inequality for $\cL^{(\omega)}_t$} The following Lemma is the key to prove the spectral result for the resolvent given in the next section. We refer to Remark \ref{rmk:constant} to recall the use of the symbols $\sC$ and $C_{\omega}$. 

\begin{lemma}\label{lem:LY-L}
For each $t\in \R^+$, $\omega \in \T^d$ and $f\in \cC^r(M)$, we have 
\begin{equation}\label{eq:LY1-L}
\|\cL^{(\omega)}_t f\|_{0,1+\alpha} \le C_\omega e^{h_{\topp}t}\|f\|_{0,1+\alpha}
\end{equation}
and
\begin{equation}\label{eq:LY2-L}
\|\cL^{(\omega)}_t f\|_{1,\alpha} \le  C_\omega e^{h_{\topp}t}(e^{-\lambda t} \|f\|_{1,\alpha} +\|f\|_{0,1+\alpha}+\|Xf\|_{0,1+\alpha}),
\end{equation}
where $\lambda>0$ is given in \eqref{eq:hyperb}.
\end{lemma}
\begin{proof} Recalling \eqref{eq:twisted} and \eqref{eq:F}, it is convenient to set 
\begin{equation}\label{eq:G}
G_{t,\omega}(x)=\exp\left( 2\pi \imath \int_0^t \langle \omega, X\rangle\circ \geo_{-a}\diff a  \right),
\end{equation}
so that 
\begin{equation}\label{eq:TO}
\cL^{(\omega)}_t f(x)= G_{t,\omega}(x)J_{-t}(x)f\circ \geo_{-t}(x).
\end{equation}
Let us fix a point $x\in M$, and let $I\in \cI_\rho$ be the segment $\{\horo_s(x) : s \in (0,\rho)\}$. Given $\varphi\in \cC_c^{1+\alpha}(I)$ and $f\in \cC^r(M)$, we need to estimate the absolute value of
\[
\int_I \varphi \cL^{(\omega)}_t f \diff  U =  \int_0^\rho \varphi\circ \horo_s(x)   G_{t,\omega} \circ \horo_s(x) J_{-t}\circ \horo_s(x)f\circ \geo_{-t}\circ \horo_s(x)\diff s.
\]
By Lemma \ref{lem:renormalization} and changing variables (recalling the definition of $J_t$ in Lemma \ref{lemma:Jacobi}) we have

\begin{equation}\label{eq:int1}
\begin{split}
\int_I \varphi \cL^{(\omega)}_t f \diff  U &= \int_0^\rho (\varphi G_{t,\omega})\circ \horo_s(x) J_{-t}\circ \horo_s(x) f\circ \horo_{\tau(s,-t,x)}\circ \geo_{-t}(x)  \diff s\\
&=  \int_0^{\tau(\rho,-t,x)}  (\varphi G_{t,\omega})\circ \geo_t\circ \horo_\eta\circ \geo_{-t}(x) f\circ \horo_\eta \circ \geo_{-t}(x)\diff \eta.
\end{split}
\end{equation}

We now partition the segment $\geo_{-t}(I)$ into the union $\geo_{-t}(I)=\bigcup_{j=1}^{N_t} I_j$, where the $I_j$ are segments of length $\rho$, and we take a smooth partition of unity $\{\psi_j\}_j$ made of functions supported on $I_j$ and such that $\sum_{j=1}^{N_t} \psi_j=1$ on $\geo_{-t}(I)$ and $\|\psi_j\|_{\cC^r(I_j)}\le C_*$, for some constant $C_*>0$ depending on $\rho$ and $r$ which are fixed. Note that the number $N_t$ is proportional (with a constant independent on $t$) to the length of $\geo_{-t}(I_j)$. Crucially, by \Cref{lemma:Marcus} (see \cite[Lemma C.3]{GiLiPo}), this length grows proportionally to $e^{h_{\topp}t}$, whereby there exists $\sC>0$ such that
\begin{equation}\label{eq:N_t}
N_t \le \sC e^{h_{\topp}t}.
\end{equation}
Using the above partition of unity and letting $\eta_j \in \operatorname{supp} \psi_j$, by \eqref{eq:int1} we have
\begin{equation}\label{eq:int2}
\begin{split}
\left|\int_I \varphi \cL^{(\omega)}_t f \diff  U \right| &\le
  \sum_{j=1}^{N_t}\left| \int_{\eta_j}^{\eta_j+\rho}  (\varphi G_{t,\omega})\circ \geo_t\circ \horo_\eta\circ \psi_j(\eta) \cdot f\circ \horo_\eta \circ \psi_j(\eta)\diff \eta \right|\\
 &\le \sC e^{h_{\topp}t}  \sup_{I_j}\left| \int_{\eta_j}^{\eta_j+\rho}  (\varphi G_{t,\omega})\circ \geo_t\circ \horo_\eta\circ \psi_j(\eta) \cdot f\circ \horo_\eta \circ \psi_j(\eta)\diff \eta \right|,
\end{split}
\end{equation}
where we have used that $\horo_\eta\circ \geo_{-t}(x)=\horo_{\tau(s,-t,x)}\circ \geo_{-t}(x)$ for any $s\in [0,\rho]$ in the first step and \eqref{eq:N_t} in the second. Inequality \eqref{eq:LY1-L} follows taking the supremum over $I\subset \cI_\rho$ and applying Proposition \ref{prop:G} in the appendix with $\varsigma=1+\alpha$. Obviously, $C_{1+\alpha}(t)$ in \eqref{eq:C(t)} is bounded by some $C_\omega$ which depends on $\omega$ only through its $\cC^2$ norm .\\


For the $\|\cdot\|_{1,\alpha}$-norm we must estimate, for $\varphi\in \cC^{1+\alpha}$, the absolute value of
\[
\mathfrak{I}_X:=\int_{I} X(\cL^{(\omega)}_tf)\varphi \diff U \qquad \text{and} \qquad \mathfrak{I}_V:= \int_{I} V(\cL^{(\omega)}_tf)\varphi \diff U.
\]
Let us start with  $\mathfrak{I}_X$. Recalling \eqref{eq:TO} and \eqref{eq:div}, we have
\begin{equation}\label{eq:I_X}
\begin{split}
X(\cL^{(\omega)}_tf)&=2\pi \imath (\langle \omega, X \rangle \circ \geo_{-t}-\langle \omega, X \rangle) \cdot G_{t,\omega}\cdot J_{-t}\cdot f\circ \geo_{-t}\\
&+ ({\Phi^-}\circ \geo_{-t} -{\Phi^-})\cdot G_{t,\omega} \cdot J_{-t}f\circ \geo_{-t}\\
&+G_{t,\omega} \cdot J_{-t} \cdot  X(f\circ \geo_{-t}).
\end{split}
\end{equation}
Hence, using the change of variables as in \eqref{eq:int2} and  the partition of unity $\psi_j$ as in \eqref{eq:int1}, we have
\[
\begin{split}
|\mathfrak I_X | \le & Ce^{h_{\topp}t}\sup_{I_j} \bigg|  \int_{\eta_j}^{\eta_j+\rho} [f\cdot (\langle \omega, X \rangle- \langle \omega, X \rangle \circ \geo_t )]\circ h_\eta \circ \psi_j \cdot (G_{t,\omega}\varphi)\circ \geo_t\circ h_\eta \circ \psi_j\diff \eta   \\
&+ \int_{\eta_j}^{\eta_j+\rho} [f\cdot ({\Phi^-} -{\Phi^-} \circ \geo_t)]\circ h_\eta \circ \psi_j \cdot (G_{t,\omega}\varphi)\circ \geo_t\circ h_\eta \circ \psi_j\diff \eta\\
&+\int_{\eta_j}^{\eta_j+\rho}  Xf \circ h_\eta \circ \psi_j \cdot (G_{t,\omega}\varphi)\circ \geo_t\circ h_\eta \circ \psi_j\diff \eta \bigg|.
\end{split}
\]
To estimate the three terms above it is sufficient to note that
\[
\begin{split}
\|(\langle \omega, X \rangle- \langle \omega, X \rangle \circ \geo_t )\circ h_\eta \circ \psi_j\|_{\cC^{\alpha}}\le C_*\|\langle \omega, X \rangle\|_{\cC^1}
\end{split}
\]
and
\[
\|({\Phi^-}-{\Phi^-} \circ \geo_t)\circ h_\eta \circ \psi_j\|_{\cC^\alpha} \le C_*\|{\Phi^-}\|_{\cC^1}.
\]
Therefore, by Proposition \ref{prop:G} with $\varsigma=\alpha$ in \eqref{eq:C(t)}, we conclude that
\begin{equation}\label{eq:J_X}
\left|  \int_{I} X(\cL^{(\omega)}_tf)\varphi \diff U \right|\le  e^{h_{\topp}t}C_\omega(\|f\|_{0,\alpha}+\|Xf\|_{0,\alpha}).
\end{equation}

Let us now estimate $\mathfrak I_V$. Recalling \eqref{eq:TO}, \eqref{eq:potential} and Remark \ref{rmk:Ju} we have,

\[
\begin{split}
V(\cL^{(\omega)}_tf)=&2\pi \imath\left(\int_0^t J^u_{-a}\cdot [V(\langle \omega, X \rangle)]\circ \geo_{-a}\diff a-\int_0^t J^u_{-a} (V{\Phi^-})\circ \geo_{-a}\diff a\right)G_{t,\omega}\cdot J_{-t}\cdot f\circ\geo_{-t} \\
&+G_{t,\omega} \cdot J_{-t}\cdot J^u_{-t}\cdot  Vf\circ \geo_{-t}.
\end{split}
\]
Multiplying the above equation by $\varphi\in \cC^{1+\alpha}$, integrating on $I\subset \cI_\rho$ and using the partition of unity $\psi_j$ as in \eqref{eq:int2}, yields
\begin{multline*}
e^{-h_{\topp}t}|\mathfrak I_V | \\
\le \sC\sup_{I_j} \bigg|  \int_{\eta_j}^{\eta_j+\rho} \left[ f\left(\int_0^t J^u_{-a}\circ \geo_t\cdot V(\langle \omega, X \rangle) \circ \geo_{t-a} \diff a\right)\right]\circ \horo_\eta\circ \psi_j\cdot (G_{t,\omega}\varphi)\circ \geo_t\circ \horo_\eta\circ \psi_j\diff \eta \bigg|  \\
+ \sC\sup_{I_j} \bigg|\int_{\eta_j}^{\eta_j+\rho} \left[ f\left(\int_0^t J^u_{-a}\circ \geo_t (V{\Phi^-})\circ \geo_{t-a}) \diff a\right)\right]\circ \horo_\eta\circ \psi_j\cdot (G_{t,\omega}\varphi)\circ \geo_t\circ \horo_\eta\circ \psi_j \diff \eta \bigg| \\
+\sC\sup_{I_j} \bigg| \int_{\eta_j}^{\eta_j+\rho} (Vf)\circ \horo_\eta\circ \psi_j \cdot J^u_{-t}\circ\horo_\eta\circ \psi_j \cdot (G_{t,\omega}\varphi)\circ \geo_t\circ \horo_\eta\circ \psi_j \diff \eta \bigg|.
\end{multline*}
Let us call the three lines above $\mathfrak I^{(1)},\mathfrak I^{(2)}$ and $\mathfrak I^{(3)}$ respectively. Recall that, by Remark \ref{rmk:Ju}, $|J_{-t}^u|\le e^{-\sqrt{\overline k}t}$; then, proving that
\[
\begin{split}
&\left\| \int_0^t J^u_{-a}\circ \geo_t \circ \horo_\eta\circ \psi_j\cdot V(\langle \omega, X \rangle)\circ \geo_{t-a} \circ \horo_\eta\circ \psi_j\diff a \right\|_{\cC^\alpha}  \le \frac{C_*}{(1+\alpha) \sqrt{-\overline k}}\|V(\langle \omega, X \rangle)\|_{\cC^\alpha},\\
&\left\|\left(\int_0^t J^u_{-a}\circ \geo_t \circ \horo_\eta\circ \psi_j\cdot (V{\Phi^-})\circ \geo_{t-a}\circ \horo_\eta\circ \psi_j) \diff a\right)\right\|_{\cC^\alpha}\le \frac{C_*}{(1+\alpha) \sqrt{-\overline k}}\|V({\Phi^-})\|_{\cC^\alpha}
\end{split}
\]
can be done exactly as in the proof of Proposition \ref{prop:G} (using $J^u_t$ and $V$ instead of $J_t$ and $U$) and it is left to the reader. On the other hand, Proposition \ref{prop:G} gives also $\|(G_{t,\omega}\varphi)\circ \geo_t\circ \horo_\eta\circ \psi_j\|_{\cC^\alpha}\le C_\alpha(t).$
It follows that
\[
\mathfrak I^{(1)}+\mathfrak I^{(2)} \le \frac{C C_* C_\alpha(t)}{(1+\alpha) \sqrt{-\overline k}}\max\{\|V(\Phi^+)\|_{\cC^\alpha}, \|V(\langle \omega, X \rangle)\|_{\cC^\alpha}\}\cdot \|f\|_{0,\alpha}\le C_\omega \|f\|_{0,\alpha}.
\]
Finally, let us estimate the $\cC^\alpha$ norm of $J_{-t}\circ \geo_t\circ \horo_\eta \circ \psi_j$ to bound $\mathfrak J^{(3)}$. Recalling \eqref{eq:div} and since $|J^u_{-t}|\le e^{-\sqrt{-\overline k}t}$,
\[
\begin{split}
|(J^u_{-t}\circ \geo_t\circ \horo_\eta \circ \psi_j)' |&=\left| \left(\exp\int_0^{-t} \Phi^+\circ \geo_a\cdot \geo_t\circ \horo_\eta\circ \psi_j \diff a \right)'  \right|\\
&=\left| \int_0^t(\Phi^+\circ g_{t-a}\circ \horo_\eta\circ \psi_j)' \diff \eta  \right| \cdot |J^u_{-t}\circ \geo_t\circ \horo_\eta \circ \psi_j|\\
&\le \frac{C_*}{\sqrt{-\overline k}}\|\Phi^+\|_{\cC^1} e^{-\sqrt{-\overline k}t},
\end{split}
\]
where we have used \eqref{eq:intJ}. Therefore,
\[
\|J^u_{-t}\circ \geo_t\circ \horo_\eta \circ \psi_j\|_{\cC^\alpha}\le \|J^u_{-t}\circ \geo_t\circ \horo_\eta \circ \psi_j\|_{\cC^1}\le \frac{C_*}{\sqrt{-\overline k}}\|\Phi^+\|_{\cC^1} e^{-\sqrt{-\overline k}t}.
\]
Also, Proposition \ref{prop:G} gives $\|(G_{t,\omega}\varphi)\circ \geo_t\circ \horo_\eta\circ \psi_j\|_{\cC^\alpha}\le C_{1+\alpha}(t)$, and we conclude that
\[
\mathfrak J^{(3)}\le  \frac{C_* C_{1+\alpha}(t)}{\sqrt{-\overline k}}\|\Phi^+\|_{\cC^1} e^{-\sqrt{-\overline k}t}\cdot \|f\|_{1,\alpha}\le C_\omega e^{-\sqrt{-\overline k}t}  \|f\|_{1,\alpha}.
\]
We have obtained 
\begin{equation}\label{eq:J_V}
\begin{split}
e^{-h_{\topp}t}\left| \int_{I} V(\cL^{(\omega)}_tf)\varphi \diff U \right| &\le C_{\omega}\big( \|f\|_{0,\alpha}+e^{-\sqrt{-\overline k}t}\|f\|_{1,\alpha}\big).
\end{split}
\end{equation}
By \eqref{eq:J_X} and \eqref{eq:J_V}, we conclude that 
\[
\|\cL^{(\omega)}_t f\|_{1,\alpha} \le C_\omega e^{h_{\topp}t} \big(e^{-\sqrt{-\overline k}t} \|f\|_{1,\alpha} +\|f\|_{0,\alpha}+\|Xf\|_{0,\alpha}\big).
\]
To obtain the estimate with the $\|\cdot\|_{0,1+\alpha}$-norm, we use a trick from \cite{GoLi} involving mollifiers. For each $0<\epsilon<\rho$, let us consider $\varphi_\epsilon$ obtained by convoluting $\varphi \in \cC_c^{\alpha}(I)$ with a mollifier $j_\epsilon$ with support in $[0,\epsilon]$ and $\int j_\epsilon=1$ so that,
\begin{equation}\label{eq:molly}
\|\varphi_\epsilon-\varphi\|_{\cC^{0}}\le \sC\epsilon \|\varphi\|_{\cC^{\alpha}}, \qquad \|\varphi_\epsilon-\varphi \|_{\cC^{\alpha}}\le \sC, \qquad \|\varphi_\epsilon\|_{\cC^{1+\alpha}}\le \sC \epsilon^{-1}.
\end{equation}
Note that the estimates above are straightforward once one observes that, if $j_\epsilon (x)=\epsilon^{-1}j(\epsilon^{-1} x)$ where $j\in \cC^{\infty}, \operatorname{supp}j \subset [-1,1],$ and $ \int j =1$, then
\[
\int j_\epsilon(x-y) \varphi(y)\diff y=\int j_\epsilon (y) \varphi (x-y)\diff y.
\]
For any $v\in \{X,V\}$ we can write
\[
\int_{I} v(\cL^{(\omega)}_tf)\varphi \diff U = \int_{I} v(\cL^{(\omega)}_tf)(\varphi-\varphi_\epsilon) \diff U+\int_{I} v(\cL^{(\omega)}_tf)\varphi_\epsilon \diff U.
\]
Hence the above computations and Proposition \ref{prop:G} with $\varphi_\epsilon$ and $\varphi-\varphi_\epsilon$ as test functions give $A_\omega,B_\omega(\epsilon)>0$, where $B_\omega(\epsilon)$ may depend on $\epsilon$, such that
\[
e^{-h_{\topp}t}\|\cL^{(\omega)}_t f\|_{1,\alpha} \le A_\omega \max\{e^{-\sqrt{-\overline k}t}, \epsilon\} \|f\|_{1,\alpha} +B_\omega(\epsilon)(\|f\|_{0,1+\alpha}+\|Xf\|_{0,1+\alpha}).
\]
Choosing $\epsilon=e^{-\sqrt{-\overline k}t}$ we conclude that there exists $B_\omega(t)>0$ such that
\[
e^{-h_{\topp}t}\|\cL^{(\omega)}_t f\|_{1,\alpha} \le A_\omega e^{-\sqrt{-\overline k}t} \|f\|_{1,\alpha} +B_\omega(t)(\|f\|_{0,1+\alpha}+\|Xf\|_{0,1+\alpha}).
\]
It remains to get rid of the dependence from $t$ of the constant in front of the weak norm. Let us choose $T:=T_\omega >0$ such that $A_\omega e^{-\sqrt{-\overline k}T}\le e^{-\lambda T} <1$ which is possible since $\sqrt{-\overline k}\ge \lambda$, as we remarked before \Cref{cor:inverse_of_Jt}. Hence, by \eqref{eq:LY1-L},
\begin{equation}\label{eq:LYiterate}
\|\cL^{(\omega)}_{t+T} f\|_{1,\alpha} \le  e^{h_{\topp}T}e^{-\lambda T} \|\cL^{(\omega)}_tf\|_{1,\alpha} +e^{h_{\topp}T}B'_\omega(t) (\|f\|_{0,1+\alpha}+\|Xf\|_{0,1+\alpha})
\end{equation}
Writing $t=n_0T+s_0$, where $s_0\in (0,T)$ and iterating \eqref{eq:LYiterate}, we find $C_\omega>0$ such that
\[
\begin{split}
\|\cL^{(\omega)}_t f\|_{1,\alpha} &\le e^{h_{\topp}Tn_0}e^{-\lambda n_0 T}\|\cL_{s_0}^{(\omega)}f\|_{1,\alpha} +\frac{ e^{h_{\topp}t} B''_\omega(T)}{1-e^{-\lambda}}(\|f\|_{0,1+\alpha}+\|Xf\|_{0,1+\alpha})\\
&\le e^{h_{\topp}t}C_\omega ( e^{-\lambda t} \|f\|_{1,\alpha} +\|f\|_{0,1+\alpha}+\|Xf\|_{0,1+\alpha}),
\end{split}
\]
where we have used \eqref{eq:LY1-L} and the fact that, by \eqref{eq:I_X}, $X$ and $\cL^{(\omega)}_t$ commute up to a term $C_\omega\|\cL_t f\|_{0,1+\alpha}$ which can be bounded by \eqref{eq:LY1-L}. This concludes the proof. 
\end{proof}


\subsection{Lasota-Yorke inequality for $\cR_z^{(\omega)}$} While the preceding result is essential for our subsequent estimates, it falls short of providing the desired spectral analysis for the operator $\cL^{(\omega)}_t$. This limitation arises due to the term $\|Xf\|_{0,1+\alpha}$ in \eqref{eq:LY2-L}, responsible for the absence of the compactness ingredient required to apply Hennion's theorem \cite{Hen}. However, a common approach to address this obstacle involves investigating the resolvent of the generator of the semigroup, which will satisfy a {\textit{true}} Lasota-Yorke inequality. This is the goal of this section.
Firstly, we need the following.

\begin{lemma}\label{lem:generator}
For each $\omega \in \bT^d$, the family $\{\cL^{(\omega)}_t\}_{t\ge 0}:\cB \to \cB$ is a strongly continuous semigroup of bounded operators and its generator is the closed operator $Z_\omega:\operatorname{Dom}(Z_\omega)\to \cB$, with $\overline{\operatorname{Dom}(Z_\omega)}=\cB$, given by 
\[
Z_\omega= - X+\Phi^- +2\pi \imath \langle \omega, X \rangle.
\]
\end{lemma}
\begin{proof}
We claim that $\lim_{t\to 0}\|\cL_t^{(\omega)}f-f\|_{\cB}=0$ for each $f\in \cB$ and each $\omega \in \T^d$, which implies strong continuity. 
For any $\epsilon >0$, let  $  f_\epsilon \in \cC^{r}(M)$ such that $\| f_\epsilon- f \|_{\cB}<\epsilon$.
Then, by \Cref{lem:LY-L},
\[
\begin{split}
\|\cL_t^{(\omega)} f- f\|_{\cB} &\le \|\cL_t^{(\omega)}( f- f_\epsilon)\|_{\cB}+\|\cL_t^{(\omega)} f_\epsilon- f_\epsilon \|_{\cB}+\| f- f_\epsilon\|_{\cB} \\
&\le (C_\omega+1) \epsilon + \|\cL_t^{(\omega)} f_\epsilon- f_\epsilon \|_{\cB}.
\end{split}
\]
Therefore, since $\partial_{s} {\cL_s^{(\omega)}}f_\epsilon \in \cC^{r-1}$,
\[
\begin{split}
\|\cL_t^{(\omega)} f_\epsilon- f_\epsilon\|_{\cB} & =\left\|\int_0^t \partial_{s} {\cL_s^{(\omega)}}f_\epsilon {\diff s} \right\|_{\cB} \leq t\sup _{0\le s\le t}\left\|\partial_{s} {\cL_s^{(\omega)}}f_\epsilon \right\|_{\cB} \leq t C(\omega,f_\epsilon),
\end{split}
\]
for some $C(\omega,f_\epsilon)>0$ which depends on $\omega$ and $f_\epsilon$ but not on $t$. This implies the claim.

Let us show the formula for $Z_\omega$. We must compute 
\[
Z_\omega  f=\lim_{t\to 0} \dfrac{\cL_t^{(\omega)} f- f}{t}=(\partial_t\cL_t^{(\omega)}  f)|_{t=0}, \qquad  f\in \cC^{r-1}(M).
\]
We have
\[
\begin{split}
\partial_t\cL_t^{(\omega)}  f =& \partial_t(\cL_t(e^{2\pi \imath F_{t,\omega}} f))=\partial_t(J_{-t}\cdot [e^{2\pi \imath F_{t,\omega}} f]\circ \geo_{-t})\\
=&\partial_t J_{-t}  \cdot (e^{2\pi \imath F_{t,\omega}} f)\circ \geo_{-t}+J_{-t}\cdot  e^{2\pi \imath F_{t,\omega}\circ \geo_{-t}} [2\pi \imath \partial_t(F_{t,\omega}\circ \geo_{-t})  f\circ \geo_{-t}+\partial_t( f\circ \geo_{-t})].
\end{split}
\]
Computing the above in $t=0$ yields
\[
(\partial_t\cL_t^{(\omega)}  f)|_{t=0}=\left(\Phi^- +2\pi \imath (\partial_t(F_{t,\omega}\circ \geo_{-t}))|_{t=0}-X \right) f.
\]
To compute the last term, by \eqref{eq:F} we have, for each $x\in M$,
\[
2\pi \imath \partial_t(F_{t,\omega}\circ \geo_{-t}(x))=2\pi \imath \partial_t \int_{-t}^0 \langle \omega, X \rangle\circ \geo_{s}(x) \diff s= 2\pi \imath \langle \omega, X \rangle\circ \geo_{-t}(x),
\]
whereby $2\pi \imath (\partial_t(F_{t,\omega}\circ \geo_{-t}))|_{t=0}=2\pi \imath \langle \omega, X \rangle$.
\end{proof}

For each $\omega\in \T^d$, let \[\cR^{(\omega)}_z f= (z-Z_\omega)^{-1}f\] be the resolvent of the generator $Z_\omega$, for $z$ not in the spectrum of $Z_\omega$ on $\cB$. 
It is easy to prove, by induction, that, for each $n \ge 1$,
\begin{equation}\label{eq:R-formula}
[\cR_{z}^{(\omega)}]^n f=\frac{1}{(n-1)!}\int_0^\infty t^{n-1}e^{-zt}\cL^{(\omega)}_t f \diff t.
\end{equation}
The following is a {\it true} Lasota-Yorke type inequality for the resolvent. 

\begin{lemma}\label{lem:LY-R}
For each $\omega\in \bT^d$ and $z\in \C$ with $\re(z)>h_{\topp}$, $n\in \N,$ and $f\in \cC^r$, we have
\begin{equation}\label{eq:LY1_R}
\|[\cR_z^{(\omega)}]^n\|_{\cB_w}\le \frac{C_\omega}{(\re(z)-h_{\topp})^{n}}
\end{equation}
and
\begin{equation}\label{eq:LY2_R}
\|[\cR_z^{(\omega)}]^n f\|_{\cB}\le C_\omega \left(\frac{1}{ (\re(z)-h_{\topp}+\lambda)^{n}}\|f\|_{\cB}+\frac{|z|+1}{(\re(z)-h_{\topp})^n}\|f\|_{\cB_w}\right).
\end{equation}
\end{lemma}

\begin{proof}
Since, for each $n \ge 1$,
\begin{equation}\label{eq:Gamma}
\left|\int_0^{\infty} \frac{t^{n-1}}{(n-1) !} e^{-z t} \diff t\right| \leqslant \int_0^{\infty} \frac{t^{n-1}}{(n-1) !} e^{-\re(z) t} \diff t \leqslant \frac{1}{\re(z)^{n}},
\end{equation}
inequality \eqref{eq:LY1_R} is obtained by \eqref{eq:R-formula} and \eqref{eq:LY1-L} :
\[
\|[\cR_z^{(\omega)}]^n f\|_{\cB_w} \leq \frac{C_\omega}{(n-1) !}\|f\|_{\cB_w} \int_0^{\infty}  e^{-(\re(z)-h_{\topp})t} t^{n-1}\diff t=\frac{C_\omega}{(\re(z)-h_{\topp})^n} \|f\|_{\cB_w}.
\]
To prove inequality \eqref{eq:LY2_R} we introduce the truncated resolvent as in \cite{GiLiPo}: for $t_0>0$ and $n\in \N$ let
\[
\overline\cR_z^{(\omega)}(n):=\frac{1}{(n-1) !} \int_{t_0}^{\infty} t^{n-1} e^{-z t} \mathcal{L}_t^{(\omega)}\diff t.
\]
By Lemma \ref{lem:LY-L}, for each $n \ge t_0 e (\re(z)-h_{\topp}+\lambda)$, we have
\[
\begin{split}
\|[\cR_z^{(\omega)}]^n-\overline\cR_z^{(\omega)}(n)\|_{\cB}& \leq C_\omega\int_0^{t_0} \frac{ t^{n-1} e^{(-\re(z)+h_{\topp}) t}}{(n-1) !} \diff t\leq C_\omega \frac{t_0^n}{n !} \leq C_\omega\left(\re(z)-h_{\topp}+\lambda \right)^{-n}.
\end{split}
\]
It is thus enough to estimate, for $\varphi\in \cC^{1+\alpha}$, $\|\varphi\|_{\cC^{1+\alpha}}\le 1$, and $v\in \{X,V\}$, the absolute value of
\begin{equation}\label{eq:intR}
\int_{t_0}^{\infty} \frac{ t^{n-1} e^{-z t}}{(n-1) !}\int_{I} v(\cL^{(\omega)}_tf)\varphi \diff U \diff t.
\end{equation}
If $v=V$ then the computation is exactly as the one done to prove \eqref{eq:J_V}, and \eqref{eq:Gamma} yields 
\[
\begin{split}
\left| \int_{t_0}^{\infty} \frac{ t^{n-1} e^{-z t}}{(n-1) !}\int_{I} v(\cL^{(\omega)}_tf)\varphi \diff U \diff t \right| 
&\le C_\omega \left( \|f\|_{1,\alpha}  \left| \int_{t_0}^{\infty} \frac{ t^{n-1} e^{(-z+h_{\topp} )t} e^{-\lambda t}}{(n-1) !} \diff t \right|+ \|f\|_{0,\alpha}\right)\\
& \le  C_\omega\left( \frac{1}{(\re(z)-h_{\topp}+\lambda)^{n}} \|f\|_{1,\alpha}+ \|f\|_{0,\alpha}\right) .
\end{split}
\]
If $v=X$, we start from \eqref{eq:I_X} from the proof of Lemma \ref{lem:LY-L}. The first two terms are estimated as in the aforementioned proof, so that their contributions in \eqref{eq:intR} are bounded by
\[
\frac{C_\omega}{(\re(z)-h_{\topp})^n}\|f\|_{0,\alpha}.
\]  
It remains to estimate the integral involving the term $G_{t,\omega} \cdot J_{-t} \cdot  X(f\circ \geo_{-t})$, which is the one responsabile of the missing ingredient in the Lasota-Yorke for the semigroup. In the present case, the key observation is $X(f\circ \geo_{-t})= -\frac{\diff }{\diff  t}f\circ \geo_{-t}$ so that, Fubini theorem and integration by parts yield
\[
\begin{split}
&\left|\int_{t_0}^{\infty} \frac{ t^{n-1} e^{-z t}}{(n-1) !}\int_{I} G_{t,\omega} \cdot J_{-t} \cdot  X(f\circ \geo_{-t})\varphi \diff U \diff t\right| \\
&\le \left|\int_{t_0}^{\infty} \int_{I}\left( \frac{ t^{n-2} e^{-z t}}{(n-2) !}-z  \frac{ t^{n-1} e^{-z t}}{(n-1) !}\right) G_{t,\omega} \cdot J_{-t} \cdot \varphi\cdot f\circ \geo_{-t} \diff U \diff t\right|\\
&+\left|\int_{t_0}^{\infty} \int_{I} \frac{ t^{n-1} e^{-z t}}{(n-1) !}\frac{\diff}{\diff t}\left( G_{t,\omega} \cdot J_{-t}\right) \cdot \varphi\cdot f\circ \geo_{-t} \diff U \diff t\right|+ \left| \int_{I} \frac{ t_0^{n-1} e^{-z t_0}}{(n-1) !}\left( G_{t_0,\omega} \cdot J_{-t_0}\right) \cdot \varphi\cdot f\circ \geo_{-t_0} \diff U \right| .
\end{split}
\]
The integrals involving the term $G_{t,\omega} \cdot J_{-t} \cdot \varphi$ in the test function can be treated as in the proof of Lemma \ref{lem:LY-L} (as for the first two terms in \eqref{eq:I_X}) and, using again \eqref{eq:Gamma}, are bounded by
\[
C_\omega\frac{|z|}{(\re(z)-h_{\topp})^n}\|f\|_{0,\alpha}.
\]
Finally, recalling \eqref{eq:G} and \eqref{eq:potential}, we have
\begin{equation}\label{eq:GJ}
\begin{split}
\frac{\diff}{\diff t}G_{t,\omega} &= 2\pi \imath \langle \omega, X \rangle\circ g_{-t} \cdot G_{t,\omega},\\
\frac{\diff}{\diff t}J_{-t} &=\Phi^-\circ g_{-t} \cdot J_{-t},
\end{split}
\end{equation}
which are also of the same kind of the the first two terms in \eqref{eq:I_X}. Therefore, the integral involving the term $\frac{\diff}{\diff t}\left( G_{t,\omega} \cdot J_{-t}\right)$ is bounded by 
\[
C_\omega\frac{1}{(\re(z)-h_{\topp})^n}\|f\|_{0,\alpha}.
\]
Gathering the above estimates and taking the sup over $\varphi\in \cC^{1+\alpha},v\in \{X,V\}, I \subset \cI_\rho$, we obtain 
\[
\|[\cR_z^{(\omega)}]^n f\|_{1,\alpha} \le C_\omega \left(\frac{1}{ (\re(z)-h_{\topp}+\lambda)^{n}}\|f\|_{1,\alpha}+\frac{|z|+1}{(\re(z)-h_{\topp})^n}\|f\|_{0,\alpha}\right).
\]
To obtain the desired estimate for each $n \ge 1$ and for the norm $\|\cdot\|_{0,1+\alpha}$ we use the same mollifier trick used in \eqref{eq:molly} and we proceed analogously by iterating the inequality.
\end{proof}

In the following, we denote by $\cB'_w$ and by $\cB'$ the dual spaces of $\cB_w$ and $\cB$ respectively.

\begin{corollary}\label{prop:spectral}
For each $\omega \in \T^d$ and $z\in \C$ with $\re(z)>h_{\topp}$, we have:
\begin{itemize}
\item[(i)] The spectral radius of $\mathcal{R}^{(\omega)}_z$ on $\cB$ is $|\re(z)-h_{\topp}|^{-1}.$
\item[(ii)] The essential spectral radius of $\mathcal{R}^{(\omega)}_z$ on $\cB$ is bounded by $|\re(z)-h_{\topp}+\lambda |^{-1}$ and the set $\{z\in \mathfrak{sp}(Z_\omega|_{\cB}): \re(z)>h_{\topp}- \lambda\}$ consists of isolated eigenvalues with finite multiplicity
\item[(iii)] $h_{\topp}$ is the only element of the set (peripheral spectrum) $ \{z\in \mathfrak{sp}(Z_0|_{\cB}) \mid \re(z)=h_{\topp}\}$ and it is a simple eigenvalue. 
\item[(iv)] If $\mu$ denotes the unique invariant probability measure for $\horo_t$, then $\mu\in \cB'_w$.
\item[(v)] The function $z\to \cR^{\omega}_z \in \cB$ admits a holomorphic extension to $\{z\in \C : \re(z) > h_{\topp}\}$ and a meromorphic extension to $\{z\in \C : \re(z) > h_{\topp}-\lambda  \}$.

\end{itemize}
\end{corollary}
\begin{proof}
Given Lemma \ref{lem:LY-R}, (i),(ii) and (iii) follow by \cite[Proposition 2.10, Corollary 2.11]{Liv}. Let us prove that $\mu \in \cB'_w$. First we note that, by unique ergodicity (see \cite{Marcus}) we have
\[
\mu(f)=\lim_{T\to \infty} \frac 1T \int f\circ \horo_s(x)\diff s.
\]
Moreover, by  \cite[Lemma 4.6]{AdBa}, for each $f\in \cC^0$, $\mu(f)=\mu(e^{-h_{\topp}t}\cL_tf)$. Next, dividing $[0,T]$ by $\frac{T}{\rho}$ intervals of length smaller than $\rho$, we have by \eqref{eq:LY1-L},

\[
 \left|\int_0^T (\cL_t f)\circ \horo_s(x)\diff s \right|\le {\rho^{-1}}{T}\|\cL_t f\|_{0,1+\alpha}\le C_\omega T e^{h_{\topp}t} \|f\|_{0,1+\alpha}.
 \]
 Therefore, 
\[
|\mu(f)|= |\mu(e^{-h_{\topp}t}\cL_tf)|=\left| e^{-h_{\topp}t} \lim_{T\to \infty} \frac 1T \int_0^T (\cL_t f)\circ \horo_s(x)\diff s \right|\le  C_\omega \|f\|_{0,1+\alpha},
\]
and $\mu \in \cB_w'$ concluding the proof of (iv).

Let us prove (v). The assertion on the holomorphic and meromorphic extensions follow from a straightforward adaptation of \cite[Theorem 2]{But12}. Indeed, the main ingredients are the bounds on the spectral and essential spectral radius of $\cR^{(\omega)}_z$ which, in our case, are given by (i) and (ii). The proof of the very same theorem also provides the following formula for the extension: for each $v \in \C$ such that $\re(v)>(a-h_{\topp})^{-1}$,
\[
\cR^{(\omega)}_{a+ib-v^{-1}}=\cR^{(\omega)}_{a+ib} [\mathbf{1}-v^{-1}\cR^{(\omega)}_{a+ib}]^{-1},
\]
which is meromorphic in the set $\{\re(v) > |a-h_{\topp}+\lambda |^{-1}\}$.
\end{proof}

\section{The spectral picture}\label{sec:spectral_picture}

\subsection{Spectral decomposition} 

We consider now the normalized operator
\[
\hat \cL_t^{(\omega)}:=e^{-h_{\topp}t}\cL_t^{(\omega)}, 
\]
for $t \ge 0$. A simple modification of the proof of Lemma \ref{lem:generator} shows that the family $\{\hat \cL_t^{(\omega)} \ : \ t \geq 0\}$ form a strongly continuous semigroup with generator
\[
\hat Z_\omega= -X+ \Phi^- -h_{\topp} + 2\pi \imath \langle \omega, X \rangle
\] 
and resolvent $\hat\cR^{(\omega)}_z=\cR^{(\omega)}_{z+h_{\topp}}$.

We can now establish the following spectral decomposition for the normalized operators.

 \begin{proposition}\label{prop:pert1}
 	There exists $\delta \in (0,h_{\topp})$ such that, for each $\omega \in \T^d$, there exist a finite set 
 	\[
 	\{z_j{(\omega)} \ :\ j=1, \dots,N_{\omega} \} \subset \{ z \in \C \ : \ -\delta < \re(z) \leq 0 \},
	\] 
	finite rank projectors $\Pi_{j,\omega}:\cB \to \cB$, nilpotent operators $\mathcal{N}_{j,\omega}:\cB \to \cB$, and a family of operators $t \mapsto Q_{\omega,t}$ satisfying $\Pi_{j,\omega} Q_{\omega,t}=Q_{\omega,t}\Pi_{j,\omega}=0$, $\Pi_{j,\omega}\Pi_{k,\omega} = \delta_{j \, k}$, and $\Pi_{j,\omega}\mathcal{N}_{j,\omega} = \mathcal{N}_{j,\omega} \Pi_{j,\omega}= \mathcal{N}_{j,\omega}$, such that
  \[
  \hat\cL^{(\omega)}_t= \sum_{j=1}^{N_{\omega}} e^{t \, z_j{(\omega)} } e^{t \, \mathcal{N}_{j,\omega}}\Pi_{j,\omega}+Q_{\omega,t}.
  \]
  Moreover, $Q_{\omega,t}$ satisfies the following: for each $p<\delta$, there exists $C_{p,\omega}>0$ such that, for each $f\in \operatorname{Dom}({\hat Z_{\omega}})$ and $t \ge 0$, 
\begin{equation}\label{eq:decay}
\|Q_{\omega, t}f\|_{\cB_w} \le C_{p,\omega} \, e^{-p t} \|{\hat Z_{\omega}} f\|_{\cB}.
\end{equation}
 \end{proposition}

One key ingredient to prove the result is the following \emph{Dolgopyat's Inequlity}. 

\begin{proposition}[Dolgopyat's Inequlity]\label{prop:dolgopyat_inequality}
	There exist $\beta,\nu, c >0$ and $C_\omega >0$ such that, for each $z=a+ib$ with  $|b|>\beta$ and $a>0$, 
	\begin{equation}\label{eq:dolg}
		\|[\cR^{(\omega)}_{a+ib+h_{\topp}}]^{\tilde n} \|_{\cB} \le C_\omega (a+\nu)^{-\tilde n}, \qquad \text {where} \qquad \tilde n= \lfloor \gamma \log |b| \rfloor
	\end{equation}
	for some $\gamma \in (0, c/\log(1+\nu/a))$.
\end{proposition}

The proof of \Cref{prop:dolgopyat_inequality} is postponed to \S\ref{sec:dolgopyat_inequality}. We now use it to prove \Cref{prop:pert1}.

\begin{proof}[Proof of \Cref{prop:pert1}]
The statement is a consequence of \cite[Theorem 1]{But} applied to the resolvent $\mathscr{R}_{z+h_{\topp}}^{(\omega)}$, once we verify the corresponding assumptions.
More precisely, we replace Assumption (A3) in Butterley's Theorem with the slightly weaker Dolgopyat's Inequality in \Cref{prop:dolgopyat_inequality}: in doing so, the information we obtain on the resonances is limited to a possibly smaller strip $\{ z \ : \ -\delta < \re(z) \leq 0 \}$ (rather than $\{ z \ : \ -\lambda < \re(z) \leq 0 \}$), which however is sufficient for our purposes.

The remaining two assumptions to verify are
\begin{itemize}
\item[(A1)] There exists $C_\omega>0$ such that $t^{-1}\|\hat \cL_t^{(\omega)}- \operatorname{Id}\|_{\cB \to \cB_w}\le C_\omega$ for all $t \ge 0$,
{
where
\[
\|\hat \cL_t^{(\omega)}- \operatorname{Id}\|_{\cB \to \cB_w}=\sup_{\substack{f\in \cB \\ \|f\|_{\cB}\le 1} }\|\hat \cL_t^{(\omega)} f-f\|_{\cB_w}.
\]
}
\item[(A2)] {  The essential spectral radius of $\hat\cR_{z}^{(\omega)}: \cB\to \cB$ is not greater than $(\Re(z) + \lambda)^{-1}$ for all $\Re(z) >0$.}\\
\end{itemize}
Condition (A2) follows  from Corollary \ref{prop:spectral} applied to $\mathcal{R}^{(\omega)}_{z+h_{\topp}}$. 
It remains to prove (A1). 

Let $f\in \cB$ with $\|f\|_{\cB}\le 1$, $\varphi\in \cC^{1+\alpha}, \|\varphi\|_{\cC^{1+\alpha}}\le 1$ and $I\in \cI_\rho$, then
\[
\begin{split}
\int_0^\rho (\hat\cL^{(\omega)}_t f-f)\varphi \diff U=&\int_0^\rho \int_0^t \partial_s(\hat\cL^{(\omega)}_s f)\varphi \diff U\diff s\\
=&e^{-h_{\topp}t} \int_0^\rho \int_0^t \partial_s(G_{s,\omega}J_{-s}f\circ \geo_{-s})\varphi \diff U \diff s\\
=& e^{-h_{\topp}t}\int_0^\rho \int_0^t \partial_s(G_{s,\omega}J_{-s})f\circ \geo_{-s}\varphi \diff U \diff s\\
&+ e^{-h_{\topp}t}  \int_0^\rho \int_0^t (G_{s,\omega}J_{-s}) \partial_s(f\circ \geo_{-s})\varphi\diff U \diff s.
\end{split}
\]
The integral in the penultimate line above is bounded by $C_\omega e^{h_{\topp}t } \|f\|_{0,1+\alpha}$ (indeed, recall \eqref{eq:GJ} and the computation in the proof of Lemma \ref{lem:LY-L}). For the integral in the last line  we can argue as in \eqref{eq:int1} and \eqref{eq:int2} to obtain 
\[
\begin{split}
&\int_0^\rho \int_0^t  \partial_s(f\circ \geo_{-t}) G_{s,\omega}J_{-s}\diff U \diff s \\
& \quad =\int_0^t \int_0^\rho  \partial_u(f\circ \geo_{-u})_{|_{u=0}}\circ g_{-s} G_{s,\omega}J_{-s}\diff U \diff s\\
&\quad\le \sC e^{h_{\topp}t}\int_0^t  \sup_{I_j}\left| \int_{\eta_j}^{\eta_j+\rho}  (\varphi G_{s,\omega})\circ \geo_s\circ \horo_\eta\circ \psi_j(\eta) \cdot \partial_u(f\circ \geo_{-u})_{|_{u=0}}\circ \horo_\eta \circ \psi_j(\eta)\diff \eta \right|\diff s\\
&\quad\le  C_\omega t e^{h_{\topp}t} \|f\|_{1,\alpha},
\end{split}
\]
since $f\in \cB$. Taking the sup over $\varphi \in \cC^{1+\alpha}, \|\varphi\|_{\cC^{1+\alpha}}\le 1, I\in \cI_\rho$, we have thus proved that
\[
\|\hat\cL^{(\omega)}_t f-f\|_{\cB_w} \le C_\omega t \|f\|_{\cB}
\]
which proves (A1) and completes the proof.
\end{proof}

\subsection{Dolgopyat's inequality}\label{sec:dolgopyat_inequality}

In this section we are proving \Cref{prop:dolgopyat_inequality}. Specifically, we are going to reduce the problem to the analogous estimate for the untwisted transfer operator and refer to Liverani's paper \cite{Liv} (actually, to \cite[Chapter 5]{DeKiLi}, which contains an expanded argument) for the conclusion of the proof.

We need to prove that  there exist $\beta >0$ (which will be chosen large), $\nu >0$ and $C_\omega, c >0$ such that, for each $z=a+ib$ with  $|b|>\beta$ and $a>0$, 
\begin{equation}\label{eq:dolg}
\|[\cR^{(\omega)}_{a+ib+h_{\topp}}]^{\tilde n} \|_{\cB} \le C_\omega (a+\nu)^{-\tilde n}, \qquad \text {where} \qquad \tilde n= \lfloor \gamma \log |b| \rfloor
\end{equation}
for some $\gamma \in (0, c/\log(1+\lambda/a))$.

In order to exploit inequality \eqref{eq:LY2_R} it is convenient to introduce the weighted norm on $\cB$,
\[
\|f\|_{\cB}^{\dagger}:=\max \{\|f\|_{\cB_w}, |z|^{-1}\|f\|_{\cB}\}.
\]
Since $|z|>1$ (taking $\beta>1$), we have $\|f\|_{\cB}^{\dagger}\le \|f\|_{\cB}\le |z|\|f\|^\dagger_{\cB}$, the norms $\|f\|_{\cB}^{\dagger}$ and $\|f\|_{\cB}$ are equivalent, and it is sufficient to prove \eqref{eq:dolg} for $\|\cdot\|^\dagger$. 
Moreover, if $\|\cR^{(\omega)}_{a+ib+h_{\topp}} f \|_{\cB}^\dagger =|z|^{-1}\|\cR^{(\omega)}_{a+ib+h_{\topp}} f\|_{\cB}$,  by \eqref{eq:LY2_R}
\[
\begin{split}
a^n\|[\cR^{(\omega)}_{a+ib+h_{\topp}}]^n f \|_{\cB}^\dagger &\le \frac{C_\omega}{|z|} \left(\frac{a}{a+\lambda}\right)^n \|f\|_{\cB}+ C_\omega (1+|z|^{-1})\|f\|_{\cB_w}\\
&=C_\omega \left({1+\lambda/a}\right)^{-n}  \|f\|_{\cB}^{\dagger}+C_\omega \|f\|_{\cB_w}.
\end{split}
\]
Therefore, choosing $n_0$ and $\delta \in (0,1)$ such that $C_\omega(1+\lambda/a)^{-n}< \delta^n$, for each $n \ge n_0$, we have
\begin{equation}\label{eq:LYdag}
\|[\cR^{(\omega)}_{a+ib+h_{\topp}}]^n f \|_{\cB}^\dagger \le (\delta a^{-1})^{n} \|f \|_{\cB}^\dagger+C_\omega a^{-n}\|f\|_{\cB_w}, \qquad \forall n \ge n_0.
\end{equation}
In order to prove \eqref{eq:dolg}, writing $\tilde n = 2n$ and using \eqref{eq:LYdag}, it is sufficient to show that there exists $\tilde \nu >0$ such that
\begin{equation}\label{eq:wdolg}
\|[\cR^{(\omega)}_{a+ib+h_{\topp}}]^{\tilde n} f \|_{\cB_w} \le C_\omega (a+\tilde \nu)^{-\tilde n} \|f\|^\dagger_{\cB} \quad \forall f \in \cB.
\end{equation}
Thus, for $f\in \cC^r, \varphi \in \cC^{1+\alpha}(I), \|\varphi\|_{\cC^{1+\alpha}}\le 1$, we want to bound
\begin{equation}\label{eq:integral1}
\left| \int_I \varphi[\cR^{(\omega)}_{z+h_{\topp}}]^n f \diff U\right|.
\end{equation}

 The rough idea to estimate the above integral is the following: thanks to the computations provided in Appendix \ref{sec:app-G}, which offer an estimate of the ${\cC}^{1+\alpha}$-norm of the weight $G_{t,\omega}$ along specific directions, we are able to partition our integral into horocycle segments. This partitioning allows us to apply here the same argument outlined in \cite[Section 5.7]{DeKiLi} to prove the analogous estimate \eqref{eq:wdolg} for the classical transfer operator. As pointed out in the aforementioned book, it is enough to get an estimate in term of the $\cC^1$ norm of $f$ and then use the mollifiers in \cite[Lemma 5.3]{DeKiLi} to get the desired bound. The critical factors enabling this are the renormalization provided by Lemma \ref{lem:renormalization} and the one-dimensionality of the stable and unstable manifolds. 
 
 For the beginning of the proof we follow closely  \cite[Section 5.7]{DeKiLi}.
 
\textit{Step I (Localizing in time)}. Let $t_*>0$ small (which corresponds to $\tau$ in \cite{DeKiLi}) to be chosen later depending on $|b|^{-1}$. Let $\tilde p:\R \to \R$ be an even function supported on $(-1,1)$ with a single maximum at $0$, satisfying $\sum_{\ell \in \Z} \tilde p(t-\ell)=1$ for any $t \in \R$. Define $p(s)=\tilde p(s/t_*)$ and write 
\begin{multline*}
 [\cR_{z+h_{\topp}}^{(\omega)}]^n f=\int_0^{\infty}  \frac{t^{n-1}}{(n-1) !} e^{-z t} \hat \cL^{(\omega)}_t f \diff t= \int_0^{t_*} p(s) \frac{s^{n-1}}{(n-1) !} e^{-z s}\hat\cL^{(\omega)}_s f \diff s \\
 + \sum_{\ell \in \mathbb{N}_0} \int_{-{t_*}}^{t_*} p(s) \frac{(s+\ell {t_*})^{n-1}}{(n-1) !} e^{-z(s+\ell {t_*})} \hat\cL^{(\omega)}_{\ell {t_*}+s}f \diff s.
\end{multline*}
It is convenient to introduce the notation
\[
\begin{split}
&p_{n, \ell, z}(s):=p(s) \frac{(s+\ell {t_*})^{n-1}}{(n-1) !} e^{-z(s+\ell {t_*})}, \text { for } \ell \geqslant 1, \\
&p_{n, 0, z}(s):=p(s) \frac{s^{n-1}}{(n-1) !} e^{-z s} {\chi}_{s \geqslant 0}
\end{split}
\]
where $\chi_{A}$ is the indicator function of the set $A$. Therefore, the integral \eqref{eq:integral1} becomes
\begin{equation}\label{eq:integral2}
\int_I \varphi[\cR^{(\omega)}_{z+h_{\topp}}]^n f \diff U=  \sum_{\ell \in \mathbb{N}} \int_{-{t_*}}^{t_*} p_{n, \ell, z}(s) \int_I \varphi \hat\cL^{(\omega)}_{\ell {t_*}+s}f \diff U \diff s.
\end{equation}

\textit{Step II (Chopping the spatial integral)}. Note that, by \eqref{eq:TO},
\[
\hat\cL^{(\omega)}_{\ell {t_*}+s}f= e^{-h_{\topp}(\ell {t_*}+s)} \cdot G_{\ell {t_*}+s, \omega} \cdot J_{-(\ell {t_*}+s)} \cdot L^{(0)}_{\ell t_*}(L^{(0)}_{s}f),
\]
where $L^{(0)}_\xi f:=f\circ \geo_{-\xi}$.
Therefore, for each $\ell\in \N$ and $s\in (-t_*,t_*)$,
\begin{equation}\label{eq:integral5}
\begin{split}
\int_I \varphi \hat\cL^{(\omega)}_{\ell {t_*}+s}f \diff U&=
e^{-h_{\topp}(\ell {t_*}+s)} \int_0^{\rho} (\varphi\cdot  G_{\ell {t_*}+s, \omega}) \circ \horo_u \cdot J_{-(\ell {t_*}+s)} \circ \horo_u  \cdot (L^{(0)}_{s}f) \circ \geo_{\ell t_*}\circ \horo_u\diff u.
\end{split}
\end{equation}
We observe that, directly from its definition \eqref{eq:potential}, we can write
\[
J_{-(\ell {t_*}+s)} = J_{-\ell {t_*}} \cdot J_{-s} \circ \geo_{-\ell {t_*}};
\]
in particular, Lemma \ref{lem:renormalization} and the usual change of variables $\eta = \tau(u, -\ell t_*, x)$ give us
\[
\begin{split}
	\int_I \varphi \hat\cL^{(\omega)}_{\ell {t_*}+s}f \diff U&=
	e^{-h_{\topp}(\ell {t_*}+s)} \int_0^{\tau(\rho, -\ell t_*, x)} (\varphi\cdot  G_{\ell {t_*}+s, \omega}) \circ \geo_{\ell t_*} \circ \horo_{\eta} \cdot J_{-s} \circ \horo_{\eta}  \cdot (L^{(0)}_{s}f) \circ \horo_{\eta}\diff \eta.
\end{split}
\]
Introducing the partition of unity $\{\psi_j\}_j$ supported on $g_{-\ell t_*}(I)$ used in \eqref{eq:int2}, we get
\[
\int_I \varphi \hat\cL^{(\omega)}_{\ell {t_*}+s}f \diff U =
e^{-h_{\topp}(\ell {t_*}+s)}  \sum_{j} \int_{I_j} [(\varphi\cdot  G_{\ell {t_*}+s, \omega}) \circ \geo_{\ell t_*} \cdot J_{-s}]\circ \horo_{\eta} \circ \psi_j(\eta) \cdot (L^{(0)}_{s}f) \circ \horo_{\eta}\circ \psi_j(\eta) \diff \eta.
\]



\textit{Step III (Localizing in space)}. Letting $\epsilon>0$ small  (which corresponds to $r$ in \cite{DeKiLi}) to be chosen later depending on $|b|^{-1}$, we choose a partition of unity $\{\phi_{i,\eps}\}_i$ subordinated to a covering by balls of radius $\eps$, with uniformly bounded multiplicity, so that
\[
\supp(\phi_{i,\eps}) \subset B(x_i,\eps),\qquad
\|\phi_{i,\eps}\|_{\mathscr C^0}\le C_\sharp,\qquad
\|\nabla \phi_{i,\eps}\|_{\mathscr C^0}\le C_\sharp \eps^{-1},
\]
and the number of such balls intersecting a fixed compact set is bounded by $C_\sharp \eps^{-3}$.
Accordingly, in the above integral we may further decompose each interval $I_j$ into pieces
$I_{i,j,\eps}$ of length comparable with $\eps$, and write
\begin{multline}\label{eq:integral3}
	\int_I \varphi \hat\cL^{(\omega)}_{\ell {t_*}+s}f \diff U =
	e^{-h_{\topp}(\ell {t_*}+s)} \\ \sum_{i} \sum_{j} \int_{I_{i,j,\eps}} \phi_{i,\eps} \, [(\varphi\cdot  G_{\ell {t_*}+s, \omega}) \circ \geo_{\ell t_*} \cdot J_{-s}]\circ \horo_{\eta} \circ \psi_j(\eta) \cdot (L^{(0)}_{s}f) \circ \horo_{\eta}\circ \psi_j(\eta) \diff \eta.
\end{multline}


\textit{Step IV (Reducing to the classical transfer operator)}. 
We know we can bound $\|J_{-s} \|_{\mathscr{C}^0} \leq C_\sharp e^{\lambda s}$ and $\|J_{-s} \circ \horo_{(\cdot)}\|_{\mathscr{C}^1} \leq C_\sharp e^{2\lambda s}$. Similarly, using \Cref{prop:G}, we can also bound
\[
\|G_{\ell {t_*}+s, \omega} \circ \geo_{\ell t_*}  \circ \horo_{(\cdot)}\|_{\mathscr{C}^1} \leq \|G_{\ell {t_*}+s, \omega} \circ \geo_{\ell t_*+s}  \circ \horo_{(\cdot)}\|_{\mathscr{C}^1} \cdot \|\tau(\cdot, -s,x)\|_{\mathscr{C}^1} \leq C_\sharp e^{\lambda s};
\]
so that
\[
\|(\varphi\cdot  G_{\ell {t_*}+s, \omega}) \circ \geo_{\ell t_*}\circ \horo_{(\cdot)}\|_{\mathscr{C}^{1}} \leq C_\sharp \|\varphi \circ \geo_{\ell t_*} \circ \horo_{(\cdot)}\|_{\mathscr{C}^{\alpha}}e^{\lambda s} \leq C_\sharp \|\varphi \circ \horo_{(\cdot)}\|_{\mathscr{C}^{1}}e^{\lambda s} \leq C_\sharp e^{\lambda s}.
\]
Let $\overline J_{-s} $ and $\overline \varphi_{t,\omega}$ be the average value of, respectively, $J_{-s}$ and $(\varphi\cdot  G_{\ell {t_*}+s, \omega}) \circ \geo_{\ell t_*}$ over $I_{i,j,\eps}$. The previous computations tell us that 
\[
\|[(\varphi\cdot  G_{\ell {t_*}+s, \omega}) \circ \geo_{\ell t_*} \cdot J_{-s}]\circ \horo_{(\cdot)} - {\overline \varphi_{t,\omega}} \, {\overline J_{-s}}\|_{\mathscr{C}^0} \leq C_\sharp \varepsilon e^{2\lambda s}.
\]

Since the interval $I_{i,j,\eps}$ has length comparable with $\eps$, the above $\mathscr C^0$ estimate implies that, on each such interval, replacing
\[
[(\varphi\cdot  G_{\ell {t_*}+s, \omega}) \circ \geo_{\ell t_*} \cdot J_{-s}]\circ \horo_{\eta}
\]
by its average value ${\overline \varphi_{t,\omega}}\,{\overline J_{-s}}$ produces an error bounded by
\[
C_\sharp \eps e^{2\lambda s}\|f\|_\infty.
\]
Since $|s|\le t_*$ and $t_*$ will be chosen small, we can absorb the factor $e^{2\lambda s}$ into a uniform constant (or, more precisely, write it as $1+\mathcal O(\lambda t_*)$). Therefore
\[
\int_{I_{i,j,\eps}} \phi_{i,\eps}\,
[(\varphi\cdot  G_{\ell {t_*}+s, \omega}) \circ \geo_{\ell t_*} \cdot J_{-s}]\circ \horo_{\eta} \circ \psi_j(\eta)\,
(L^{(0)}_{s}f) \circ \horo_{\eta}\circ \psi_j(\eta)\diff \eta
\]
equals
\[
{\overline \varphi_{t,\omega}}\,{\overline J_{-s}}
\int_{I_{i,j,\eps}} \phi_{i,\eps}\,
(L^{(0)}_{s}f) \circ \horo_{\eta}\circ \psi_j(\eta)\diff \eta
+\mathcal O_\omega\!\left(\eps (1+\lambda t_*)\|f\|_\infty\right).
\]

Furthermore, since $\| \varphi\cdot  G_{\ell {t_*}+s, \omega} \|_{\mathscr{C}^0} \leq 1$, from \eqref{eq:integral3}, we obtain
\begin{multline}\label{eq:integral4}
\int_I \varphi \hat\cL^{(\omega)}_{\ell {t_*}+s}f \diff U =
e^{-h_{\topp}(\ell {t_*}+s)} \\ \sum_{i} \sum_{j} \left( {\overline J_{-s}} \int_{I_{i,j,\eps}} \phi_{i,\eps} \, \left[ (L^{(0)}_{s}f) \circ \horo_{\eta}\circ \psi_j(\eta)  +\cO_\omega(\eps (1+ {\lambda t_*})\|f\|_{\infty}) \right] \diff \eta \right),
\end{multline}
where $\cO_\omega$ may include constants of the type $C_\omega$. 

The number of summands in $j$ is proportional to $e^{h_{\topp} \ell {t_*}}$ (recall \eqref{eq:integral5}), and ${\overline J_{-s}} = \cO(e^{h_{\topp}s})$. 
Gathering all the above into \eqref{eq:integral2}, and using the fact that  
\[
\left|\sum_{\ell \in \mathbb{N}_0} \int_{-{t_*}}^{t_*} p(s) \frac{(s+\ell {t_*})^{n-1}}{(n-1) !} e^{-z(s+\ell {t_*})}\diff s\right|\le a^{-n},
\]
we have obtained
\[
\begin{split}
\int_I \varphi[\cR^{(\omega)}_{z+h_{\topp}}]^n f \diff U = & \sum_{\ell \in \mathbb{N}}\sum_{i}\sum_{j} \int_{-{t_*}}^{t_*} p_{n, \ell, z}(s)e^{-h_{\topp}(\ell t_*+s)} {\overline J_{-s}} \int_{I_{i,j,\eps}}\phi_{\eps,i} (L^{(0)}_{s}f)\circ \horo_\eta\circ \psi_j(\eta) \diff \eta \diff s\\
&+\cO_\omega(\eps {t_*} \, a^{-n}\|f\|_\infty).
\end{split}
\]
We have finally reduced the problem to the setting of \cite[\S5.7]{DeKiLi}: since $L^{(0)}_s$ is the classical transfer operator for the geodesic flow, the above integral can be estimated proceeding exactly\footnote{Note that in the integral \cite[(5.7.8)]{DeKiLi} the presence of the Jacobian $J_{\ell,j,i}$ is replaced here by $e^{-h_{\topp}t} \cdot {\overline J_{-s}}$. As we remarked before, recall the fact that the manifolds $I_j$ grow with a volume (length) proportional to $e^{h_{\topp} \ell t_*}$ and ${\overline J_{-s}} \cdot e^{h_{\topp}s} = \cO(1)$.} as in \cite{DeKiLi}, starting from integral \cite[(5.7.8)]{DeKiLi} and choosing $\eps$ and $t_*$ opportunely small with respect to $|b|^{-1}$, giving the desired estimate.


\begin{remark}
Let us stress why the error term in the above proof is of order $\eps t_*$. In \cite[\S5.7]{DeKiLi}, the corresponding localization error is estimated in terms of the H\"older exponent of the anisotropic norm and is therefore of size $r^\alpha$ (with $0<\alpha<1$). In our setting, thanks to the estimates from Appendix~\ref{sec:app-G} and to the one-dimensionality of the stable/unstable directions, the coefficient
\[
[(\varphi\cdot  G_{\ell {t_*}+s, \omega}) \circ \geo_{\ell t_*} \cdot J_{-s}]\circ \horo_{(\cdot)}
\]
is controlled in $\mathscr C^1$ uniformly for $|s|\le t_*$; hence, since the intervals are of size $\eps$, this yields an error of order $\eps$ rather than $\eps^\alpha$. After inserting this estimate into \eqref{eq:integral2} and integrating over the time window $(-t_*,t_*)$, one gains the additional factor $t_*$, which leads to the remainder
\[
\mathcal O_\omega(\eps t_*\, a^{-n}\|f\|_\infty).
\]
Thus the stronger remainder comes from the extra $\mathscr C^1$ control available in the present setting, not from a modification of the Dolgopyat cancellation argument.
\end{remark}


\subsection{A Guivarc'h-Nagaev approach} In this section we are going to use a spectral argument, originally due to Guivarc'h and Nagaev, to study the analytic properties of the leading eigenvalue of $\cL^{(\omega)}_t$; we refer the reader to  \cite{Gou} for a review of the argument. 
Unfortunately, 
we cannot use quasi-compactness of the semigroup. Nevertheless, the information on the peripheral spectrum of $\hat \cL_t$ and the quasi-compactness of the resolvent will be enough to performe an analogous computation.

We show that, in \Cref{prop:pert1}, the only pole with zero real part occurs at $\omega =0$ and is simple.

\begin{proposition}\label{prop:zero_is_simple}
	Let $\omega \in \T^d$, and let $j \in \{1,\dots, N_{\omega}\}$ be so that $\re \, z_j(\omega) = 0$. Then, $\omega=0$, $j=1$, $z_1(0)=0$, and the associated eigenprojection $\Pi_{1,0}$ has rank 1.

Furthermore, let $\phi_0 \in \Pi_{1,0}(\cB)$ with $\mu(\phi_0) =1$, where  $\mu = \hat \cL_t' \mu$ is the invariant measure for $(\horo_t)_{t\in \R}$. Then,  $\phi_0$ is a measure and the functional $m$ defined by $m(\psi) =\mu(\psi \, \phi_0)$ is the measure of maximal entropy for the geodesic flow.
\end{proposition}
\begin{proof}
	From \Cref{prop:pert1}, we already know that all poles have real part not larger than 0. Assume that $z_j(\omega)$ is such that $\re \, z_j(\omega) = 0$. We first show that the associated nilpotent $\mathcal{N}_j$ is zero.
	
	By contradiction, assume that $\mathcal{N}_j \neq 0$: then, there exists $u \in \cB$ so that $\|e^{t \, \mathcal{N}_j} u \|_{\cB} \geq t^n/n!$ for some $n \geq 1$. Since $\mathscr{C}^r$ is dense in $\cB$ and $\Pi_{j,\omega}$ is continuous, we deduce $ \Pi_{j,\omega}\mathscr{C}^r = \Pi_{j,\omega} \cB$. This implies that $u = \Pi_{j,\omega}f$ for some $f \in \mathscr{C}^r$.
	Then, the Lasota-Yorke bound \eqref{eq:LY2-L} yields
	\[
	\frac{t^n}{n!} \leq \|e^{t \, \mathcal{N}_j} u \|_{\cB} = \| {\hat \cL^{(\omega)}_t} f\|_{\cB} \leq C_{\omega}(\|f\|_{\cB} + \|Xf\|_{\cB_w} );
	\]
	which is the desired contradiction, since the right hand side above is uniformly bounded in $t$.
	Furthermore, as in \cite[Lemma 4.6]{AdBa}, for $\omega = 0$, the pole $z_1(0)=0$ is simple, the associated eigenvector of the dual operator is the invariant measure $\mu = \Pi_{1,0}'(1)$, where $\Pi_{1,0}'$ is the corresponding projection in the dual spectral decomposition, and there are no other poles on the imaginary axis. 
	
	Let $\phi_{\omega}\in \Pi_{j,\omega}(\cB)$ with $\phi_{\omega}\neq 0$. Then, we can write $\phi_{\omega} = \Pi_{j,\omega}f_{\omega}$ for some $f_{\omega} \in \mathscr{C}^r$ and
	\[
	\phi_{\omega} = \lim_{T \to \infty} \frac{1}{T} \int_0^T e^{-t \, z_j(\omega)} {\hat \cL^{(\omega)}_t} f_{\omega} \diff t,
	\]
	where the latter limit holds in $\cB_w$. 
	In the case $\omega = 0$, we can take $f_{\omega}=1$ and, by \Cref{prop:spectral}, since there are no other poles on the imaginary axis,  $\phi_0 = \lim_{t \to \infty} {\hat \cL_t}(1)$.

Recall that, by \Cref{lem:useful}, $\cB_w \subset (\mathscr{C}^r)^{\ast}$. Since, for any $\psi \in \mathscr{C}^r$, we have
	\begin{equation}\label{eq:variat_of_phi_omega}
	|\phi_{\omega}(\psi) |\leq \lim_{T \to \infty}  \frac{1}{T} \int_0^T  \left\lvert \langle {\hat \cL^{(\omega)}_t} f_{\omega}, \psi\rangle  \right\rvert \diff t \leq \|f\|_{\infty}  \lim_{t \to \infty}   \langle {\hat \cL_t} 1, | \psi | \rangle  \diff t \leq \|f\|_{\infty} \cdot \|\psi\|_{\infty},
	\end{equation}
	by \cite[Sec. I4]{Schw}, the distribution $\phi_{\omega}$ is actually a (signed) Radon measure. More precisely, for any $x \in M$, the distribution $\phi_{\omega}$ induces a measure on the horocycle orbit of $x$ defined as follows: let $I$ denote a horocycle arc based at $x$, then
	\[
	\psi \mapsto \mathcal{I}[x,\psi|_I](\phi_{\omega}),
	\]
	where we used the notation of \Cref{lem:integral_is_cont_funct}. Moreover, the inequality \eqref{eq:variat_of_phi_omega} above also shows that the variation $|\phi_{\omega}|$ of $\phi_{\omega}$ is absolutely continuous with respect to $\phi_0$. 

Let $\psi \in \mathscr{C}^r$ be supported in a flowbox $\{\horo_s(B) \ : \ s\in [0,1]\}$, where $B$ is contained in a weak-unstable leaf for the geodesic flow; let us compute $\mu(\psi)$. We have
\[
\begin{split}
\mu(\psi) &= \Pi_{1,0}'(1)(\psi) = \lim_{t \to \infty} [{\hat \cL_t}'(1)](\psi) = \lim_{t \to \infty} \langle 1,{\hat \cL_t} \psi  \rangle =  \lim_{t \to \infty} e^{-h_{\topp}t} \langle J_{-t}\circ \geo_t , \psi  \rangle \\
&= \lim_{t \to \infty} e^{-h_{\topp}t} \int_{B} \int_0^1 (J_{-t}\circ \geo_t \cdot \psi) \circ\horo_s(x) \diff s \diff \nu(x),
\end{split}
\]
where $\nu$ is the disintegration of the measure $\vol$ on the weak-unstable leaves for $(\geo_t)_{t\in \R}$. We define 
\[
\begin{split}
&\mathcal{J}(x,s) = \exp \left(\int_0^{\infty} \Phi^{-} \circ \geo_r \circ \horo_s(x) - \Phi^{-} \circ \geo_r (x) \diff r \right), \qquad \text{and} \\
&E_t(x,s) = 1- \exp \left(\int_t^{\infty} \Phi^{-} \circ \geo_r \circ \horo_s(x) - \Phi^{-} \circ \geo_r (x) \diff r \right).
\end{split}
\]
Note that $J_{-t}\circ \geo_t \circ\horo_s(x) =  J_{-t}\circ \geo_t(x) \mathcal{J}(x,s) + J_{-t}\circ \geo_t \circ\horo_s(x) E_t(x,s)$ and that
$|E_t| \leq C \lambda^{-t}$ by \eqref{eq:hyperb}. 
The latter upper bound implies that 
\[
\lim_{t \to \infty} e^{-h_{\topp}t} \int_{B} \int_0^1 (J_{-t}\circ \geo_t \cdot \psi) \circ\horo_s(x) E_t(x,s) \diff s \diff \nu(x) =   \lim_{t \to \infty} e^{-h_{\topp}t} \langle J_{-t}\circ \geo_t , \psi \, E_t \rangle = \lim_{t \to \infty} [{\hat \cL_t}'(1)](\psi \, E_t) = 0.
\]
Consequently, we can rewrite $\mu(\psi)$ as 
\[
\begin{split}
\mu(\psi) &=\lim_{t \to \infty} e^{-h_{\topp}t}  \int_{B} J_{-t}\circ \geo_t(x) \int_0^1 \psi \circ\horo_s(x) \mathcal{J}(x,s) \diff s \diff \nu(x)= \int_{B} \left( \int_0^1 \psi \circ\horo_s(x) \mathcal{J}(x,s) \diff s \right) \diff \vartheta(x),
\end{split}
\]
where $\diff \vartheta(x)$ is the (weak) limit of $e^{-h_{\topp}t}J_{-t}\circ \geo_t(x)\diff \nu(x)$.
We can similarly define the measure $m$ by 
\[
m(\psi) = \mu( \psi \, \phi_0) = \int_{B} \mathcal{I}[x, \psi \, \mathcal{J}(x,\cdot)](\phi_0) \diff \vartheta(x)
\]
for any $\psi$ as above. In particular, the measures $m$ and $\mu$ have the same disintegration on the weak-unstable manifolds for the geodesic flow. Furthermore, $m$ is invariant for the geodesic flow, since 
\[
m(\psi \circ \geo_{-t}) =  \mu( \psi \circ \geo_{-t} \, \phi_0) = \mu( \psi \circ \geo_{-t} \, {\hat \cL_t}\phi_0)  = \mu( {\hat \cL_t}(\psi \phi_0)) =  {\hat \cL_t}'\mu( \psi \, \phi_0) = m(\psi).
\] 
By \Cref{prop:Margulis}, $m$ is the measure of maximal entropy for $(\geo_t)_{t\in \R}$.

Let us go back to the measures $|\phi_{\omega}|$ and $\phi_0$. Since $ {\hat \cL^{(\omega)}_t} \phi_{\omega} = \phi_{\omega}$, we have that $J_{-t} (\geo_{-t})_{\ast}|\phi_{\omega}| = |\phi_{\omega}|$; similarly,  $J_{-t} (\geo_{-t})_{\ast} \phi_0 = \phi_0$. Let $\ell$ denote the Radon-Nikodym derivative of $|\phi_{\omega}|$ with respect to $\phi_0$. Then, it follows that $\ell \, \phi_0 = J_{-t} (\geo_{-t})_{\ast} (\ell \, \phi_0) = J_{-t} ( \ell\circ \geo_{-t}) (\geo_{-t})_{\ast}  \phi_0 =  ( \ell\circ \geo_{-t})  \phi_0 $. Thus, for any function $\psi$, we deduce that $m( \ell \psi )= m( \ell\circ \geo_{-t} \, \psi)$. This implies that $\ell$ is constant $m$-almost everywhere; up to normalization, we can assume $\ell = 1$. 

Since $|\phi_{\omega}| = \phi_0$, there exists a measurable function $\ell_{\omega}$ such that $\phi_{\omega} = e^{2\pi\imath \ell_\omega} \phi_0$.
We deduce
	\[
		{\hat \cL_t}\phi_0 = \phi_0 =e^{-2\pi\imath \ell_\omega} \phi_{\omega} =e^{-2\pi\imath\ell_\omega-t \, z_j(\omega)}  \hat \cL^{(\omega)}_t \phi_{\omega} = \hat \cL_t( e^{-2\pi\imath\ell_\omega\circ \geo_t-t \, z_j(\omega) + 2\pi\imath F_{\omega,t}} \phi_{\omega}).
	\]
As before, the above equation implies that 
\[
m(\psi) = m ( e^{-2\pi\imath\ell_\omega\circ \geo_t-t \, z_j(\omega) + 2\pi\imath F_{\omega,t} + 2\pi\imath\ell_\omega} \psi)
\]
for any $\psi \in \mathscr{C}^0(M)$. In turn, this yields
	\[
F_{\omega,t}-t \, \frac{z_j(\omega)}{2\pi \imath} =\ell_\omega\circ \geo_t-\ell_\omega
	\]
	$m$-almost everywhere. By ergodicity, it follows that $z_j(\omega)=2\pi \imath \int  \langle \omega, X \rangle \diff m = 0$, see  \cite{KatSun}.
	Furthermore, since $F_{\omega,t}$ is \emph{not} a coboundary whenever $\omega \neq 0$, we obtain the desired contradiction.

\end{proof}

\begin{remark}\label{rk:leafwise_measure}
The proof of \Cref{prop:zero_is_simple} show that, for every $x \in M$, the eigenvector $\phi_0$ induces a measure on the horocycle orbit of $x$ given by $\psi \mapsto \mathcal{I}[x,\psi](\phi_0)$. In the language of \cite[Sect. 9]{GoLi2}, $\mathcal{I}[x,\cdot](\phi_0)$ is a continuous leafwise measure. Since $\phi_0$ is an eigenvector of ${\hat \cL_t}$, the following self-similarity property holds: let $x \in M$, and let $I_{\rho}$ be the horocycle segment at $x$ of length $\rho$. Then, for any $\psi$ supported on $I_\rho$, we have
\[
\begin{split}
 \mathcal{I}[x,\psi](\phi_0) &=  \lim_{t \to \infty}  \mathcal{I}[x,\psi]({\hat \cL_t}(1)) = \lim_{t \to \infty} e^{-h_{\topp} (t+T)} \int_0^{\rho} \psi\circ \horo_s(x) \, \cL_t1 \circ \geo_{-T}\circ \horo_s(x) \, J_{-T} \circ \horo_s(x) \diff s\\
 &=e^{-h_{\topp} T} \lim_{t \to \infty} e^{-h_{\topp}t} \int_0^{\tau(\rho,-T,x)} \psi \circ \geo_{T}\circ \horo_s(\geo_{-T}(x)) \, \cL_t1 \circ \horo_s(\geo_{-T}(x)) \diff s \\
 &= e^{-h_{\topp} T} \mathcal{I}[\geo_{-T}(x),\psi \circ \geo_{T}|_{I_{\tau(\rho,-T,x)}}](\phi_0).
\end{split}
\]
Let us call $\varpi(\rho,x) = \mathcal{I}[x,\one_{I_{\rho}}](\phi_0) $, which is well-defined since $\phi_0$ is a measure. Then, the equation above translates to 
\begin{equation}\label{eq:self_similarity_of_alpha}
 e^{-h_{\topp} T}\varpi(\rho,x) = \varpi(\tau(\rho,T,x),\geo_{T}(x)), \qquad \text{for any $T\in\R$.}
\end{equation}
Equation \eqref{eq:self_similarity_of_alpha} shows that $\varpi(\rho,x)$ defined above is precisely the Margulis parametrization of \Cref{prop:Margulis}.
\end{remark}

We are ready to state the main result of this section.
  
  \begin{proposition}\label{prop:pert2}
  	There exist $\delta > 0$ and a neighborhood $\mathcal N_0$ of $\omega=0$ for which the following properties hold.
  	\begin{itemize}
	\item[(A)] $\|\hat\cL^{(\omega)}_t f\|_{\cB_w} \leq C e^{-\delta t} \|\hat Z_{\omega} f\|_{\cB} $ for all $\omega \notin \mathcal N_0$.
	\item[(B)] For all $\omega \in \mathcal N_0$, there exist $z(\omega) \in \C$, with $z(0)=0$ and $\re z(\omega) \leq 0$, a rank-1 projector $\Pi_\omega$, and a family of operators $t\mapsto Q_{\omega,t}$ with $\Pi_\omega \, Q_{\omega,t} = Q_{\omega,t} \, \Pi_\omega = 0$, such that
	\[
	\hat\cL^{(\omega)}_t f = e^{t\, z(\omega)} \Pi_\omega + Q_{\omega,t},
	\]
	and $Q_{\omega,t}$ satisfies $\|Q_{\omega,t} f\|_{\cB_w} \leq C e^{-\delta t}\|\hat Z_{\omega} f\|_{\cB}$.
	\item[(C)] The functions $\omega \to z_\omega:=z(\omega),\Pi_\omega, Q_{\omega,t}$ are analytic on $\mathcal N_0$ and, denoting by $D_{\omega}$ the derivative in direction $\omega$,   
	\begin{equation}\label{eq:taylor}
		\Re(D_{\omega}z(0))=0 \qquad \text{and} \qquad \Im(D_{\omega}z(0))=2\pi \int_M \langle \omega, X \rangle \diff m=0.
	\end{equation}
\item[(D)]The immaginary part of the second order term in the Taylor expansion of $z_\omega$ in $\mathcal N_{0}$ is zero, while the real part is given by the following {negative}-definite quadratic form $\sigma(\omega)$ on $H^1(S, \mathbb{R})$:
\begin{equation}\label{eq:sigma}
	\sigma(\omega)={-4\pi^2} \lim_{t\to \infty}\frac 1t \int_M \left( \int_0^t \langle \omega, X \rangle\circ \geo_s(x) \diff s \right)^2 \diff m.
\end{equation}
  	\end{itemize}
  	\end{proposition}
  \begin{proof}
  (A) and (B) follow from \Cref{prop:pert1} and \Cref{prop:zero_is_simple}, combined with the fact that the eigenvalues $z_j(\omega)$ vary continuously in $\omega$.
  
  Let us prove (C). The analyticity results follow by standard perturbation theory \cite{Kat}; let us prove \eqref{eq:taylor}. 
Recalling \eqref{eq:twisted} and using ${\hat\cL_t}'\mu = \mu$, we have 
\[
\begin{split}
\int \exp\left(\frac{2\pi \imath}{t}F_{\omega,t}\right)\diff \mu &=\int {\hat\cL_t}\left[\exp\left(\frac{2\pi \imath}{t}F_{\omega,t}\right)\right] \diff \mu = \int {\hat\cL_t}^{(\frac{\omega}{t})}(1) \diff \mu.
\end{split}
\]

We use the spectral decomposition of $\hat\cL_t^{(\omega)}$ to obtain
\begin{equation}\label{eq:limit2}
\begin{split}
\int \exp\left(\frac{2\pi \imath}{t}F_{\omega,t}\right)\diff \mu  =\int [\exp(tz_{\frac{\omega}{t}})\Pi_{\frac{\omega}{t}}(1)+Q_{\frac{\omega}{t},t}(1)] \diff \mu.
\end{split}
\end{equation}
Let us evaluate the right hand side of the above equation. 
For the second term, we use \eqref{eq:decay} {and the fact that $\mu\in \cB_w'$}, obtaining that
\[
\left| \int Q_{\frac{\omega}{t},t}(1) \diff \mu \right| \le \| Q_{\frac{\omega}{t},t}(1)\|_{\cB_w} \le C_{p,\omega} \,e^{-p t} \|\hat Z(1)\|_{\cB},
\]
which implies  
\[
\lim_{t\to \infty } \int Q_{\frac{\omega}{t},t}(1) \diff \mu =0.
\]
We focus now on the first term in the right hand side of \eqref{eq:limit2}. 
We note that, denoting by $\Pi'_0$ the spectral projector on the eigenvalue $z=0$ in the decomposition of the dual operator $\cL_t'$, it follows that
\[
\int \Pi_{0}(1)\diff \mu=[\Pi'_0(\mu)]1=\mu(1)=1,
\]
hence
\[
\lim_{t\to \infty} \exp(tz_{\frac{\omega}{t}}) \int \Pi_{\frac{\omega}{t}}(1) \diff \mu = \exp(\langle \nabla z(0), \omega \rangle) \int \Pi_{0}(1) \diff \mu = \exp(\langle \nabla z(0), \omega \rangle).
\]
We have then proved that 
\[
\lim_{t\to \infty} \int \exp\left(\frac{2\pi \imath}{t}F_{\omega,t}\right) \diff \mu  = \exp(\langle \nabla z(0), \omega \rangle).
\]
On the other hand, by ergodicity, for $m$--almost every $x \in M$ we have
\[
\frac{2\pi \imath}{t}F_{\omega,t} = \frac{2\pi \imath}{t}  \int_0^t \langle \omega, X \rangle\circ \geo_{s}(x) \diff s \to 2\pi \imath \int \langle \omega, X \rangle \diff m,
\]
as $t \to \infty$. 
Since $m$ and $\mu$ have the same disintegration along weak-unstable manifolds, the same convergence happens for $\mu$-almost every $x \in M$. 
Thus, 
\[
\exp(\langle \omega, \nabla z(0)\rangle) = \exp\left(  2\pi \imath \int \langle \omega, X \rangle \diff m \right),
\]
which implies that, for every $\omega \in \T^d$, we have
\[
\langle \nabla z(0), \omega \rangle =  2\pi \imath \int \langle \omega, X \rangle \diff m.
\]
Considering the real and imaginary parts yields 
\[
\Re(D_{\omega}z(0))=0 \qquad \text{and} \qquad \Im(D_{\omega} z(0))=2\pi \int_M \langle \omega, X \rangle \diff m.
\]
By \cite{KatSun}, we have $m( \langle \omega, X \rangle)=0$, which completes the proof.

Let us now prove (D).  
By the analytic perturbation result in (C) and \Cref{prop:zero_is_simple}, the projections $\Pi_{\omega}$ are all of rank 1, thus there exists an analytic curve of eigenvectors $u_\omega \in \cB$, with $u_0$ as in \Cref{prop:zero_is_simple}, satisfying $\hat\cL^{(\omega)}_t u_\omega = e^{t \, z(\omega)} u_\omega$.
We now differentiate twice both sides of this equation.
Differentiating once in direction $\eta$ gives us
\[
\hat\cL_t \left(2 \pi \imath \left(\int_0^t \langle \eta, X \rangle \circ \geo_s \diff s\right) e^{2 \pi \imath F_{\omega, t}} u_\omega + e^{2 \pi \imath F_{\omega, t}} D_\eta u_\omega \right) = t \, D_\eta z(\omega) e^{t \, z(\omega)} u_\omega + e^{t \, z(\omega)} D_\eta u_\omega.
\]
We differentiate again in direction $\varsigma$ at $\omega = 0$, using the fact that $D_\eta z(0)=0$, and we apply the dual eigenvector $\mu$ on both sides; we obtain
\begin{multline*}
D_\varsigma D_\eta z(0) = \frac{1}{t} \mu \Bigg[ 2 \pi \imath \left(\int_0^t \langle \eta, X \rangle \circ \geo_s \diff s\right) (D_\varsigma u_0) +  2 \pi \imath \left(\int_0^t \langle \varsigma, X \rangle \circ \geo_s \diff s\right) (D_\eta u_0) \\ - 4 \pi^2  \left(\int_0^t \langle \eta, X \rangle \circ \geo_s \diff s\right) \left(\int_0^t \langle \varsigma, X \rangle \circ \geo_s \diff s\right) u_0\Bigg].
\end{multline*}
Taking the limit as $t \to \infty$, ergodicity of the geodesic flow implies that the first two summands in the right hand side above vanish; hence, we conclude
\[
D_\varsigma D_\eta z(0) = \lim_{t \to \infty} - \frac{4\pi^2}{t} \int \left(\int_0^t \langle \eta, X \rangle \circ \geo_s \diff s\right) \left(\int_0^t \langle \varsigma, X \rangle \circ \geo_s \diff s\right) \diff m,
\]
which proves the result.
\end{proof}


\begin{remark}\label{rmk:CLT}
Using the Taylor expansion for $z(\omega)$ and the spectral decomposition of  $\cL^{(\frac{\omega}{\sqrt t})}_t1$, the above proposition should provide another proof of the CLT for geodesic flows on compact manifolds for the observable $\psi(x)=\langle \omega, X \rangle(x)$ with $\mu(\psi)=0$. Indeed,  $\varphi_t(\omega)=\mu(\cL^{(\frac{\omega}{\sqrt t})}_t1)$ is the characteristic function of the process $\{ (\sqrt t)^{-\frac 12}\int_0^t \langle \cdot, X \rangle\circ \geo_s\diff s\}_{t\in \R^+}$, and the above proposition gives the convergence as $t\to \infty$ of $\varphi_t(\omega)$ to  $e^{\frac 12 \sigma({\omega})}$.
\end{remark}


\section{Horocycle ergodic integrals and proofs of theorems}\label{sec:horocycle_integrals}

In this section we prove our main results: \Cref{thm:main_result_3} on the power law in the compact case, \Cref{thm:main_result} on the asymptotics of ergodic integrals of the horocycle flow on the Abelian cover $\widetilde{M}$ and \Cref{thm:main_result_2} on equidistribution for geodesic translates of horocycle segments.
We first collect some preliminary facts we will need in the proofs.

\subsection{Preliminaries}
Let $T \geq 1$ and $x \in \tM$. From the renormalization relation \eqref{eq:renormalization} between geodesic and horocycle flow, it follows that the push-forward of the horocycle orbit segment starting at $x$ of length $T$ under $\geo_{t}$ becomes a horocycle orbit segment of length $\tau(T,t,x)$, and, by \Cref{lemma:Marcus}, we know that $C_\tau^{-1} \frac{T}{e^{h_{\topp}t}} \leq \tau(T,t,x) \leq C_\tau \frac{T}{e^{h_{\topp}t}}$ for all $t\in \R$ such that $\tau(T,t,x) \geq 1$. We will now prove some sharper estimates on $\tau(T,t,x)$  and on $\varpi(t,x)$, using a simplified version of the argument we will use for the proof of \Cref{thm:main_result}. We also recall from \Cref{cor:Jt_invariant} that the functions $\tau(T,t,\cdot)$ and $\varpi(t,\cdot)$ are $\Deck$-invariant, and hence well defined on $M$.

As in \Cref{rmk:constant}, here and henceforth by the symbol $\sC$ we indicate a positive constant that depends only on the geometry of the system whose value is allowed to change in different occurrences.
We are interested in the ergodic integrals along horocycle orbits, and we want to make sure that we can exploit the spectral results of \Cref{sec:spectral_picture} to study them. In order to do that, we need the operation of taking an ergodic integral to be a well-defined functional on our Banach spaces. Up to introducing an appropriate weight, this is indeed a direct consequence of the definition of the weak norm, see \Cref{lem:integral_is_cont_funct}. It follows that we need to introduce a $\mathscr{C}^{1+\Lip}$ cut-off function in the ergodic integrals along horocycle orbits in order to see them as functionals on our Banach spaces. 
\begin{lemma}\label{lem:smoothened_erg_int}
\sloppy	Fix $x \in \tM$ and $T \geq 1$; let $t\in \R$ be such that $\tau(T,t,x) \geq 4$. There exists $\psi=\psi_{T,t} \colon [0,\tau(T,t,x)] \to [0,1]$ of class $\mathscr{C}^{1+\Lip}$ so that $\|\psi'\|_{\mathscr{C}^{\Lip}} \leq  16$ and, for any bounded function $f$ on $\tM$, we have
	\[
	\left\lvert \int_{0}^{T} f \circ \horo_s(x) \diff s - \int_{0}^{T} f \circ \horo_s(x) \cdot [\psi \circ \tau(s,t,x) ]\diff s \right\rvert \leq \sC \|f\|_{\infty} e^{h_{\topp}t}.
	\]
\end{lemma}
\begin{proof}
	Let us abbreviate $\tau = \tau(T,t,x)$. We define the continuous, piecewise linear function $\psi'=\psi_{T,t}'$ as 
\[
	\psi'(r):=
	\begin{cases}
		4\left(1-\left|4r-3\right|\right) & \text{if } \frac{1}{2} \leq r\leq 1,\\[1mm]
		0 & \text{if }r \in [0,\frac{1}{2}] \cup [1, \tau-1] \cup [\tau-\frac{1}{2},\tau],\\[1mm]
		4\left(4\left|r-\tau+\frac{3}{4}\right|-1\right) & \text{if } \tau-1 \leq r\leq \tau-\frac{1}{2},
	\end{cases}
	\]
which clearly satisfies the Lipschitz bound $\|\psi'\|_{\mathscr{C}^{\Lip}} \leq 4^2$. Let $\psi(r) = \int_0^r \psi'(\xi)\diff \xi$ be its primitive with $\psi(0) = 0$. By definition, $\psi(r) = 1$ for all $r \in [1, \tau(T,t,x)-1]$. Then, using again \Cref{lemma:Marcus}, we see that $\psi \circ \tau(s,t,x) =1$ for all $s \in [\sC e^{h_{\topp}t}, \tau - \sC e^{h_{\topp}t}]$; the claim now follows easily from the fact that $\|\psi\|_{\infty} \leq 1$.
\end{proof}

Recall that $\varpi(s,x)$ denotes the Margulis parametrization of the horocycle flow in \Cref{prop:Margulis}, which we recovered in \Cref{rk:leafwise_measure}.

\begin{proposition}\label{prop:estimate_on_tau}
There exists $\delta_{\ast}>0$ so that for all $x \in \tM$, and for all $T\geq 1$
, we have
\[
\left\lvert T- \varpi(T,x)\right\rvert  \leq \sC T^{1-\delta_{\ast}}.
\]
Moreover, for every $0\leq t\leq \frac{\log T}{h_{\topp}}$ for which $\tau(T,t,x) \geq 4$, we have
\[
\left\lvert \tau(T,t,x) - e^{-h_{\topp}t}T \right\rvert \leq \sC e^{-h_{\topp}t}T^{1 -\delta_{\ast}}.
\]
\end{proposition}
\begin{proof}
Since the functions $\tau(T,t,x)$ and $\varpi(t,x)$ are periodic on $\tM$ by  \Cref{cor:Jt_invariant}, we can replace $x$ with its projection $p(x) \in M$ and we work on the compact manifold $M$.

Let $t$ be as in the assumptions and let $\psi = \psi_{T,t}$ be given by \Cref{lem:smoothened_erg_int}. We have
\[
\left\lvert T - \int_{0}^{T} 1 \circ \horo_s(x) \cdot \psi\circ \tau(s,t,x) \diff s\right\rvert \leq \sC e^{h_{\topp}t} .
\]
We now focus on the integral above.
From \Cref{eq:renormalization} and  \Cref{cor:inverse_of_Jt}, we obtain
\begin{multline*}
\int_0^T 1 \circ \horo_s(x) \cdot [\psi \circ \tau(s,t,x) ] \diff s = \int_0^T 1 \circ \geo_{-t}  \circ \geo_{t} \circ \horo_s(x)\cdot [\psi \circ \tau(s,t,x) ] \diff s \\
 = \int_0^T 1 \circ \geo_{-t} \circ \horo_{\tau(s,t,x)} \circ \geo_{t}(x) \cdot J_{-t} \circ \geo_{t} \circ \horo_s(x) \cdot \frac{\partial \tau}{\partial s}(s,t,x) \cdot [\psi \circ \tau(s,t,x) ] \diff s.
\end{multline*}
Doing a change of variable and setting $x_{t} = \geo_{t}(x)$, we obtain
\[
\int_0^T 1 \circ \horo_s(x) \cdot [\psi \circ \tau(s,t,x) ] \diff s = \int_{0}^{\tau(T,t,x)} (1 \circ \geo_{-t})\circ \horo_s(x_{t}) \cdot J_{-t} (\horo_s (x_{t})) \cdot \psi(s)\diff s,
\]
and the integral can be rewritten, using the normalized untwisted transfer operator, as
\[
\int_{0}^{T} 1 \circ \horo_s(x) \cdot [\psi \circ \tau(s,t,x) ] \diff s = e^{h_{\topp}t} \int_{0}^{\tau(T,t,x)} \hat{ \cL}_{t} (1) \circ \horo_s(x_{t}) \cdot \psi(s) \diff s.
\]
We now decompose the horocycle arc of length $\tau(T,t,x)$ into arcs of length 1 starting at points $\{x_{t,j}\}_j$ and introduce a smooth partition of unity $\{\rho_j\}_j$ supported on these segments. 
Using the notation introduced in \Cref{lem:integral_is_cont_funct}, we obtain
\[
\int_{0}^{T} 1 \circ \horo_s(x) \cdot \psi \circ \tau(s,t,x) \diff s = e^{h_{\topp}t} \sum_j \mathcal{I}[x_{t,j},\psi \cdot \rho_j](\hat{ \cL}_{t} (1)). 
\]
We can now use the results in \Cref{sec:spectral_picture}:
the spectral decomposition of \Cref{prop:pert2} at $\omega = 0$ gives us
\[
\| \hat{ \cL}_{t} (1) - \Pi_0(1)\|_{\cB_w} \leq \sC \|\hat{Z}1 \|_{\cB}e^{-\delta t} \leq \sC e^{-\delta t},
\]
for some $\delta >0$; moreover, 
$\Pi_0(1) = \mu(1)\phi_0=\phi_0$. 
Therefore, we deduce,
\begin{multline*}
\left\lvert \int_{0}^{T} 1 \circ \horo_s(x) \cdot \psi \circ \tau(s,t,x) \diff s - e^{h_{\topp}t} \sum_j \mathcal{I}[x_{t,j},\psi \cdot \rho_j](\phi_0)\right\rvert \\
=\left\lvert \int_{0}^{T} 1 \circ \horo_s(x) \cdot \psi \circ \tau(s,t,x) \diff s - e^{h_{\topp}t}  \mathcal{I}[x_{t},\psi|_{[0,\tau(T,t,x)]}](\phi_0)  \right\rvert \\
\leq \sC e^{(h_{\topp}-\delta) t} |\tau(T,t,x)| \cdot \|\psi\|_{\mathscr{C}^{1+\alpha}}\leq \sC T e^{-\delta t}.
\end{multline*}
Since 
\[
\left\lvert   \mathcal{I}[x_{t},\psi|_{[0,\tau(T,t,x)]}](\phi_0)   - \varpi(\tau(T,t,x),x_t)\right\rvert\leq \sC,
\]
we have then showed that 
\[
\left\lvert T - e^{h_{\topp}t} \varpi(\tau(T,t,x),x_t) \right\rvert \leq \sC (e^{h_{\topp}t} +T e^{-\delta t}).
\]
The bound $\left\lvert T- \varpi(T,x)\right\rvert  \leq \sC T^{1-\delta_{\ast}}$ follows from \eqref{eq:self_similarity_of_alpha}, with $1-\delta_{\ast} =h_{\topp}/(h_{\topp}+\delta)$ and choosing $t=\frac{1-\delta_{\ast}}{h_{\topp}} \log T =  \frac{\log T}{h_{\topp} + \delta}$.

Let us prove the second claim. The bound above applied to $\tau(T,t,x)$ and \eqref{eq:self_similarity_of_alpha} yield
\[
\left\lvert e^{-h_{\topp}t} \varpi(T,x) - \tau(T,t,x)\right\rvert =\left\lvert  \varpi(\tau(T,t,x),x_t) - \tau(T,t,x)\right\rvert
 \leq \sC \tau(T,t,x)^{1-\delta_{\ast}} \leq \sC (e^{-h_{\topp}t}T)^{1-\delta_{\ast}}.
\]
Since $\left\lvert e^{-h_{\topp}t}T- e^{-h_{\topp}t}\varpi(T,x)\right\rvert  \leq \sC e^{-h_{\topp}t}T^{1-\delta_{\ast}}$, the proof is complete.
\end{proof}

\subsection{Proof of \Cref{thm:main_result_3}}\label{sec:proofC} 
The proof follows closely that of \Cref{prop:estimate_on_tau}.
Let $f\in \cC^r(M)$ and $x\in M$. Let $\psi = \psi_{T,t}$ be given by \Cref{lem:smoothened_erg_int}, where $t = \frac{\log T}{h_{\topp} + \delta}$, so that $e^{h_{\topp} t} = T^{\frac{h_{\topp}}{h_{\topp} + \delta}} =  T^{1-\delta_{\ast}}$, as in the proof of \Cref{prop:estimate_on_tau}.  \Cref{lem:smoothened_erg_int} provides
\begin{equation}\label{eq:power1}
\left\lvert \int_0^T f\circ \horo_s(x)\diff s - \int_{0}^{T} f \circ \horo_s(x) \cdot \psi\circ \tau(s,t,x) \diff s\right\rvert \leq \sC e^{h_{\topp}t} \|f\|_{\cC^r}.
\end{equation}
We now focus on the second integral above.
From \Cref{eq:renormalization} and  \Cref{cor:inverse_of_Jt}, we obtain
\begin{multline*}
\int_0^T f \circ \horo_s(x) \cdot [\psi \circ \tau(s,t,x) ] \diff s = \int_0^T f \circ \geo_{-t}  \circ \geo_{t} \circ \horo_s(x)\cdot [\psi \circ \tau(s,t,x) ] \diff s \\
 = \int_0^T f \circ \geo_{-t} \circ \horo_{\tau(s,t,x)} \circ \geo_{t}(x) \cdot J_{-t} \circ \geo_{t} \circ \horo_s(x) \cdot \frac{\partial \tau}{\partial s}(s,t,x) \cdot [\psi \circ \tau(s,t,x) ] \diff s.
\end{multline*}
Doing a change of variable and setting $x_{t} = \geo_{t}(x)$, we obtain
\[
\int_0^T f \circ \horo_s(x) \cdot [\psi \circ \tau(s,t,x) ] \diff s = \int_{0}^{\tau(T,t,x)} (f \circ \geo_{-t})\circ \horo_s(x_{t}) \cdot J_{-t} (\horo_s (x_{t})) \cdot \psi(s)\diff s,
\]
and the integral can be rewritten as
\[
\int_{0}^{T} f \circ \horo_s(x) \cdot [\psi \circ \tau(s,t,x) ] \diff s = e^{h_{\topp}t} \int_{0}^{\tau(T,t,x)} \hat{ \cL}_{t} (f) \circ \horo_s(x_{t}) \cdot \psi(s) \diff s.
\]
We now decompose the horocycle arc of length $\tau(T,t,x)$ into arcs of length 1 starting at points $\{x_{t,j}\}_j$ and introduce a smooth partition of unity $\{\rho_j\}_j$ supported on these segments. 
Using the notation introduced in \Cref{lem:integral_is_cont_funct}, we obtain
\[
\int_{0}^{T} f \circ \horo_s(x) \cdot \psi \circ \tau(s,t,x) \diff s = e^{h_{\topp}t} \sum_j \mathcal{I}[x_{t,j},\psi \cdot \rho_j](\hat{ \cL}_{t} (f)). 
\]
We can now use the results in \Cref{sec:spectral_picture}:
the spectral decomposition of \Cref{prop:pert2} at $\omega = 0$ gives
\[
\| \hat{ \cL}_{t} (f) - \Pi_0(f)\|_{\cB_w} \leq \sC \|\hat{Z}f \|_{\cB}e^{-\delta t} \leq \sC e^{-\delta t}\|\hat Z f\|_{\cC^{r-1}}\le \sC e^{-\delta t}\|f\|_{\cC^r},
\]
for some $\delta >0$; moreover, $\Pi_0(f) = \mu(f) \phi_0$. Therefore, we deduce,
\begin{multline*}
\left\lvert \int_{0}^{T} f \circ \horo_s(x) \cdot \psi \circ \tau(s,t,x) \diff s - e^{h_{\topp}t} \sum_j \mathcal{I}[x_{t,j},\psi \cdot \rho_j](f)\right\rvert \\
=\left\lvert \int_{0}^{T} f \circ \horo_s(x) \cdot \psi \circ \tau(s,t,x) \diff s - \mu(f)e^{h_{\topp}t}  \mathcal{I}[x_{t},\psi|_{[0,\tau(T,t,x)]}](\phi_0)  \right\rvert \\
\leq \sC e^{(h_{\topp}-\delta) t} |\tau(T,t,x)| \cdot \|\psi\|_{\mathscr{C}^{1+\alpha}}\|f\|_{\cC^r}\leq \sC T e^{-\delta t}\|f\|_{\cC^r},
\end{multline*}
where in the last line we have used Lemma \ref{lemma:Marcus}.
Then, as in the proof of \Cref{prop:estimate_on_tau}, we deduce 
\[
\left\lvert  \int_0^T f\circ \horo_s(x)\diff s - \mu(f)e^{h_{\topp}t} \varpi(\tau(T,t,x),x_t) \right\rvert \leq \sC (e^{h_{\topp}t} +T e^{-\delta t})\|f\|_{\cC^r} =  \sC  T^{1-\delta_{\ast}} \|f\|_{\cC^r}.
\]
On the other hand, by \Cref{prop:estimate_on_tau}, $|e^{h_{\topp}t} \varpi(\tau(T,t,x),x_t)-T| =|\varpi(T,x)-T| \le \sC T^{1-\delta_{\ast}}$, hence
\[
\left\lvert \frac 1T \int_0^T f\circ \horo_s(x)\diff s - \mu(f) \right\rvert \leq \sC  T^{1-\delta_{\ast}} \|f\|_{\cC^r}.
\]
This completes the proof of \Cref{thm:main_result_3}.

\subsection{Proof of \Cref{thm:main_result}}\label{sec:proof_main_result}

We are now ready to prove \Cref{thm:main_result}. In the course of the proof, we follow the arguments in \Cref{prop:estimate_on_tau} applied for the \emph{twisted} transfer operators.

Let $f \in \mathscr{C}^{2}_c(\tM)$ and $x \in \tM$ be fixed. Henceforth, we assume that we fixed a fundamental domain $\cF \subset \tM$ containing $x$ and we also take\footnote{Although the final result is independent of this choice, it simplifies the notation.} $x=x_0$ to be the base point in $\cF$ as in \Cref{sec:xi_and_x0}.

By \Cref{lem:proj_cl} and \Cref{lem:proj_Xi}, we can decompose
\begin{equation}\label{eq:split_of_f}
f(x) = \int_{\T^d} f_\omega(x) \diff \omega, \qquad \text{where} \qquad 
f_\omega = \pi_\omega(f) = \Xi_{-\omega} \circ \pi_0 \circ \Xi_{\omega}(f)
\end{equation}
We will write
\[
f_\omega = \Xi_{-\omega} (u_\omega), \qquad \text{where} \qquad u_\omega = \pi_0 \circ \Xi_{\omega}(f) \in \mathscr{C}^{2}(M,0) = \mathscr{C}^{2}(M).
\]
\begin{lemma}\label{lem:norm_u_omega}
The $\mathscr{C}^2$-norm of $u_\omega$ is bounded by a constant $C(f)$ depending on $f$ only (namely, on the $\mathscr{C}^2$-norm of $f$ and on the diameter of its support). Moreover, the norm of the gradient of $\omega \mapsto u_\omega$ is bounded by a constant $C(f,x)$ that depends on $C(f)$ and on the distance between $x$ and the support of $f$.
\end{lemma}
\begin{proof}
Recall that, by \Cref{lem:proj_Xi}, we have
\[
u_\omega (y) = \sum_{D \in \Deck} f\circ D(y) \cdot \exp \left(2 \pi \imath \int_{x_0}^{D(y)} p^{\ast} \omega\right),
\]
and the sum above is finite. The first claim is proved as in \Cref{lem:proj_Xi}. On the other hand, differentiating the expression above with respect to $\omega$ in direction, say, $\eta$ amounts to 
\[
D_{\eta} u_\omega (y)= \sum_{D \in \Deck} f\circ D(y) \cdot \left(2 \pi \imath \int_{x_0}^{D(y)} p^{\ast} \eta \right) \exp\left(2 \pi \imath \int_{x_0}^{D(y)} p^{\ast} \omega\right).
\]
The term in brackets can be bounded by a constant depending only on the distance between $x_0$ and $D(y)$, and the only non zero summands in the sum are those for which $D(y)$ is in the support of $f$. This implies the claim.
\end{proof}

We fix $\omega \in \T^d$ and study the horocycle ergodic integral of $f_\omega$ up to time $T\geq \sC$. 
We also fix $t > 1$ so that 
\begin{equation}\label{eq:choice_of_t}
e^{h_{\topp}t}  = \frac{T}{(\log T)^d}.
\end{equation}
The choice of the exponent $d$ above is not important, as long as it is larger than $(d+1)/2$.
We now follow the same argument as in the proof of \Cref{prop:estimate_on_tau}.
 
By \Cref{lem:proj_cl}, there exists a constant $C(f)$ depending on $f$ only and independent of $\omega$, so that $\|f_\omega\|_{\infty} \leq C(f)$. 
Thus, by \Cref{lem:smoothened_erg_int}, we have
\begin{equation}\label{eq:approx_1}
\left\lvert \int_{0}^{T} f_\omega \circ \horo_s(x) \diff s - \int_{0}^{T} f_\omega \circ \horo_s(x) \cdot [\psi \circ \tau(s,t,x) ]\diff s \right\rvert \leq \sC C(f) \frac{T}{(\log T)^d}.
\end{equation}
We now focus on the second integral above:
\begin{multline*}
\int_0^T f_\omega \circ \horo_s(x) \cdot [\psi \circ \tau(s,t,x) ] \diff s = \int_0^T f_\omega \circ \geo_{-t}  \circ \geo_{t} \circ \horo_s(x)\cdot [\psi\circ \tau(s,t,x) ] \diff s \\
 = \int_0^T f_\omega \circ \geo_{-t} \circ \horo_{\tau(s,t,x)} \circ \geo_{t}(x) \cdot J_{-t} \circ \geo_{t} \circ \horo_s(x) \cdot \frac{\partial \tau}{\partial s}(s,t,x) \cdot [\psi \circ \tau(s,t,x) ] \diff s,
\end{multline*}
hence, doing a change of variable and setting $x_{t} = \geo_{t}(x)$, we obtain
\[
\int_0^T f_\omega \circ \horo_s(x) \cdot [\psi\circ \tau(s,t,x) ] \diff s = \int_{0}^{\tau(T,t,x)} (f_\omega \circ \geo_{-t})\circ \horo_s(x_t) \cdot J_{-t} (\horo_s (x_{t})) \cdot \psi(s)\diff s.
\]
Note that, recalling \eqref{eq:TTO}, the integral in the right hand side above contains
\[
\cL_{t} f_\omega \circ \horo_s = [\cL_{t} \circ \Xi_{-\omega}(u_\omega)] \circ \horo_s =  [\Xi_{-\omega} \circ \cL^{(\omega)}_{t} (u_{\omega})]\circ \horo_s;
\]
therefore, for the normalized transfer operator $\cL_{t}^{(\omega)}  = e^{h_{\topp}t} \hat{\cL}_{t}^{(\omega)} $, we can rewrite
\[
\int_0^T f_\omega \circ \horo_s(x) \cdot [\psi \circ \tau(s,t,x) ] \diff s = e^{h_{\topp}t} \int_{0}^{\tau(T,t,x)} e^{- 2 \pi \imath \xi_{\omega}(\horo_s(x_{t}))} \cdot \hat{\cL}^{(\omega)}_{t} (u_{\omega}) \circ \horo_s(x_{t}) \cdot \psi(s) \diff s.
\]
We introduce a smooth partition of unity $\{\rho_j\}_{j \in \mathcal{J}}$ on the horocycle segment starting at $x_t$ of length $\tau(T,t,x)$ so that each $\rho_j$ is supported on a subarc $I_j$ of length 1 starting at $x_{t,j}$, and $|\mathcal{J}|\leq \sC\tau(T,t,x) \leq \sC (\log T)^d$ by our choice \eqref{eq:choice_of_t}. Moreover, since $I_j$ is of length 1, $\|\rho_j\|_{C^{1+\operatorname{Lip}}}\le C$ for some uniform $C<\infty$.
We have
\[
e^{- 2 \pi \imath \xi_{\omega}(\horo_s(x_{t,j}))} = e^{- 2 \pi \imath \xi_{\omega}(x_{t,j})} \cdot \exp\left(- 2 \pi \imath \int_{x_{t,j}}^{\horo_s(x_{t,j})}p^{\ast}\omega \right),
\]
thus, denoting $\psi_j$ the restriction of $\psi$ to $I_j$, we deduce
\[
\begin{split}
\int_0^T f_\omega \circ \horo_s(x) \cdot [\psi \circ \tau(s,t,x) ] \diff s &= e^{h_{\topp}t} \sum_{j \in \mathcal{J}} \int_{0}^{1} e^{- 2 \pi \imath \xi_{\omega}(\horo_s(x_{t,j}))} \cdot \hat{\cL}^{(\omega)}_{t} (u_{\omega}) \circ \horo_s(x_{t}) \cdot \psi_j(s) \cdot \rho_j(s) \diff s \\
 &= e^{h_{\topp}t} \sum_{j \in \mathcal{J}} e^{- 2 \pi \imath \xi_{\omega}(x_{t,j})}\int_{0}^{1} \hat{\cL}^{(\omega)}_{t} (u_{\omega}) \circ \horo_s(x_{t,j}) \cdot \Psi_j(s,\omega) \diff s,
\end{split}
\]
where we defined
\begin{equation}\label{eq:def_of_Psi}
\Psi_j(s, \omega) := \exp\left(- 2 \pi \imath \int_{x_{t,j}}^{\horo_s(x_{t,j})}p^{\ast}\omega \right)\cdot \psi_j(s)\cdot \rho_j(s).
\end{equation}
We note that, by \Cref{lem:smoothened_erg_int}, the $\mathscr{C}^{1+\Lip}$-norm of the function $\Psi_j(s, \omega)$ is \emph{uniformly bounded} in $t$ and $j$. 
Moreover, the integrands $\hat{\cL}^{(\omega)}_{t} (u_{\omega}) \cdot \Psi_j(s, \omega)$ are $\Deck$-invariant, hence we can replace $x_{t,j}\in \tM$ by $p(x_{t,j}) \in M$.
To sum up, using the notation of \Cref{lem:integral_is_cont_funct}, from \eqref{eq:split_of_f} and \eqref{eq:approx_1} we deduced
\begin{equation}\label{eq:approx_2}
	\left\lvert \int_{0}^{T} f \circ \horo_s(x) \diff s - e^{h_{\topp}t} \sum_{j \in \mathcal{J}} \int_{\T^d} e^{- 2 \pi \imath \xi_{\omega}(x_{t,j})} \mathcal{I}[p(x_{t,j}),\Psi_j(\cdot, \omega)]\left(\hat{\cL}^{(\omega)}_{t} u_{\omega} \right) \diff \omega \right\rvert \leq \sC C(f) \frac{T}{(\log T)^d}.
\end{equation}
We now want to exploit the spectral results of \Cref{sec:spectral_picture}. Let $B(0,R) \subset \T^d$ denote the ball of centre 0 and radius $R$ in $\T^d$. By \Cref{prop:pert2}, there exists $\delta >0$ for which the following facts hold:
\begin{enumerate}
\item for all $\omega \notin B(0,\delta)$, 
\[
\|\hat{\cL}^{(\omega)}_{t} u_\omega\|_{\cB_w} \leq \sC e^{-\delta t } \|\hat Z_{\omega} u_\omega\|_{\cB} \leq \sC e^{-\delta t } \|u_\omega\|_{\mathscr{C}^2}\leq \sC C(f,x) e^{-\delta t},
\]
where we used \Cref{lem:generator} in the second inequality and \Cref{lem:norm_u_omega} in the third;
\item for all $\omega \in B(0,\delta)$, 
\[
\|\hat{\cL}^{(\omega)}_{t} u_\omega - e^{z(\omega) t } \Pi_{\omega} u_\omega\|_{\cB_w} \leq \sC e^{-\delta t } \|\hat Z_{\omega} u_\omega\|_{\cB} \leq \sC C(f,x) e^{-\delta t }.
\]
\end{enumerate}
Hence, by \Cref{lem:integral_is_cont_funct}, our choice of $t$, recalling that $|\mathcal{J}| \leq \sC (\log T)^d$ and since $\|\Psi_j(\cdot, \omega)\|_{\mathscr{C}^{1+\Lip}} \leq \sC$, from \eqref{eq:approx_2} we deduce
\begin{multline*}
\left\lvert \int_{0}^{T} f \circ \horo_s(x) \diff s - e^{h_{\topp}t} \sum_{j \in \mathcal{J}} \int_{B(0,\delta)} e^{- 2 \pi \imath \xi_{\omega}(x_{t,j}) + z(\omega) t } \mathcal{I}[p(x_{t,j}),\Psi_j(\cdot, \omega)]\left(\Pi_{\omega} u_\omega \right) \diff \omega \right\rvert \\ \leq \sC C(f,x) \left( e^{(h_{\topp}- \delta )t} |\mathcal{J}| + \frac{T}{(\log T)^d}\right) \leq \sC C(f,x) \frac{T}{(\log T)^d}. 
\end{multline*}
If we define 
\[
A_j(\omega) := \mathcal{I}[p(x_{t,j}),\Psi_j(\cdot, \omega)]\left(\Pi_{\omega} u_\omega \right),
\]
then we can rewrite the inequality above as 
\begin{equation}\label{eq:approx_4}
	\left\lvert \int_{0}^{T} f \circ \horo_s(x) \diff s -e^{h_{\topp}t} \sum_{j \in \mathcal{J}} \int_{B(0,\delta)} e^{- 2 \pi \imath \xi_{\omega}(x_{t,j}) +z(\omega) t } A_j(\omega) \diff \omega \right\rvert  \leq \sC C(f,x) \frac{T}{(\log T)^d}. 
\end{equation}
The key observation is the following.
\begin{lemma}\label{lem:smoothness_of_A}
	The functions $\omega \mapsto A_j(\omega)$ are analytic in $\omega$, their derivatives are bounded uniformly in $T$ and $j$, and depend on $f$ and on the distance between $x$ and $\supp(f)$. Furthermore, $A_j(0) = \mu(f) \cdot  \mathcal{I}[p(x_{t,j}),\psi_j \, \rho_j](\phi_0)$.
\end{lemma}
The analyticity of $A_j(\omega)$ might be used to prove a full asymptotic expansion for the horocycle integrals, see \Cref{rk:asymptotic_expansion}. Here, we limit ourselves to the leading term, thus we only use the fact that it is of class $\mathscr{C}^1$.
\begin{proof}
The function $\omega \mapsto u_\omega$ is analytic in $\mathscr{C}^2(M)$ (see \Cref{lem:norm_u_omega}), and hence is an analytic family of vectors in $\cB$. By \Cref{prop:pert2}, so is the family $\omega \mapsto \Pi_\omega u_\omega$. The definition \eqref{eq:def_of_Psi} also shows that $\omega \mapsto \Psi_j(\cdot, \omega)$ are analytic in $\mathscr{C}^{1+\alpha}((0,1))$ and their derivatives are uniformly bounded, since the integral in their definitions is over a horocycle orbit segment of length at most 1. 
This proves the analyticity of $A_j$.

Finally, at $\omega=0$ we have $\Psi = \psi_j \cdot \rho_j$ and, from \Cref{lem:proj_Xi} and the fact that $\Pi_0(\cdot)=\mu(\cdot) \phi_0$,
\[
\Pi_0 u_0 = \mu(\pi_0(f)) \phi_0 =  \left( \int_M  \sum_{D \in \Deck} f\circ D \diff \mu\right) \phi_0 =\left(  \sum_{D \in \Deck} \int_{D(\cF)} f \diff \mu \right) \phi_0= \mu(f)\phi_0.
\]
Therefore,
\[
A_j(0) = \mu(f) \cdot \mathcal{I}[p(x_{t,j}), \psi_j \cdot \rho_j] (\phi_0),
\] 
which completes the proof.
\end{proof}
Let us fix $j \in \mathcal{J}$; we need to study the integral
\[
\int_{B(0,\delta)} e^{- 2 \pi \imath \xi_{\omega}(x_{t,j})+z(\omega) t } A_j(\omega) \diff \omega.
\]
One can use stationary phase methods to write an expansion of the integral above; however, the computations become rather complicated when $d$ increases. We limit ourselves to compute the leading term, and we refer to the computations carried out in \cite{BFRT}, in the case of translation flows.

Let $\Sigma$ denote the positive definite matrix associated to the quadratic form $-\frac{1}{4\pi^2}\sigma$ in \eqref{eq:sigma}. By Taylor's Theorem and \Cref{lem:smoothness_of_A}, we get
\begin{multline*}
	\left\lvert \int_{B(0,\delta)} e^{- 2 \pi \imath \xi_{\omega}(x_{t,j})+z(\omega) t } A_j(\omega) \diff \omega - A_j(0)\int_{B(0,\delta)} e^{- 2 \pi \imath \xi_{\omega}(x_{t,j}) -2\pi^2 t \omega \cdot \Sigma \omega }  \diff \omega \right\rvert \\ 
	\leq \sC C(f,x) \int_{B(0,\delta)} e^{-t \|\omega \|^2  } \|\omega\| \diff \omega 
	\leq \sC C(f,x) t^{-\frac{d+1}{2}} \int_{\R^d} e^{-\|\omega\|^2  } \|\omega\| \diff \omega \\  
	\leq \sC C(f,x) t^{-\frac{d+1}{2}}.
\end{multline*} 
We also notice that we can restrict the domain of integration to the ball $B(0,\delta t^{-1/4})$, since for all $\|\omega\| \geq \delta t^{-1/4}$ the integrand can be bounded by $e^{-2\pi^2 \delta \sqrt{t}}$, which is less than $t^{-\frac{d+1}{2}}$.
Hence, we are reduced to study 
\[
A_j(0)\int_{B(0,\delta t^{-1/4})} e^{- 2 \pi \imath \xi_{\omega}(x_{t,j}) -2\pi^2 t \omega \cdot \Sigma \omega }  \diff \omega.
\]
Let us define $t_{\ast}$ to be the \emph{normalizing time}, that is the positive real 
\begin{equation}\label{eq:choice_t_ast} 
t_{\ast} >0 \quad  \text{ is such that } \quad  \tau(T,t_{\ast},x)=1.
\end{equation}
We observe that $t_{\ast} > t$ and $|t_{\ast} - t| \leq \sC \log t$. Since the 1-form $\omega$ is harmonic, we can write
\[
\xi_{\omega}(x_{t,j}) = \int_x^{x_{t,j}}p^{\ast}\omega =  \int_x^{\geo_{t_{\ast}}(x)}p^{\ast}\omega + \int_{\geo_{t_{\ast}}(x)}^{\geo_{t_{\ast}-t}(x_{t,j})}p^{\ast}\omega + \int_{\geo_{t_{\ast}-t}(x_{t,j})}^{x_{t,j}}p^{\ast}\omega. 
\] 
The first term above is the Frobenius function (the geodesic winding cycle) as in \eqref{eq:F}; the points $\geo_{t_{\ast}}(x)$ and $\geo_{t_{\ast}-t}(x_{t,j})$ in the second integral both lie on a horocycle orbit segment of length 1, and the last integral is over a segment of length $\mathcal{O}(\log t)$. Thus, there exists a vector $\eb \in \R^d$ of norm $\|\eb\|\leq \sC\log t$ so that 
\[
\left\| \xi_{\omega}(x_{t,j})  - \int_x^{\geo_{t_{\ast}}(x)}p^{\ast}\omega \right\| \leq \sC | \omega \cdot \eb|.
\]
Here, we have used the identification of the vector space $\mathcal{H}$ with $\R^d$ with the aid of the basis of 1-forms $\omega_k$ for $k=1,\dots, d$   we fixed in \Cref{sec:twisted_hilbert_spaces}.
Let us define $F_{\ast} = F_{\ast}(x,t_{\ast})$ to be the vector in $\R^d$ whose components are $F_{t_{\ast}, \omega_k}(x) = \int_x^{\geo_{t_{\ast}}(x)}p^{\ast}\omega_k$.
Then, since for all $\omega \in B(0,\delta t^{-1/4})$ we can bound 
\[
|e^{- 2 \pi \imath  \xi_{\omega}(x_{t,j}) } - e^{- 2 \pi \imath  \omega \cdot F_{\ast} } | \leq \sC |\omega \cdot \eb| \leq \sC \|\omega \| \log t,
\]
we deduce 
\begin{multline*}
\left\lvert  \int_{B(0,\delta t^{-1/4})} e^{- 2 \pi \imath  \xi_{\omega}(x_{t,j})  -2\pi^2 t \omega \cdot \Sigma \omega }  \diff \omega - \int_{B(0,\delta t^{-1/4})} e^{- 2 \pi \imath \omega \cdot F_{\ast}-2\pi^2 t \omega \cdot \Sigma \omega }  \diff \omega \right\rvert \\
\leq \sC \log t \int_{B(0,\delta t^{-1/4})} e^{-t \|\omega\|^2  } \|\omega\| \diff \omega \leq \sC (\log t) \, t^{-\frac{d+1}{2}}.
\end{multline*} 
Thus, from \eqref{eq:approx_4} and recalling \eqref{eq:choice_of_t}, we obtain
\begin{multline}\label{eq:approx_5}
	\left\lvert \int_{0}^{T} f \circ \horo_s(x) \diff s -e^{h_{\topp}t}  \left(\sum_{j \in \mathcal{J}} A_j(0)\right)  \int_{B(0,\delta t^{-1/4})} e^{- 2 \pi \imath \omega \cdot F_{\ast}-2\pi^2 t \omega \cdot \Sigma \omega } \diff \omega\right\rvert \\  \leq \sC C(f,x) \frac{T \cdot \log \log T}{(\log T)^{\frac{d+1}{2}}}. 
\end{multline}
Now, by \Cref{lem:smoothness_of_A}, we have
\[
\left\lvert \sum_{j \in \mathcal{J}} A_j(0) -\mu(f) \varpi(\tau(T,t,x),p(x_t)) \right\rvert = |\mu(f)| \left\lvert  \mathcal{I}[p(x_{t}),\psi](\phi_0) -\mathcal{I}[p(x_{t}),\one_{[0,\tau(T,t,x)]}](\phi_0) \right\rvert \leq \sC C(f).
\]
Furthermore, since $|t_{\ast}-t| \leq \sC \log t \leq \sC \log \log T$, we can also replace $t$ with $t_{\ast}$ in \eqref{eq:approx_5}.
Using also \eqref{eq:self_similarity_of_alpha} and \Cref{prop:estimate_on_tau} to bound $|e^{h_{\topp}t} \varpi(\tau(T,t,x),p(x_t))  - T| = |\varpi(T,p(x))  - T|\leq \sC T^{1-\delta_{\ast}}$, we deduce
\begin{equation}\label{eq:approx_6}
	\left\lvert \int_{0}^{T} f \circ \horo_s(x) \diff s - T \mu(f)  \int_{B(0,\delta t_{\ast}^{-1/4})} e^{- 2 \pi \imath \omega \cdot F_{\ast}-2\pi^2 t_{\ast} \omega \cdot \Sigma \omega } \diff \omega\right\rvert  \leq \sC C(f,x) \frac{T \cdot \log \log T}{(\log T)^{\frac{d+1}{2}}}. 
\end{equation}

We can now finish the proof by computing the second integral in the left hand side above. Let $H$ be the symmetric positive definite square root of the matrix $\Sigma$, then
\begin{multline*}
	\int_{B(0,\delta t_{\ast}^{-1/4})} e^{- 2 \pi \imath \omega \cdot F_{\ast} -2\pi^2  t_{\ast} \omega \cdot \Sigma \omega }  \diff \omega
	=  t_{\ast}^{-\frac{d}{2}} \int_{B(0,\delta t_{\ast}^{1/4} )} e^{- 2 \pi \imath \omega \cdot \frac{ F_{\ast}}{\sqrt{ t_{\ast}}} -2\pi^2 (H\omega) \cdot (H\omega) }  \diff \omega \\
	=  t_{\ast}^{-\frac{d}{2}} \int_{\R^d} e^{- 2 \pi \imath \omega \cdot \frac{ F_{\ast}}{\sqrt{ t_{\ast}}} -2\pi^2(H\omega) \cdot (H\omega) }  \diff \omega +O( t_{\ast}^{-d})
	= \frac{1}{ t_{\ast}^{\frac{d}{2}}\sqrt{\det \Sigma}} \int_{\R^d} e^{- 2 \pi \imath \yb \cdot \frac{H^{-1}  F_{\ast}}{\sqrt{ t_{\ast}}} -2\pi^2\yb \cdot \yb }  \diff \yb +O( t_{\ast}^{-1}) \\
	= \frac{1}{(2\pi  t_{\ast})^{\frac{d}{2}}\sqrt{\det \Sigma}} e^{-\frac{1}{2}  \left\| \frac{H^{-1}  F_{\ast}}{\sqrt{ t_{\ast}} } \right\|^2}+O( t_{\ast}^{-d})
	= \frac{1}{(2\pi  t_{\ast})^{\frac{d}{2}}\sqrt{\det \Sigma}}  e^{-\frac{1}{2}  \left\| \frac{ F_{\ast}}{\sqrt{ t_{\ast}} } \right\|_{\Sigma}^2 }+O( t_{\ast}^{-d}).
\end{multline*} 
where, we recall, the norm $\|\cdot\|_{\Sigma}$ on $\R^d$ is defined as $\|\xb\|_{\Sigma}^2 = \xb \cdot \Sigma^{-1} \xb$.

Since $| t_{\ast} - (\log T)/h_{\topp}| \leq\sC$, we can bound $|t_{\ast}^{-\frac{d}{2}} - (h_{\topp}/(\log T))^{\frac{d}{2}}| \leq \sC (\log T)^{-\frac{d}{2}-1}$. 
From \eqref{eq:approx_6} we conclude
\begin{equation*}
	\left\lvert \int_{0}^{T} f \circ \horo_s(x) \diff s -   \frac{h_{\topp}^{\frac{d}{2}} T }{ (2\pi \log T)^{\frac{d}{2}}\sqrt{\det \Sigma}} \left(\int_{\tM} f \diff \mu \right) e^{-\frac{1}{2}  \left\| \frac{ F_{\ast}(p(x),t_{\ast} )}{\sqrt{t_{\ast} } } \right\|_{\Sigma}^2 } \right\rvert  \leq \sC C(f,x) \frac{T \cdot \log \log T}{(\log T)^{\frac{d+1}{2}}}. 
\end{equation*}
The vector $F_{\ast} = F_{\ast}(p(x), t_{\ast} )$ has components $F_{t_{\ast} ,\omega_k}(x)$, where $\omega_k$, for $k=1,\dots,d$ form the basis we fixed in \Cref{sec:twisted_hilbert_spaces}. Thus, the Central Limit Theorem for the geodesic flow implies
\[
\frac{F_{\ast}(\cdot, t_{\ast} )}{\sqrt{t_{\ast} }} \to \mathcal{N}(0, \Sigma) \qquad \text{in distribution.}
\]
The proof of \Cref{thm:main_result} is therefore complete.


\subsection{Proof of \Cref{thm:main_result_2}}\label{sec:proof_main_result_2}

Let us fix $x \in \tM$ and $\sigma >0$, and recall that we denoted by $\gamma_{x, \sigma}$ the horocycle segment $s \mapsto \horo_s(x)$ for $s \in [0,\sigma]$. Let $\eta$ be a 1-form of class $\mathscr{C}^2$ on $\tM$ with compact support, and let  $f:= f_\eta \in \mathscr{C}^{2}_c(\tM)$ be given by $f(x) = \langle \eta, U \rangle_x $. We also assume that we fixed a fundamental domain $\cF \subset \tM$ containing $x$ and we take $x=x_0$ to be the base point in $\cF$ as in \Cref{sec:xi_and_x0}.

The proof is analogous to the one presented in \Cref{sec:proof_main_result}. 
By \eqref{eq:potential}, we have
\[
\int_{\geo_{-t} \circ \gamma_{x, \sigma}} \eta = \int_0^{\sigma} \eta_{\geo_{-t} \circ \gamma_{x, \sigma}(s)} \left( D\geo_{-t} U_{\gamma_{x, \sigma}(s)} \right) \diff s = \int_0^{\sigma} f \circ \geo_{-t} \circ \horo_s(x) \cdot J_{-t}(\horo_s(x)) \diff s.
\]
Let 
$r = \frac{d}{h_{\topp}} \log t$, so that $e^{-h_{\topp} r} = t^{-d}$.
Calling $x_{-r} = \geo_{-r}(x)$, we rewrite the equation above as 
\[
\begin{split}
\int_{\geo_{-t} \circ \gamma_{x, \sigma}} \eta &= \int_0^{\sigma} \cL_{t} f \circ \horo_s(x) \diff s= \int_0^{\sigma} \cL_{t-r}( f) \circ \geo_{-r} \circ \horo_s(x) \cdot J_{-r}(\horo_s(x_{-r})) \diff s \\
&= \int_0^{\tau(\sigma,-r,x)} \cL_{t-r}(f) \circ \horo_s(x_{-r}) \diff s,
\end{split}
\]
We note that, by \Cref{lemma:Marcus}, $\tau = \tau(\sigma,-r,x)$ satisfies 
\[
C_{\tau}^{-1} \sigma \, e^{h_{\topp} r} \leq \tau \leq C_{\tau} \sigma \,  e^{h_{\topp} r}.
\]
As in the proof of \Cref{lem:smoothened_erg_int}, we define a function $\psi \colon [0, \tau] \to [0,1]$ of class $\mathscr{C}^{1+\Lip}$ with $\|\psi'\|_{\Lip} \leq 16$ and equal to 1 in $[1, \tau-1]$, so that 
\[
\left\lvert e^{-h_{\topp} t}\int_{\geo_{-t} \circ \gamma_{x, \sigma}} \eta - e^{-h_{\topp} r}\int_0^{\tau} \hat{\cL}_{t-r}(f) \circ \horo_s(x_{-r}) \cdot \psi(s) \diff s \right\rvert \leq \sC \|f\|_{\infty} e^{-h_{\topp} r} = \sC \|f\|_{\infty} t^{-d},
\]
where we used the fact that $\int_0^1 \cL_{t-r}(|1|) \circ \horo_s(y) \diff s \leq C_{\tau} e^{h_{\topp} (t-r)}$ for any $y\in \tM$.

As in \Cref{sec:proof_main_result}, let $\{\rho_j\}_{j \in \mathcal{J}}$ be a smooth partition of unity on the horocycle segment starting at $x_{-r}$ of length $\tau$ so that each $\rho_j$ is supported on a subarc $I_j$ of length 1 starting at $x_{j}$, and $|\mathcal{J}|\leq \sC\tau \leq \sC \sigma e^{h_{\topp} r}$. With the same decomposition of $f$ as in \eqref{eq:split_of_f}, we obtain
\[
\begin{split}
\int_0^{\tau} \hat{\cL}_{t-r}(f) \circ \horo_s(x_{-r}) \cdot \psi(s) \diff s &= \sum_{j \in \mathcal{J}} \int_{\T^d} \int_0^1 e^{-2\pi \imath \xi_\omega(\horo_s(x_j))} \cdot \hat{\cL}^{(\omega)}_{t-r} u_{\omega} \circ \horo_s(x_j) \cdot \psi_j(s) \cdot \rho_j(s) \diff s\diff \omega \\
&= \sum_{j \in \mathcal{J}} \int_{\T^d} e^{-2\pi \imath \xi_\omega(x_j)} \cdot \mathcal{I}[p(x_j), \Psi_j(\cdot, \omega)] \Big( \hat{\cL}^{(\omega)}_{t-r} u_{\omega}  \Big) \diff \omega,
\end{split}
\]
where $\psi_j$ is the restriction of $\psi$ to the arc $I_j$ and, as in \eqref{eq:def_of_Psi}, we defined
\[
\Psi_j(s, \omega) := \exp\left(- 2 \pi \imath \int_{x_{j}}^{\horo_s(x_{j})}p^{\ast}\omega \right)\cdot \psi_j(s)\cdot \rho_j(s).
\]
Clearly, $t_0:= t-r > t/2$; hence, following the same steps as in the previous subsection, we use the spectral decomposition from \Cref{prop:pert2} to reduce the integral above to the following: define
\[
A_j(\omega) := \mathcal{I}[p(x_{j}),\Psi_j(\cdot, \omega)]\left(\Pi_{\omega} u_\omega \right),
\]
then,
\[
\left\lvert e^{-h_{\topp} t}\int_{\geo_{-t} \circ \gamma_{x, \sigma}} \eta - e^{-h_{\topp} r} \sum_{j \in \mathcal{J}}  \int_{B(0,\delta)} e^{- 2 \pi \imath \xi_{\omega}(x_{j}) +z(\omega) t_0 } A_j(\omega) \diff \omega  \right\rvert \leq \sC C(f,x) t^{-d}.
\]
We focus on one fixed $j \in \mathcal{J}$. 
By \Cref{lem:smoothness_of_A}, we have
\begin{multline*}
	\left\lvert \int_{B(0,\delta)} e^{- 2 \pi \imath \xi_{\omega}(x_{j})+z(\omega) t_0 } A_j(\omega) \diff \omega - 
A_j(0) \int_{B(0,\delta)} e^{- 2 \pi \imath \xi_{\omega}(x_{j}) -2\pi^2 t_0 \omega \cdot \Sigma \omega }  \diff \omega \right\rvert \\ 
	\leq \sC C(f,x) t_0^{-\frac{d+1}{2}}\leq \sC C(f,x) t^{-\frac{d+1}{2}},
\end{multline*} 
where $\Sigma$ is once again the positive definite matrix associated to the quadratic form $-\frac{1}{4\pi^2}\sigma$ in \eqref{eq:sigma}.

We notice that we can restrict the integral above to $B(0,\delta t_0^{-1/4})$. For all $\omega \in B(0,\delta t_0^{-1/4})$, we have
\[
\left\lvert  \xi_{\omega}(x_{j}) \right\rvert \leq \sC \|\omega\| (r+\sigma) \leq \sC \frac{\log t}{t^{1/4}},
\]
where we used the fact that $x_j$ lies on the segment $\geo_{-r}\circ \gamma_{x,\sigma}$.
Therefore, $|e^{- 2 \pi \imath \xi_{\omega}(x_{j})} -1| \leq \sC \|\omega\| \cdot \log t$ and hence 
\begin{multline*}
	\left\lvert \int_{B(0,\delta)} e^{- 2 \pi \imath \xi_{\omega}(x_{j})+z(\omega) t_0 } A_j(\omega) \diff \omega - 
A_j(0) \int_{B(0,\delta t_0^{-1/4})} e^{ -2\pi^2 t_0 \omega \cdot \Sigma \omega }  \diff \omega \right\rvert \\ 
	\leq \sC C(f,x) t_0^{-\frac{d+1}{2}}\leq \sC C(f,x) \frac{\log t}{ t^{\frac{d+1}{2}}}.
\end{multline*} 
The integral on the right hand side above is computed as at the end of  \Cref{sec:proof_main_result}.

Combining everything together, and using the fact that $\sum_j A_j(0) = \mu(f) \cdot \mathcal{I}[x_{-r},\psi](\phi_0)$, we conclude
\[
\left\lvert e^{-h_{\topp} t}\int_{\geo_{-t} \circ \gamma_{x, \sigma}} \eta - \mu(f) \frac{e^{-h_{\topp} r} \cdot \varpi(\tau,x_{-r})}{(2\pi  t_0)^{\frac{d}{2}}\sqrt{\det \Sigma}} \right\rvert \leq \sC C(f,x) \frac{\log t}{ t^{\frac{d+1}{2}}}.
\]
Finally, by  \eqref{eq:self_similarity_of_alpha} applied to the point $\varpi(\tau,x_{-r}) = \varpi(\tau(\sigma,-r,x),\geo_{-r}(x))$, we obtain
\[
e^{-h_{\topp} r} \varpi(\tau,x_{-r}) = \varpi(\sigma,x).
\]
This finishes the proof of \Cref{thm:main_result_2}.


\appendix
\section{Norm estimates for the winding cycle}\label{sec:app-G}
This Appendix is devoted to the proof of some bounds regarding the winding function defined in \eqref{eq:F}. Firstly, it is convenient to introduce some notations.
For each $\omega \in \T^d$, we define
\begin{equation}\label{eq:AB}
\begin{split}
&A(X,\omega):= \max\{\|\Phi^-\|_{\cC^1}\|\langle X, \omega \rangle\|_{\cC^1}, \|\langle X, \omega \rangle\|_{\cC^1}^2, \|\langle X, \omega \rangle\|_{\cC^2}\}\\
&B(\bar k):= \max\{(-\overline k)^{-\frac 12}, (-\overline k)^{-1}\},
\end{split}
\end{equation}\label{eq:Cvar}
where ${\overline k}<0$ is given in Lemma \ref{lemma:Jacobi} and $\Phi^-$ in \eqref{eq:div}.

\begin{proposition}\label{prop:G}
Setting 
\[
G_{t,\omega}(x)=\exp\left( 2\pi \imath \int_0^t \langle \omega, X\rangle\circ \geo_{-a}(x)\diff a  \right),
\]
then for each $t\in \R^+$, $ I\subset \cI_\rho, \omega \in \T^d, \psi\in \cC^r_c(I)$ with $\|\psi\|_{\cC^r(I)}\le C_*$ and $\varphi\in \cC^{1+\alpha}_c(I)$, we have
\[
 \|(\varphi G_{t,\omega})\circ \geo_t\circ \horo_{(\cdot)}\circ \psi \|_{\cC^{\varsigma}(I)}  \le C_\varsigma(t) \|\varphi\|_{\cC^\varsigma(I)}, \quad \varsigma \in \{0,\alpha, 1+\alpha\},
\]
where
\begin{equation}\label{eq:C(t)}
C_\varsigma(t)=\begin{cases}  1 \qquad &\text{if}\qquad  \varsigma=0,\\ 					   
					    2\pi C_*\big(\frac{1 }{\sqrt{-\overline k}}\|\langle \omega, X \rangle \|_{\cC^1} +e^{-\sqrt{-\overline k}\alpha t}\big) \qquad &\text{if} \qquad\varsigma=\alpha,\\
					   C_*+8\pi^2\max\{C_*^3,C_*^4\} B(\overline k)A(X,\omega)\big(1+4e^{-\sqrt{-\overline k}(1+\alpha)t}\big) \qquad &\text{if} \qquad\varsigma=1+\alpha.
		\end{cases}
\end{equation}

\end{proposition}
\begin{proof}
The case $\varsigma=0$ is trivial since $ \|G_{t,\omega}\|_{\cC^0}=1$.
Let us use the symbol $^\prime$ to denote the derivative with respect to $\eta \in I$. We have
\[
\begin{split}
((\varphi G_{t,\omega})\circ & \geo_t\circ \horo_\eta\circ \psi )'\\
=&(G_{t,\omega}\circ \geo_t\circ \horo_\eta \circ \psi)'\cdot (\varphi\circ \geo_t\circ \horo_\eta \circ \psi)+ G_{t,\omega}\circ \geo_t\circ \horo_\eta \circ \psi\cdot (\varphi\circ \geo_t\circ \horo_\eta \circ \psi)'\\
=&2\pi \imath\psi'\left(  \int_0^t (D\geo_{t-a}(U) \langle \omega, X\rangle)\circ \geo_{t-a}\circ \horo_\eta\circ \psi \diff a   \right) (G_{t,\omega}\cdot\varphi)\circ \geo_t\circ \horo_\eta \circ \psi\\
&+G_{t,\omega}\circ \geo_t\circ \horo_\eta \circ \psi\cdot \psi' (D\geo_t(U)\varphi)\circ \geo_t\circ \horo_\eta \circ \psi.
\end{split}
\]
Recalling that $D\geo_tU=J_t U$ and $J_{-t}=D\geo_{-t}(U)$, we note that
\begin{equation}\label{eq:great}
(\varphi\circ \geo_t\circ \horo_\eta \circ \psi)'=\psi' (D\geo_t(U)\varphi)\circ \geo_t\circ \horo_\eta \circ \psi= \psi' J_t (U\varphi)\circ \geo_t\circ \horo_\eta \circ \psi.
\end{equation}
Therefore, since $\|\psi\|_{\cC^r}\le C_*$ and $\|G_{t,\omega}\|_{\cC^0}=1$, we have
\[
\begin{split}
|((\varphi G_{t,\omega})\circ \geo_t\circ \horo_\eta\circ \psi )'| 
&\le 2\pi C_*\|\varphi\|_{\cC^0}\|\langle \omega, X \rangle \|_{\cC^1}\int_0^t |J_{t-a} \circ \horo_\eta \circ \psi |\diff  a+C_*|J_t\circ \horo_\eta|\|\varphi\|_{\cC^1}.
\end{split}
\]
By Lemma \ref{lemma:Jacobi}, $|J_t(x)|\le e^{-\sqrt{-\overline k}t}$ for each $x\in M$, and
\begin{equation}\label{eq:intJ}
\left|\int_0^t J_{t-a}\circ h_\eta \diff a\right|\le \int_0^t e^{-\sqrt{-\overline k}(t-a)}\diff a \le \dfrac{1}{\sqrt{-\overline k}},
\end{equation}
therefore,
\begin{equation}\label{eq:C1G}
\begin{split}
\|(\varphi G_{t,\omega})\circ \geo_t\circ \horo_\eta\circ \psi \|_{\cC^1} 
&\le \frac{2\pi C_*}{\sqrt{-\overline k}}\|\langle \omega, X \rangle \|_{\cC^1} \|\varphi\|_{\cC^0}+C_*e^{-\sqrt{-\overline k}t}\|\varphi\|_{\cC^1}\\
&\le C_*\left(\frac{2\pi  }{\sqrt{-\overline k}}\|\langle \omega, X \rangle \|_{\cC^1} +e^{-\sqrt{-\overline k}t}\right)\|\varphi\|_{\cC^1}.
\end{split}
\end{equation}
For $\varsigma=\alpha$ it is enough to observe that, for each $\eta,\tilde \eta \in (0,\rho)$,
\begin{equation}\label{eq:Holder}
\begin{split}
|(\varphi \circ \geo_t\circ \horo_{(\cdot)}\circ \psi)(\tilde\eta)- (\varphi \circ \geo_t\circ \horo_{(\cdot)}\circ \psi_{\tilde\eta})(\eta)| 
&\le d_{I}(\geo_t\circ \horo_\eta\circ \psi,\geo_t\circ \horo_{\tilde \eta}\circ \psi )^{\alpha} \|\varphi\|_{\cC^\alpha}  \\
&\le C_* e^{-\sqrt{-\overline k} \alpha t}|\eta-\tilde \eta|^{\alpha}  \|\varphi\|_{\cC^\alpha},
\end{split}
\end{equation}
where $d_I(x,y)$ is the distance of two points $x,y$ along $I$. Hence, since $G_{t,\omega}\in \cC^r$, using the inequality
\begin{equation}\label{eq:Calpha}
\|\varphi G\|_{\cC^\alpha} \le \|\varphi\|_{\cC^0}\|G\|_{\cC^1}+\|\varphi\|_{\cC^\alpha}\|G\|_{\cC^0},
\end{equation}
we have
\[
\begin{split}
\|(\varphi G_{t,\omega})\circ & \geo_t\circ \horo_\eta\circ \psi \|_{\cC^\alpha(I)}\le \\
&\le \|G_{t,\omega}\circ \geo_t\circ \horo_\eta\circ \psi\|_{\cC^1}\|\varphi\|_{\cC^0}+\|G_{t,\omega}\|_{\cC^0}\|\varphi\circ \geo_t\circ \horo_\eta\circ \psi\|_{\cC^{\alpha}}\\
&\le  C_*\left(\frac{2\pi }{\sqrt{-\overline k}}\|\langle \omega, X \rangle \|_{\cC^1} +e^{-\sqrt{-\overline k}t}\right)\|\varphi\|_{\cC^0}+C_* e^{-\sqrt{-\overline k} \alpha t}  \|\varphi\|_{\cC^\alpha}\\
&\le C_*\left(\frac{2\pi }{\sqrt{-\overline k}}\|\langle \omega, X \rangle \|_{\cC^1} + e^{-\sqrt{-\overline k} \alpha t}  \right) \|\varphi\|_{\cC^\alpha}.
\end{split}
\]
It remains the case $\varsigma=1+\alpha$. Using \eqref{eq:Calpha}, we can check that if $G\in \cC^2$ and $\varphi \in \cC^{1+\alpha}$, then 
\begin{equation}\label{eq:C1+}
\|\varphi G\|_{\cC^{1+\alpha}}\le  \|\varphi\|_{\cC^{1+\alpha}}(1+4\|G\|_{\cC^{2}}).
\end{equation}
 Let us thus compute the second derivative of $(G_{t,\omega}\circ \geo_t\circ \horo_\eta \circ \psi)$: using again that $D\geo_tU=J_t U$,
\begin{equation}\label{eq:2deriv}
\begin{split}
|(G_{t,\omega}\circ & \geo_t\circ \horo_\eta \circ \psi)''|\\
=&\left|\left( 2\pi\imath (G_{t,\omega}\circ \geo_t\circ \horo_\eta\circ \psi)\psi' \int_0^t J_{t-a}\circ \horo_\eta \cdot U(\langle \omega, X \rangle\circ \geo_{t-a}\circ \horo_\eta\circ \psi )\diff a \right)'\right|\\
\le & 2\pi C_* \bigg( \left| \int_0^t (J_{t-a}\circ \horo_\eta)' \cdot U(\langle \omega, X \rangle\circ \geo_{t-a}\circ \horo_\eta\circ \psi )\diff a \right|\\
&+ \left| \int_0^t J_{t-a}\circ \horo_\eta \cdot (U(\langle \omega, X \rangle\circ \geo_{t-a}\circ \horo_\eta\circ \psi ))'\diff a \right|\bigg)\\
&+2\pi C_*  \left| \int_0^t J_{t-a}\circ \horo_\eta \cdot U(\langle \omega, X \rangle\circ \geo_{t-a}\circ \horo_\eta\circ \psi )\diff a \right|(1+\|G_{t,\omega}\circ \geo_t\circ \horo_\eta\circ \psi\|_{\cC^1}).
\end{split}
\end{equation}
We need to estimate the above terms. By \eqref{eq:Jbound}, \eqref{eq:potential}, \eqref{eq:div} and \eqref{eq:intJ},
\[
\begin{split}
\left| (J_{t-a}\circ {\horo_\eta})' \right| &\le \left| \int_0^t (\Phi^{-} \circ \geo_a\circ \horo_\eta)' \right|\left| J_{t-a}\circ \horo_\eta \right|\le  \dfrac{\|\Phi^{-} \|_{\cC^1}}{\sqrt{-\overline k}}e^{-\sqrt{-\overline k}(t-a)}.
\end{split}
\]
On the other hand
\[
\left| (U(\langle \omega, X \rangle\circ \geo_{t-a}\circ \horo_\eta\circ \psi ))'\right| \le C_*\|\langle \omega, X \rangle\|_{\cC^2}\|J_{t-a}\|_{\cC^0}\le C_*\|\langle \omega, X \rangle\|_{\cC^2}e^{-\sqrt{-\overline k}(t-a)}
\]
and
\[
\int_0^t |U(\langle \omega, X \rangle\circ \geo_{t-a}\circ \horo_\eta\circ \psi )|\diff a \le C_*\|\langle \omega, X \rangle\|_{\cC^1}\int_0^t e^{-\sqrt{-\overline k}(t-a)}\diff a.
\]
Using the above three inequalities into \eqref{eq:2deriv}, and recalling \eqref{eq:C1G}, we obtain
\[
\begin{split}
\|G_{t,\omega}\circ  \geo_t\circ \horo_\eta \circ \psi\|_{\cC^2}&\le 2\pi C_*^2 \dfrac{\|\Phi^{-} \|_{\cC^1}\|\langle \omega, X \rangle\|_{\cC^1}}{2{\sqrt{-\overline k}}} + \frac{2\pi C_*^2\|\langle \omega, X \rangle\|_{\cC^2}}{2\sqrt{-\overline k}}\\
&+\frac{2\pi C_*^2\|\langle \omega, X \rangle\|_{\cC^1}}{2\sqrt{-\overline k}}\left(1+C_*\left(\frac{2\pi  }{\sqrt{-\overline k}}\|\langle \omega, X \rangle \|_{\cC^1} +e^{-\sqrt{-\overline k}t}\right)\right).
\end{split}
\]
We have thus obtained
\[
\|G_{t,\omega}\circ  \geo_t\circ \horo_\eta \circ \psi\|_{\cC^2}\le 2\pi^2\max\left\{\frac{C_*^2}{\sqrt{-\overline k}}, \frac{ C_*^3}{-\overline k}\right\}A(X,\omega)\left(4+e^{-\sqrt{-\overline k}t}\right).
\]
Hence, by \eqref{eq:C1+}, we have
\begin{multline*}
\|(\varphi G_{t,\omega})\circ \geo_t\circ \horo_\eta\circ \psi \|_{\cC^{1+\alpha}} \\ \le (1+8\pi^2\max\{C_*^2,C_*^3\}B(\overline k)A(X,\omega)(4+e^{-\sqrt{-\overline k}t}))\|\varphi\circ \geo_t\circ \horo_\eta\circ \psi \|_{\cC^{1+\alpha}}.
\end{multline*}
Finally, by \eqref{eq:great}, and arguing as in \eqref{eq:Holder},
\[
\|(\varphi\circ \geo_t\circ \horo_\eta\circ \psi)'\|_{\cC^{\alpha}}=\| \psi' J_t (U\varphi)\circ \geo_t\circ \horo_\eta \circ \psi\|_{\cC^\alpha}\le C_*e^{-\sqrt{\overline k}t}e^{-\sqrt{-\overline k} \alpha t}\|\varphi'\|_{\cC^\alpha}.
\]
Inserting the above in the previous equation we conclude:
\begin{multline*}
\|(\varphi G_{t,\omega})\circ \geo_t\circ \horo_\eta\circ \psi \|_{\cC^{1+\alpha}} \\ \le  \left(C_*+8\pi^2\max\{C_*^3,C_*^4\} B(\overline k)A(X,\omega)\big(1+4e^{-\sqrt{-\overline k}(1+\alpha)t}\big)\right)\|\varphi\|_{\cC^{1+\alpha}}.
\end{multline*}
The proof is therefore complete.
\end{proof}


\subsection*{Acknowledgements}
This work was started when RC was working at University of Pisa and it is partially supported by the research project PRIN 2022NTKXCX ``Stochastic properties of dynamical systems'' funded by the Ministry of University and Scientific Research of Italy.\\ This research is part of RC's activity within the INdAM (Istituto Nazionale di Alta Matematica) group GNFM, and RC and DR's activity within the UMI Group ``DinAmicI''.\\ It is a pleasure to thank Claudio Bonanno, Dmitri Dolgopyat, Fran\c cois Ledrappier, Paolo Giulietti, Carlangelo Liverani, Omri Sarig, and Milo Viviani for many useful discussions. Finally, we thank the referees for the careful reading of the manuscript and for the helpful comments and suggestions, which have improved the clarity of the paper.



\begin{thebibliography}{99}

\bibitem{Aar} J. Aaronson. An Introduction to Infinite Ergodic Theory. Math. Surveys and Monographs 50,
Amer. Math. Soc. (1997).

\bibitem{AdBa} A. Adam, V. Baladi. Horocycle averages on closed manifolds and transfer operators. Tunisian J. Math 4 387--441 (2022).

\bibitem{AvDoDu} A. Avila, D. Dolgopyat, E. Duryev, and O. Sarig. The visits to zero of a random walk driven by an irrational rotation. Isr. J. Math. 207, 653--717 (2015).

\bibitem{BaCa} V. Baladi, R. Castorrini. Thermodynamic formalism for transfer operators of piecewise expanding maps in finite dimension. Discrete Contin. Dyn. Syst., 44, 2169-2192 (2024).

\bibitem{BaLe} M. Babillot, F. Ledrappier. Geodesic paths and horocycle flows on Abelian covers. Lie
groups and ergodic theory (Mumbai, 1996), 1--32, Tata Inst. Fund. Res. Stud. Math. 14,
Tata Inst. Fund. Res., Bombay, (1998).

\bibitem{BlKeLi} M. Blank, G. Keller, and C. Liverani, Ruelle-Perron-Frobenius spectrum for Anosov maps, Nonlinearity 15 (2002), 1905-1973. MR 1938476. Zbl 1021. 37015.

\bibitem{Bow} R. Bowen. Periodic orbits for hyperbolic flows, Amer. J. Math. 94 (1972), 1--30.

\bibitem{BFRT} H. Bruin, C. Fougeron, D. Ravotti, and D. Terhesiu. On asymptotic expansions of ergodic integrals for $\Z^d$-extensions of translation flows. Stud. Math. 284, No. 3, 229--280 (2025).

\bibitem{BuFo} A. Bufetov, G. Forni. Limit theorems for horocycle flows. Ann. Sci. \'{E}c. Norm. Sup\'{e}r. 47(5) (2014), 851--903.

\bibitem{Bur} M. Burger. Horocycle flow on geometrically finite surfaces. Duke Math. J. 61(3) (1990), 779--803.

\bibitem{But} O. Butterley. A note on operator semigroups associated to chaotic flows--corrigendum. Ergodic Theory Dynam. Systems 36.5, pp. 1409--1410, (2016).

\bibitem{But12} O. Butterley. Expanding semiflows on branched surfaces and one-parameter semigroups of operators. Nonlinearity 25(12) (2012), 3487--3503.

\bibitem{BuLi} O. Butterley and C. Liverani. Smooth Anosov flows: correlation spectra and stability. J. Mod. Dyn. 1(2)
(2007), 301--322.

\bibitem{BuSi} O. Butterley and L.D. Simonelli. Parabolic flows renormalized by partially hyperbolic maps. Bollettino dell'Unione Matematica Italiana 13, 341--360 (2020).

\bibitem{BuCaCa} O. Butterley, G. Canestrari, R. Castorrini, Discontinuities cause essential spectrum on surfaces, Ann. Henri Poincar\'e 26, 3075-3101 (2025).

\bibitem{BuCaJa} O. Butterley, G. Canestrari, S. Jain. Discontinuities cause essential spectrum. Commun. Math. Phys., (2022).

\bibitem{CaLi} R. Castorrini, C. Liverani. Quantitative statistical properties of two-dimensional partially hyperbolic systems. Adv. Math., 409, Part A, 1--122 (2022).

\bibitem{Dav} B. Davies.  Linear operators and their spectra. Vol. 106. Cambridge Studies in Advanced Mathematics. Cambridge University Press, Cambridge, pp. xii+451 (2007).

\bibitem{DeKiLi}  M.~F.~Demers, N.~Kiamari, C.~Liverani, Transfer operators in hyperbolic dynamics. An introduction. $33^o$  Col\'oq. Bras. Mat. Instituto Nacional de Matem\'atica Pura e Aplicada (IMPA), Rio de Janeiro, 2021.

\bibitem{doCarmo} M.P. do Carmo.  Riemannian Geometry. Mathematics: Theory and Applications. Birkh\"auser Boston, MA (1992).

\bibitem{FaGoLa} F. Faure, S. Gouezel, and E. Lanneau. Ruelle spectrum of linear pseudo-Anosov maps. J. \'{E}c. Polytech. Math. 6, 811-877 (2019).

\bibitem{FlFo} L. Flaminio, G. Forni. Invariant distributions and time averages for horocycle flows. Duke Math. J. 119(3) (2003), 465--526.

\bibitem{FlRa} L. Flaminio, D. Ravotti. Abelian covers of hyperbolic surfaces: equidistribution of spectra and infinite mixing
asymptotics for horocycle flows. Nonlinearity 37, No. 7, Article ID 075015 (2024).

\bibitem{GiLi} P. Giulietti and C. Liverani, Parabolic dynamics and anisotropic Banach spaces, JEMS
21 2793--2858 (2019).

\bibitem{GiLiPo} P. Giulietti, C. Liverani, and M. Pollicott, Anosov flows and dynamical zeta functions, Ann. of Math. (2) 178:2 (2013), 687--773.

\bibitem{Gou} S. Gou\"ezel , Limit theorems in dynamical systems using the spectral method (Proceedings of Symposia in Pure Mathematics 89:161--193, 2015)

\bibitem{GoLi} S. Gou\"ezel, C. Liverani, Banach spaces adapted to Anosov systems, Ergodic Theory Dynam. Systems 26 (2006), 189--217.

\bibitem{GoLi2} S. Gou\"ezel, C. Liverani. Compact locally maximal hyperbolic sets for smooth maps: fine statistical properties, J. Differential Geom. 79:3 (2008), 433--477. 

\bibitem{Gre} L.W. Green, Remarks on uniformly expanding horocycle reparametrization, J. Differential Geometry, { 13} (1978) 263--271.

\bibitem{Gur} B.M. Gurevic. The entropy of a horocycle flow, Soviet Math. Doklady 2 (1961), 124--130.

\bibitem{GuLe} Y. Guivarc'h, Y. Le Jan. Asymptotic winding of the geodesic flow on modular surfaces and continuous fractions. Annales scientifiques de l' \'Ecole Normale Sup\'erieure, Serie 4, Volume 26 (1993) no. 1, pp. 23--50.

\bibitem{LeSa} F. Ledrappier, O. Sarig. Unique ergodicity for non-uniquely ergodic horocycle flows. Discrete and Continuous Dynamical Systems, 16(2) (2006), 411--433.

\bibitem{LeSa2} F. Ledrappier, O. Sarig. Fluctuations of ergodic sums for horocycle flows
on $\Z^d$-covers of finite volume surfaces. Discrete Contin. Dyn. Syst., 22(1-2) (2008), 247--325.

\bibitem{Liv} C. Liverani, On contact Anosov flows, Ann. of Math. (2) 159:3 (2004), 1275--1312. 

\bibitem{HasKat} B. Hasselblatt, A. Katok. Introduction to the Modern Theory of Dynamical Systems. Encyclopedia of Mathematics and its Applications, Issue 54, Cambridge University Press (1995).

\bibitem{Hen} H. Hennion. Sur un th\'eor\`eme spectral et son application aux noyaux lipchitziens. Proc. Amer. Math. Soc. 118:2 (1993), 627--634. 

\bibitem{HuKa} S. Hurder, A. Katok. Differentiability, rigidity and Godbillon-Vey classes for Anosov flows. Publications math\'ematiques de l'IHES., tome 72 (1990), p. 5--61.

\bibitem{Kat} T. Kato. Perturbation theory for linears operators. Grundlehren der Mathematicschen Wissenschaften 132 (1996).

\bibitem{KatSun} A. Katsuda, T. Sunada. Closed orbits in homology classes. Publications Mathematiques de l'IHES, Volume 71 (1990), pp. 5--32.

\bibitem{Margulis} G.A. Margulis. Certain measures associated with U-flows on compact manifolds. Functional Anal. Appl. 4 (1) (1970), 55--67.

\bibitem{Marcus} B. Marcus. Unique ergodicity of the horocycle flow: variable negative curvature case. Israel J. Math., 21(2-3) (1975), 133--144.

\bibitem{Marcus2} B. Marcus. Ergodic properties of horocycle flows for surfaces of negative curvature. Annals of Mathematics, Second Series, Vol. 105, No. 1 (1977), pp. 81--105.

\bibitem{Pla} J. Plante. Anosov flows, Amer. J. Math. 94 (1972), 729--754.

\bibitem{Rav} D. Ravotti. Asymptotics and limit theorems for horocycle ergodic integrals \`a la Ratner (with an appendix by Emilio Corso), Journal de l'\'Ecolepolytechnique-Math\'ematiques, Tome 10 (2023), 305-334. 

\bibitem{Sar} O. Sarig. Invariant measures for the horocycle flow on Abelian covers, Inv. Math. 157 (2004),
519--551.

\bibitem{SaSc}
O. Sarig, B. Schapira. The generic points for the horocycle flow on a class of hyperbolic surfaces
with infinite genus, IMRN (2008), Article ID rnn086.

\bibitem{Schw} L. Schwartz, Th\'eorie des distributions, Publications de l'Institut de Math\'ematique de l'Universit\'e de Strasbourg, Hermann: Paris, 1966.

\bibitem{Str} A. Str\"{o}mbergsson. On the deviation of ergodic averages for horocycle flows, J. Mod. Dyn.
7(2) (2013), 291--328.

\bibitem{Tsu} M. Tsujii, Quasi-compactness of transfer operators for contact Anosov flows, Nonlinearity 23:7 (2010), 1495--1545.

\end{thebibliography}
\end{document}